\documentclass[10pt,a4paper]{article}
\usepackage{ifpdf}
\pdfoutput=1
\usepackage[a4paper,top=3cm,bottom=3cm,left=3cm,right=3cm,%
bindingoffset=5mm]{geometry}

\usepackage[utf8]{inputenc}
\usepackage[english]{babel}
\usepackage{amsmath}
\usepackage{amsfonts}
\usepackage[output-decimal-marker={,},exponent-product=\cdot]{siunitx}
\usepackage{icomma}
\usepackage{amsmath}
\usepackage{amssymb}
\usepackage{graphicx}
\usepackage{amsthm}
\usepackage{epstopdf}
\usepackage{tabularx}
\usepackage{multirow}
\usepackage{booktabs}
\usepackage{todonotes}
\usepackage{bm}
\usepackage{caption}
\usepackage{subcaption}
\usepackage[percent]{overpic}
\newcommand{\numberset}{\mathbb}
\usepackage{authblk}

\addcontentsline{toc}{chapter}{Chapter1}
\newtheorem{theorem}{Theorem}[section]
\newtheorem{corollary}{Corollary}[section]
\newtheorem{lemma}[theorem]{Lemma}
\newtheorem{proposition}{Proposition}[section]

\newcommand{\cfun}[1]{\textnormal{\texttt{\detokenize{#1}}}}
\newcommand{\cfan}[1]{\textnormal{\textbf{\detokenize{#1}}}}
\newcommand{\N}{\numberset{N}}
\newcommand{\R}{\numberset{R}}

\newcommand{\Pk}{\numberset{P}}

\renewcommand{\epsilon}{\varepsilon}
\renewcommand{\theta}{\vartheta}
\renewcommand{\rho}{\varrho}
\renewcommand{\phi}{\varphi}

\newcommand{\xx}{\boldsymbol{x}}

\newcommand{\uu}{\boldsymbol{u}}
\newcommand{\uht}{\boldsymbol{U}}

\newcommand{\pht}{P}
\newcommand{\qht}{Q}
\newcommand{\vv}{\boldsymbol{v}}
\newcommand{\vht}{\boldsymbol{V}}

\newcommand{\zz}{\boldsymbol{z}}

\newcommand{\ww}{\boldsymbol{w}}
\newcommand{\wht}{\boldsymbol{W}}

\newcommand{\ff}{\boldsymbol{f}}
\newcommand{\nn}{\boldsymbol{n}}
\newcommand{\tf}{\boldsymbol{t}}

\newcommand{\dd}{{\rm div}}
\newcommand{\gr}{\nabla}
\newcommand{\Gr}{\boldsymbol{\nabla}}

\newcommand{\cc}{{{\rm curl}}}
\newcommand{\CC}{\mathbf{{curl}}}

\newcommand{\divv}{\mathbf{{div}}}
\newcommand{\epseps}{\boldsymbol{\epsilon}}

\def\LstabG{\mathcal{L}}

\def\CstabG{\mathcal{C}}

\def\JstabG{\mathcal{J}}

\def\FstabG{\mathcal{F}}

\def\AstabGn{\mathcal{A}_{{\rm stab},n}}
\def\AstabGN{\mathcal{A}_{{\rm stab},N}}

\def\rgn{{\mathcal{R}}_{{\rm stab},n}}
\def\rgN{{\mathcal{R}}_{{\rm stab},N}}

\def\KstabG{\mathcal{K}_{\rm stab}}

\def\nstabj{{\rm stab},j}
\def\nstabn{{\rm stab},n}

\def\reg{s}

\newcommand{\uui}{\mathcal{P}\uu}

\newcommand{\eei}{\boldsymbol{\eta}}
\newcommand{\eea}{\boldsymbol{\Theta}}
\newcommand{\ssi}{\phi_{\eei}}

\newcommand{\I}{I}
\newcommand{\VCG}{\boldsymbol{\mathcal{V}}}
\newcommand{\ZCG}{\boldsymbol{\mathcal{Z}}}
\newcommand{\QCG}{\mathcal{Q}}

\newcommand{\VDG}{\boldsymbol{\mathcal{V}}_h}

\newcommand{\VDGt}{\boldsymbol{\mathcal{V}}_h^\tau}

\newcommand{\ZDG}{\boldsymbol{\mathcal{Z}}_h}

\newcommand{\ZDGt}{\boldsymbol{\mathcal{Z}}_h^\tau}

\newcommand{\QDG}{\mathcal{Q}_h}

\newcommand{\QDGt}{\mathcal{Q}_h^\tau}

\def\errA{\zeta_{\mathcal{A}}}
\def\errC{\zeta_{\mathcal{C}}}
\def\errc{\zeta_{c}}
\def\errJ{\zeta_{\mathcal{J}}}
\def\err0{\zeta_0}

\newcommand{\PCG}{\Phi}
\newcommand{\PDG}{\Phi_h}

\newcommand{\jump}[1]{\lbrack\!\lbrack\,#1\,\rbrack\!\rbrack}

\usepackage{color}

\newtheorem{rmk}{Remark}

\title{Pressure robust SUPG-stabilized finite elements for the unsteady Navier-Stokes equation}

\author[1]{L. Beir\~ao da Veiga \thanks{lourenco.beirao@unimib.it}}

\author[1]{F. Dassi \thanks{franco.dassi@unimib.it}}

\author[2]{G. Vacca \thanks{giuseppe.vacca@uniba.it}}

\affil[1]{Dipartimento di Matematica e Applicazioni,
Universit\`a degli Studi di Milano Bicocca,
Via Roberto Cozzi 55 - 20125 Milano, Italy}

\affil[2]{Dipartimento di Matematica, 
Universit\`a degli Studi di Bari, 
Via Edoardo Orabona 4  - 70125 Bari, Italy}

\begin{document}
\maketitle

\abstract{

In the present contribution we propose a novel conforming Finite Element scheme for the time-dependent Navier-Stokes equation, which is proven to be both convection quasi-robust and pressure robust.
The method is built combining a ``divergence-free'' velocity/pressure couple (such as the Scott-Vogelius element), a Discontinuous Galerkin in time approximation, and a suitable SUPG-curl stabilization.
A set of numerical tests, in accordance with the theoretical results, is included.
}



\section{Introduction}

The development and analysis of finite element (FE) schemes for the incompressible Navier-Stokes (NS) equation, both stationary and time-dependent, has been a major objective of research for a long time in the mathematical community. In addition to the more specific articles below, we mention here as an example the monographs \cite{girault-raviart:book,quartapelle:book,volker:book} and the recent article on modern perspectives \cite{VKN18}. Among many, we here recall two important numerical aspects in the FE literature of the NS equation, the first (and oldest) one being the difficulty related to convection dominated flows, which may lead to poor convergence of the numerical scheme and unphysical oscillations in the discrete solution. In order to deal with situations where convection dominates over diffusion, many stabilization approaches has been considered. Limiting the bibliography to a few non exhaustive references, we here cite the famous Streamline Upwind Petrov-Galerkin (SUPG), or its variants such as Least-Squares or Douglas-Wang, \cite{FF:1982,BH:1982,TB:1996}, Continuous Interior Penalty (CIP) \cite{BFH:2006,BF:2007}, Grad-Div stabilization \cite{OLHL:2009,grad-div}, Local Projection (LP) \cite{BB:2006,MT:2015,LPS-NS}, Variational MultiScale \cite{JK:2005,HJ:2000} and, resorting to non-conforming schemes, also up-winding techniques for Discontinuous Galerkin \cite{BMS-conv-04} possibly based on $H_{div}$ conforming spaces \cite{Hdiv1,Hdiv2,Hdiv3}. At the theoretical level, we here label a scheme as \emph{convection quasi-robust} if, assuming sufficient regularity of the exact solution, the velocity error is independent of the diffusion coefficient $\nu$ and the order of convergence in convection dominated cases is optimal (that is $O(h^{k+1/2})$ in a suitable velocity norm including an $L^2$ term and a convection term, where $k$ represents the ``polynomial order'' and $h$ the mesh size).

Another important, and more recent, aspect is that of pressure robustness \cite{benchmark3}. Briefly speaking, a pressure robust scheme is one for which modifications of the continuous problem that only affects the pressure leads to changes in the discrete solution that only affect the discrete pressure. As it has been deeply investigated in a series of papers, pressure robustness is an important property that leads to a series of advantages, see for instance \cite{linke-merdon:2016,john-linke-merdon-neilan-rebholz:2017}. One way to acquire pressure robustness is to use a FE scheme that guarantees a divergence free discrete kernel, provided the involved forms are approximated exactly, or possibly using ad-hoc projections \cite{benchmark3}. We here mention also the discussion on the Virtual Element schemes in \cite{BLV:2018,BMV:2018}. 

When looking in deep at the two important aspects above, it becomes clear that it is not easy to develop a numerical scheme which satisfies both convection quasi-robustness and pressure robustness. Although the latter helps also in handling convection \cite{SchroederLube}, it is not fully sufficient. Among the stabilizing schemes above, many, such as SUPG, introduce a coupling between pressure and velocity which destroys the pressure robustness property. Others, such as CIP, are in principle free of such hindrance but a convergence proof reflecting pressure robustness would be a major challenge and does not exist in the literature. To the best of these authors knowledge, the only FE method which satisfies both properties for the Navier-Stokes equation is the up-winding stabilized $H_{div}$ conforming method studied in \cite{Hdiv1,Hdiv2,Hdiv3}. Indeed such scheme is one of the best performing also at the practical level, see for instance the interesting comparison in \cite{Hdiv2}. Regarding the Oseen equation, an interesting variant of the SUPG approach was proposed in \cite{Linke:2020}. By stabilizing the curl of the momentum equation and adding suitable jump terms the authors are able to propose a convection stabilization that does not ruin pressure robustness. 

The objective of the present contribution is to propose for the first time in the literature a conforming (discrete velocities are in $H^1$) FE scheme for the time-dependent NS equation, which is both convection quasi-robust and pressure robust. The starting point is to apply the idea of \cite{Linke:2020} to the time-dependent NS equation, an important difficulty being that the SUPG approach is not suitable to standard ``difference based'' time discretization approaches. We need therefore to resort, in the spirit of \cite{HS:1990,tohomee:book}, to a discontinuous Galerkin approximation in time. The ensuing scheme therefore starts from a ``divergence-free'' velocity/pressure couple (such as the Scott-Vogelius element, see for instance \cite{Falk-Neilan-13,john-linke-merdon-neilan-rebholz:2017}), applies a DG in time approximation and adds a suitable SUPG-curl stabilization. We are able to prove that the proposed method is indeed pressure-robust and convection quasi-robust, the proof being quite involved since, in addition to combining the difficulties of the two approaches above, it leads to additional challenges peculiar to the current situation. Due to the complex technical nature of the scheme, we restrict our analysis to the case with first order polynomials in time and require a mesh-quasi uniformity assumption. A discussion related to avoiding such assumption, and the associated effects on the final convergence result, is included in the article. 

Summing up, the proposed approach yields, in a suitable discrete norm, velocity error estimates that reflect pressure robustness (for instance, are free of influence from the pressure), are $\nu$-independent and of order $O(\tau^2 + h^{k+1/2})$ in convection dominated regimes, where $\tau$ and $h$ are the discretization parameters, respectively in time and space. We furthermore present a set of numerical tests which are in accordance with the theoretical predictions and compare the proposed method with a non-stabilized version.  It is finally worth noting that, in the classification of the recent review \cite{garcia} (which classifies robust FE schemes for time-dependent Navier-Stokes in terms of its convergence order in convection dominated regimes, that is $O(h^{k-1})$, $O(h^k)$ or $O(h^{k+1/2})$) our method would fall into the $O(h^{k+1/2})$ class.

The paper is organized as follows. We introduce the continuous problem in Section \ref{sec:cont} and some important preliminaries in Section \ref{sec:notations}. Afterwards, in Section \ref{sec:stab} we introduce the discrete problem and some related consistency and stability results. In Section \ref{sec:theo} we carry out the convergence analysis of the proposed method and draw some comments. Finally, in Section \ref{sec:num} we show some numerical results.

\section{Continuous problem}\label{sec:cont}

We start this section with some standard notations. 
Let $\Omega \subset \R^2$ be the computational domain, 
we denote with $\xx = (x_1, \, x_2)$ the independent variable. 
The symbols $\gr$ and $\Delta$   
denote the gradient and the Laplacian for scalar functions, while 
$\Gr$, $\epseps$ and  $\dd$
denote 
the gradient, the symmetric gradient operator and the divergence operator,
whereas $\divv$ denotes the vector valued divergence operator for tensor fields. 
Furthermore for a scalar function $\phi$ and a vector field $\vv := (v_1, v_2)$ we set
$
\CC \phi := \left(\frac{\partial \phi}{\partial x_2}, - \frac{\partial \phi}{\partial x_1}\right)^T$ and 
$
\cc \vv := \frac{\partial v_2}{\partial x_1} - \frac{\partial v_1}{\partial x_2} \,.  
$
%
Throughout the paper, we will follow the usual notation for Sobolev spaces
and norms \cite{Adams:1975}.
Hence, for an open bounded domain $\omega$,
the norms in the spaces $W^r_p(\omega)$ and $L^p(\omega)$ are denoted by
$\|{\cdot}\|_{W^r_p(\omega)}$ and $\|{\cdot}\|_{L^p(\omega)}$ respectively.
Norm and seminorm in $H^{r}(\omega)$ are denoted respectively by
$\|{\cdot}\|_{r,\omega}$ and $|{\cdot}|_{r,\omega}$,
while $(\cdot,\cdot)_{\omega}$ and $\|\cdot\|_{\omega}$ denote the $L^2$-inner product and the $L^2$-norm (the subscript $\omega$ may be omitted when $\omega$ is the whole computational
domain $\Omega$).
We recall the following well known functional space which will be useful in the sequel
\[
    H(\dd, \Omega) := \{\vv \in [L^2(\Omega)]^2 \,\,\, 
    \text{s.t.} \,\,\,   
    \dd \, \vv \in L^2(\Omega)\} \,.
\]

For a Banach space $V$ we denote with $V'$ the dual space of $V$.
Let $(T_0, T_F) \subset \R$ denote the time interval. For space-time functions $v(\xx, t)$ defined on $\omega \times (T_0, T_F)$, we denote with $v_t$ the derivative with respect to the time variable. Furthermore, using standard notations \cite{quarteroni-valli:book}, for a Banach space $V$ with norm $\|\cdot\|_V$, we introduce the spaces $W^s_q(T_0, T_F, V)$ and $H^s(T_0, T_F, V)$  endowed with norms $\|{\cdot}\|_{W^s_q(T_0, T_F, V)}$ and $\|{\cdot}\|_{H^s(T_0, T_F, V)}$ respectively.
In similar way we consider the space $C^0((T_0, T_F), V)$.

\smallskip
Let now $\Omega \subseteq \R^2$ be a polygonal simply connected domain, let $T > 0$ be the final time and set $\I:= (0, T)$. 
We consider the unsteady Navier-Stokes equation 
(see for instance \cite{quarteroni-valli:book}) 

\begin{equation}
\label{eq:ns primale}
\left\{
\begin{aligned}
& \text{ find $(\uu,p)$ such that}\\
&\begin{aligned}
\uu_t -  \nu \, \divv (\epseps (\uu) ) + (\Gr \uu ) \,\uu    -  \nabla p &= \ff\qquad  & &\text{in $\Omega \times \I$,} \\
\dd \, \uu &= 0 \qquad & &\text{in $\Omega \times \I$,} \\
\uu &= 0  \qquad & &\text{on $\partial \Omega \times \I$,}
\\
\uu(\cdot, 0) &= \uu_0 
\qquad & &\text{in $\Omega$,}
\end{aligned}
\end{aligned}
\right.
\end{equation}
where $\uu$, $p$ are the velocity and the pressure fields respectively, 
$\nu \in \R^+$ represents the diffusive coefficient,
$\ff \in [L^2(\Omega \times I)]^2$ is the volume source term, 
and $\uu_0 \in H_0(\dd, \Omega)$ with $\dd \, \uu_0 = 0$ is the initial data.
 Note that here above we made some small modification to the equations in order to keep a smoother notation; indeed the $\nu$ above is double the standard one and we changed the sign of the pressure variable.
For the sake of simplicity we here  consider Dirichlet homogeneous boundary conditions, different boundary conditions can be treated as well. 
Let us define the continuous spaces
\[
\VCG := \left[ H_0^1(\Omega) \right]^2, \qquad 
\QCG:= L^2_0(\Omega) = \left\{ q \in L^2(\Omega) \quad \text{s.t.} \quad \int_{\Omega} q \,{\rm d}\Omega = 0 \right\} 
\]
endowed with natural norms, and the continuous forms
\begin{align}
\label{eq:forma a}
a(\cdot, \cdot) &\colon \VCG \times \VCG \to \R,    
&\,\,\,
a(\uu,  \vv) &:=  \int_{\Omega} \epseps(\uu) : \epseps (\vv) \, {\rm d}\Omega \,,
\\
\label{eq:forma b}
b(\cdot, \cdot) &\colon \VCG \times \QCG \to \R, 
&\,\,\,
b(\vv, q) &:=  \int_{\Omega}q\, \dd \vv \,{\rm d}\Omega \,,
\\
\label{eq:forma c}
c(\cdot; \, \cdot, \cdot) &\colon \VCG \times \VCG \times \VCG \to \R, 
&\,\,\,
c(\ww; \, \uu, \vv) &:=  \int_{\Omega} [( \Gr \uu ) \, \ww] \cdot \vv  \,{\rm d}\Omega \,,
\end{align}
for all $\uu, \vv, \ww \in \VCG$ and $q \in \QCG$.
Then the variational formulation of Problem~\eqref{eq:ns primale} reads as follows:

\begin{equation}
\label{eq:ns variazionale}
\left\{
\begin{aligned}
& \text{find $(\uu, p) \in L^2(0, T, \VCG) \times L^2(0, T, \QCG)$, with  $\uu_t \in L^2(0, T, \VCG')$, such that} \\
&\begin{aligned}
\int_0^T \bigl( (\uu_t, \vv) + \nu a(\uu, \vv) + c(\uu; \, \uu, \vv)  + b(\vv, p)\bigr) \,{\rm d}t &= \int_0^T (\ff, \vv) \,{\rm d}t \,\,\, & \text{$\forall \vv \in L^2(0,T, \VCG)$,} \\
\int_0^T  b(\uu, q) \,{\rm d}t &= 0 \,\,\, & \text{$\forall q \in L^2(0,T, \QCG)$,}
\\
\uu(0) &= \uu_0 
\,\,\, &\text{in $\Omega$.}
\end{aligned}
\end{aligned}
\right.
\end{equation}

In the context of the analysis of incompressible flows, it is useful to introduce the concept of the Helmholtz--Hodge projector (see for instance \cite[Lemma~2.6]{john-linke-merdon-neilan-rebholz:2017} and \cite[Theorem~3.3]{GLS:2019}).
For every $\ww \in [L^2(\Omega)]^2$ there exist $\ww_0 \in H({\rm div}; \, \Omega)$ and $\zeta \in H^1(\Omega) / \R$ such that
\begin{equation}
\label{eq:helmholtz-hodge}
\ww = \ww_0 + \nabla \, \zeta,
\end{equation}
where $\ww_0$ is $L^2$-orthogonal to the gradients, that is $(\ww_0 , \, \nabla \phi) = 0$ for all $\phi \in H^1(\Omega)$ (which implies, in particular, that $\ww_0$ is solenoidal, i.e. ${\rm div} \, \ww_0 = 0$).
The orthogonal decomposition \eqref{eq:helmholtz-hodge} is unique and is called Helmholtz--Hodge decomposition, and $\mathcal{H}(\ww) :=\ww_0$ is the Helmholtz--Hodge projector of $\ww$.
Let us introduce the kernel of the  bilinear form $b(\cdot,\cdot)$
that corresponds to the functions in $\VCG$ with vanishing divergence
\begin{equation}
\label{eq:Z}
\ZCG := \{ \vv \in \VCG \quad \text{s.t.} \quad \dd \, \vv = 0  \}\,.
\end{equation}
Then, Problem~\eqref{eq:ns variazionale} can be formulated in the equivalent kernel form:
\begin{equation}
\label{eq:ns variazionale ker}
\left\{
\begin{aligned}
& \text{find $\uu \in L^2(0, T, \ZCG)$, with  $\uu_t \in L^2(0, T, \ZCG')$, such that} \\
&\begin{aligned}
\int_0^T \bigl( (\uu_t, \vv) + \nu a(\uu, \vv) + c(\uu; \, \uu, \vv) \bigr) \,{\rm d}t &= \int_0^T (\mathcal{H}(\ff), \vv) \,{\rm d}t \,\,\, & \text{$\forall \vv \in L^2(0,T, \ZCG)$,} 
\\
\uu(0) &= \uu_0 
\,\,\, &\text{in $\Omega$.}
\end{aligned}
\end{aligned}
\right.
\end{equation}
By a direct computation, it is easy to check that the convective form  $c(\cdot; \, \cdot,  \cdot)$ satisfies 
\begin{equation}
\label{eq:c-skew}
c(\ww; \, \uu, \vv) = -c(\ww; \, \vv, \uu) \qquad \text{for all $\uu, \vv \in \VCG$, for all $\ww \in \ZCG$.}
\end{equation}

It is well known that,
discretizing problem \eqref{eq:ns variazionale} by standard inf-sup stable finite elements and  classical anti-symmetric formulations for the $c(\cdot;\cdot,\cdot)$ form, 
leads to instabilities when the convective term is dominant with respect to the diffusive term $\nu$  (see for instance \cite{BH:1982,FF:1982,HS:1990,TB:1996,BFH:2006}).
In such situations a stabilized form of the problem is required in order to prevent spurious oscillations that can completely spoil the numerical solution. 
In the following sections we propose a stabilized finite elements discretization for the Navier-Stokes equation.
Borrowing the ideas from \cite{HS:1990} and \cite{Linke:2020}, in the present contribution we consider a stabilizing
term that penalizes the equation for the vorticity, where pressure gradients are automatically eliminated.

\section{Notations and preliminaries}
\label{sec:notations}

We now introduce some basic tools and notations useful in the construction and the theoretical analysis of the proposed stabilized method.

\subsection{Space and time decompositions}
\label{sub:mesh}

Let $\{\Omega_h\}_h$ be a family of  conforming decompositions of $\Omega$ into triangular elements $E$ of diameter $h_E$. We denote by 
$h := \sup_{E \in \Omega_h} h_{E}$. 
We make the following mesh assumptions.

\begin{description}

\item[\textbf{(M1)}] The mesh family $\{ \Omega_h \}_h$ is shape regular: each element $E$ is star shaped with respect to a ball of radius $\rho_E$ and it exists a real constant $c_{\rm{M}1}$ such that $h_E \leq c_{\rm{M}1} \rho_E$ for all $E \in \{ \Omega_h \}_h$.
\item[\textbf{(M2)}] The mesh family $\{ \Omega_h \}_h$ is quasi-uniform: it exists a real constant $c_{\rm{M}2}$ such that $h \leq c_{\rm{M}2} h_E$ for all $E \in \{ \Omega_h \}_h$.
\end{description}
Condition \textbf{(M1)} is classical in FEM. Condition \textbf{(M2)}, which is needed for the theoretical analysis of the method, can be avoided at the price of an additional term in the convergence estimates, see Remark \ref{rem:no-quasi}.

For any $E \in \Omega_h$, $\nn_E$ and $\tf_E$ denote the outward normal vector and the tangent vector to $\partial E$ respectively.

Concerning the temporal discretization, we introduce a sequence of time steps $t_n = n \tau$, $n = 0, \dots, N+1$ with time step size $\tau$ and where $t_0 = 0$ and $t_{N+1} = T$. We set $I_n := (t_n, t_{n+1})$ and $\I^\tau := \{I_n\}_{n=0}^N$. 

\noindent
For $m \in \N$ we introduce the polynomial spaces 
\begin{itemize}
\item $\Pk_m(\omega)$: the set of polynomials on $\omega$ of degree $\leq m$, with $\omega$ a generic measurable set;
\item $\Pk_m(\Omega_h) := \{q \in L^2(\Omega) \quad \text{s.t} \quad q|_{E} \in  \Pk_m(E) \quad \text{for all $E \in \Omega_h$}\}$;
\item $\Pk_m(\I^\tau) := \{q \in L^2(I) \quad \text{s.t} \quad q|_{I_n} \in  \Pk_m(I_n) \quad \text{for all $I_n \in \I^\tau$}\}$.
\end{itemize}
For $s \in \R^+$ and  $p \in [1,+\infty]$ let us define the following broken Sobolev spaces:
\begin{itemize}
\item $W^s_p(\Omega_h) := \{v \in L^2(\Omega) \quad \text{s.t} \quad v|_{E} \in  W^s_p(E) \quad \text{for all $E \in \Omega_h$}\}$,
\end{itemize}
equipped with the broken norm and seminorm
\[
\begin{aligned}
\|v\|^p_{W^s_p(\Omega_h)} &:= \sum_{E \in \Omega_h} \|v\|^p_{W^s_p(E)}\,,
&\quad
|v|^p_{W^s_p(\Omega_h)} &:= \sum_{E \in \Omega_h} |v|^p_{W^s_p(E)}\,, 
&\qquad \text{if $1 \leq p < \infty$,}
\\
\|v\|_{W^s_\infty(\Omega_h)} &:= \max_{E \in \Omega_h} \|v\|_{W^s_\infty(E)}\,,
&\quad
|v|_{W^s_\infty(\Omega_h)} &:= \max_{E \in \Omega_h} |v|_{W^s_\infty(E)}\,, 
&\qquad \text{if $p= \infty$.}
\end{aligned}
\]
Analogously we define the spaces
\begin{itemize}
\item $W^s_p(\I^\tau) := \{v \in L^2(\I^\tau) \quad \text{s.t} \quad v|_{I_n} \in  W^s_p(I_n) \quad \text{for all $I_n \in \I^\tau$}\}$,
\end{itemize}
with the associated norm and seminorm.

\noindent
For any mesh edge $e$ let $\nn^e$ be a fixed unit normal vector to the edge $e$.
The jump operator on $e$ is defined for every piecewise continuous function w.r.t. $\Omega_h$ by
\[
\jump{w}_e(\xx) :=
\left \{
\begin{aligned}
& \lim_{s \to 0^+} \left( w(\xx - s \nn^e) 
- w(\xx + s \nn^e) \right) 
&\quad 
&\text{if $e \not \in \partial \Omega$} 
\\
&0
&\quad 
&\text{if $e \in \partial \Omega$.} 
\end{aligned}
\right. 
\]
Note that the boundary values are purposefully set to zero, which is different from the classical definition.
We introduce the following notation for the piecewise continuous functions w.r.t. the decomposition $\I^\tau$
\begin{itemize}
\item $C^0(\I^\tau) := \{q \in L^2(I) \quad \text{s.t} \quad q|_{I_n} \in  C^0(I_n) \quad \text{for all $I_n \in \I^\tau$}\}$.
\end{itemize}

%
For every function $w \in C^0(\I^\tau)$ we define
\[
w_+^n := \lim_{s \to 0^+} w(t_n+s)\,, 
\quad
w_-^{n+1} := \lim_{s \to 0^+} w(t_{n+1}-s)\,,
\qquad
\text{for $n=0, \dots, N$,} 
\]
and the (time) jumps
\[
\jump{w}_{t_n} := 
 w_+^n - w_-^n
\quad
\text{for $n=0, \dots, N$,}
\]
where $w_-^0$ will be explicitly defined at each occurrence.
If $w \in C^0(\I)$ (resp. $C^0(\overline{\I})$) we set $w^n := w(t_n)$ for $n=1, \dots, N$ (resp. $n=0, \dots, N+1$).

Furthermore we define the following \textbf{time interpolation operator}  $\mathcal{I}^\tau \colon C^0(\I^\tau) \to \Pk_1(I^\tau)$ that  associates to each piecewise continuous function $w$ the piecewise polynomial of degree one interpolating $w$, i.e. 
\[
\mathcal{I}^\tau w(t) = w^n_+ + \frac{t- t_n}{\tau}(w^{n+1}_- - w^n_+) 
\quad
\text{for $t \in I_n$ and  $n=0,\dots, N$.}
\]
%
Note that, if $w \in C^0(I)$ then also $\mathcal{I}^\tau w \in C^0(I)$.

\begin{rmk}[Generalized notation]
The definitions above of time jump and time interpolation can be naturally extended to piecewise continuous functions with values in a generic Banach space $V$, which we denote by $C^0(\I^\tau,V)$. In the following we will still denote such operators with the same symbols. Analogously, we will denote by $\Pk_\ell(\I^\tau,V)$, $\ell \in {\mathbb N}$, functions $\I \rightarrow V$ that are piecewise polynomial in time.
\end{rmk}

Finally, given $k \in \N$, we define the \textbf{space} $\boldsymbol{L^2}$\textbf{-projection operator} $\Pi_{k} \colon L^2(\Omega_h) \to \Pk_k(\Omega_h)$:
\begin{equation}
\label{eq:P0_k^E}
\int_{\Omega} q_k (w - \, {\Pi}_{k}  w) \, {\rm d} \Omega = 0 \qquad  \text{for all $w \in L^2(\Omega_h)$  and $q_k \in \Pk_k(\Omega_h)$.} 
\end{equation}

\subsection{Preliminary theoretical results} \label{sub:pre-theo}
%
In the present section we introduce some useful tools that we will adopt in the analysis.
Let us introduce the space of stream functions
\[
\PCG := H^2_0(\Omega) = 
\{\phi \in H^2(\Omega) \quad \text{s.t.} \quad \phi = 0 \,, \quad  \nabla \phi \cdot \nn = 0 \quad \text{on $\partial \Omega$} \}
\]
then the following sequence (in a simply connected domain $\Omega$)
is exact (see for instance \cite{guzman-neilan:2014}):
\begin{equation}
\label{eq:exact cont}
0 \, \xrightarrow[]{ \,\,\,\,\, \text{{$i$}} \,\,\,\,\,}  \,
\PCG \, \xrightarrow[]{   \quad \text{{$\CC$}} \quad  }\,
\VCG \, \xrightarrow[]{  \, \, \, \, \text{{$\dd$}} \, \, \, \, }\,
\QCG \, \xrightarrow[]{ \,\,\,\,\, 0 \,\,\,\,\,} \,
0 
\end{equation}
where $i$  denotes the mapping that to every real number $r$
associates the constant function identically equal to $r$
and we recall that a sequence is exact if the range of
each operator coincides with the kernel of the following one. 
In particular we have the following characterization of the kernel (cf. \eqref{eq:Z})
\begin{equation}
\label{eq:psi-z}
\CC(\PCG) = \ZCG \,.
\end{equation}
In the following the symbol $\lesssim$ will denote a bound up to a generic positive constant,
independent of the mesh size $h$, of the time step $\tau$ but which may depend on 
$\Omega$, on the ``polynomial'' order of the method $k$ and on the mesh shape regularity constants $c_{\rm{M}1}$ and $c_{\rm{M}2}$.

The following result is trivial to check.
\begin{lemma}[Stream functions]
\label{lm:stream}
Let $\ww \in \ZCG \cap [H^{s}(\Omega_h)]^2$ with $s \geq 1$.
Then there exists $\psi \in \PCG \cap H^{s+1}(\Omega_h)$ such that
\[
\CC \psi = \ww \,,
\quad \text{and} \quad
|\psi|_{s+1,\Omega_h} \lesssim |\ww|_{s,\Omega_h}  \,.
\]
\end{lemma}
\noindent
We now mention a list of classical results (see for instance \cite{brenner-scott:book}) useful in the sequel.


\begin{lemma}[Trace inequality]
\label{lm:trace}
For any $E \in \Omega_h$ and for any  function $v \in H^1(E)$ it holds 
\[
\|v\|^2_{\partial E} \lesssim h_E^{-1}\|v\|^2_{E} + h_E\|\nabla v\|^2_{E} \,.
\]
\end{lemma}

\begin{lemma}[Bramble-Hilbert]
\label{lm:bramble}
For any $E \in \Omega_h$ and for any  smooth enough function $\phi$ defined on $\Omega$, it holds 
\[
|\phi - \Pi_k \phi|_{W^m_p(E)} \lesssim h_E^{s-m} |\phi|_{W^s_p(E)} 
\qquad  \text{$s,m \in \N$, $m \leq s \leq k+1$, $p \in [1, \infty]$.}
\]
\end{lemma}

\begin{lemma}[Inverse estimate]
\label{lm:inverse}
Let $k \in {\mathbb N}$. For any $E \in \Omega_h$, let $\gamma_{E}=\gamma_{E}(k,m,p)$, for $m \in \{0,1,\dots, k\}$ and $p \in [1,\infty]$, denote the smallest positive constant such that for any $p_k \in \Pk_k(E)$ 
\[
\|p_k\|_{W^m_p(E)} \leq \gamma_{E} h_E^{-m} \|p_k\|_{L^p(E)} \, .
\]
Then there exists $\gamma=\gamma(k) \in {\mathbb R}^+$ such that $\gamma_{E} \le \gamma$ for all $E \in \Omega_h$.
\end{lemma}

\begin{rmk}
\label{rm:jumps}
We notice that for any $K \in [L^{\infty}(\Omega_h)]^{2 \times 2}$ and for any $\ww \in [H^{1}(\Omega_h)]^2$, Lemma \ref{lm:trace} yields the following estimate
\[
\begin{aligned}
\sum_{E \in \Omega_h} \|\jump{K \ww}\|^2_{\partial E}
&\lesssim  
\sum_{E \in \Omega_h} \|K\|^2_{L^\infty(E)} \|\ww\|^2_{\partial E}
\lesssim  
\sum_{E \in \Omega_h} \|K\|^2_{L^\infty(E)} \left( h_E^{-1} \|\ww\|^2_{E} + h_E \|\Gr \ww\|^2_{E} \right) \,.
\end{aligned}
\]
In particular if $\ww \in [\Pk_{k}(\Omega_h)]^2$ by Lemma \ref{lm:inverse} it holds that
\begin{equation}
\label{eq:utile-jump}
\sum_{E \in \Omega_h} \|\jump{K \ww}\|^2_{\partial E}
\lesssim  
\sum_{E \in \Omega_h}  h_E^{-1} \|K\|^2_{L^\infty(E)}   \|\ww\|^2_{E} \,.
\end{equation}
\end{rmk}
We omit the simple proof of the following generalized interpolation result.
\begin{lemma}[Time-interpolation]
\label{lm:time-interp}
Let $V$ be a Banach space. For any $I_n \in \I^\tau$ and for any smooth enough function $\phi \in C^0(\I^\tau,V)$, it holds 
\[
\|\phi - \mathcal{I}^\tau \phi\|_{H^r(I_n,V)} \lesssim \tau^{\ell-r} |\phi|_{H^\ell(I_n,V)} 
\qquad \ell \in \{1,2\} \, , \ r \in \{0,1\}  \, .
\]
\end{lemma}
We close the section with the following well known result.
\begin{proposition}[Gronwall lemma]
\label{prp:gronwall}
Let $\{y_n\}_{n \in \N}$,  $\{f_n\}_{n \in \N}$ and  $\{g_n\}_{n \in \N}$ be nonnegative sequences such that 
\[
y_n \leq f_n + \sum_{j=0}^{n-1} g_j y_j 
\qquad \text{for $n \geq 0$.}
\]
Then
\[
y_n \leq f_n + \sum_{j=0}^{n-1} f_j g_j \, \exp \left(\sum_{k=j+1}^{n-1} g_j \right)
\qquad \text{for $n \geq 0$.}
\]
\end{proposition}

\section{Stabilized method}
\label{sec:stab}

In the present section we describe the proposed method and the associated consistency and stability results.

\subsection{Finite Elements and Discontinuous Galerkin time method}
\label{sub:fem}

Let us consider a family of conforming finite element spaces
\begin{equation}
\label{eq:spaces}
\PDG \subset \PCG \,,
\quad
\VDG \subset \VCG \,,
\quad
\QDG \subset \QCG
\end{equation}
satisfying the following properties:

\begin{description}

\item[\textbf{(P1) sub-complex structure.}]
The FEM couple $(\VDG,\QDG)$ is inf-sup stable \cite{boffi-brezzi-fortin:book} and the FEM 
spaces \eqref{eq:spaces} compose the sub-complex (cf. \eqref{eq:exact cont}) \cite{john-linke-merdon-neilan-rebholz:2017}
\begin{equation}
\label{eq:exact disc}
0 \, \xrightarrow[]{ \,\,\,\,\, \text{{$i$}} \,\,\,\,\,}  \,
\PDG \, \xrightarrow[]{   \quad \text{{$\CC$}} \quad  }\,
\VDG \, \xrightarrow[]{  \, \, \, \, \text{{$\dd$}} \, \, \, \, }\,
\QDG \, \xrightarrow[]{ \,\,\,\,\, 0 \,\,\,\,\,} \,
0 \,.
\end{equation}

\item[\textbf{(P2) polynomial inclusions.}] $[\Pk_k(\Omega_h)]^2 \subseteq \VDG \subseteq [\Pk_{k+3}(\Omega_h)]^2$, for some $k \in {\mathbb N}$.

\item[\textbf{(P3) interpolation estimates.}]  The following estimates hold
\begin{itemize}
\item for any $\phi \in \PCG \cap H^{s+2}(\Omega_h)$ with $0 < s \leq k$, there exists $\phi_h \in \PDG$ s.t.
\begin{equation} 
\label{eq:int est 1}
\|\phi - \phi_h\|_{E}  + 
h_E |\phi - \phi_h|_{1,E}  +
h_E^2 |\phi - \phi_h|_{2,E} 
\lesssim
h_E^{s+2} |\phi|_{s+2, E} 
\end{equation}
\item for any $\vv \in \VCG \cap [H^{s+1}(\Omega_h)]^2$ with $0< s \leq k$, there exists $\vv_h \in \VDG$ s.t.
\begin{equation} 
\label{eq:int est 2-a}
\|\vv - \vv_h\|_{E}  + 
h_E |\vv - \vv_h|_{1,E}  
\lesssim
h_E^{s+1} |\vv|_{s+1, E} 
\end{equation}
\item for any $q \in \QCG \cap H^{s}(\Omega_h)$ with $0 < s \leq k$, there exists $q_h \in \QDG$ s.t.
\begin{equation} 
\label{eq:int est 3}
\|q - q_h\|_{E}   
\lesssim
h_E^{s} |q|_{s, E} 
\end{equation}
\end{itemize}
for any $E \in \Omega_h$.
\end{description}

\noindent
Consider the discrete kernel
\begin{equation}
\label{eq:ZDG}
\ZDG := \{\vv_h \in \VDG \,\,\, \text{s.t.}  \,\,\, b(\vv_h, q_h) = 0 \,\,\, \text{for all $q_h \in Q_h$} \} 
\end{equation}
then, as a consequence of \textbf{(P1)}--\textbf{(P3)} the following hold

\begin{description}

\item[\textbf{(P4) kernel inclusion.}] The discrete kernel satisfies the inclusion $\ZDG \subset \ZCG$.

\item[\textbf{(P5) kernel characterization.}] The discrete kernel is such that $\ZDG = \CC (\PDG)$.

\item[\textbf{(P6) kernel interpolation estimates.}]  for any $\vv \in \ZCG \cap H^{s+1}(\Omega_h)$ with $0< s \leq k$, there exists $\vv_h \in \ZDG$ s.t. for any $E \in \Omega_h$
\begin{equation} 
\label{eq:int est 2}
\|\vv - \vv_h\|_{E}  + 
h_E |\vv - \vv_h|_{1,E}  
\lesssim
h_E^{s+1} |\vv|_{s+1,E} \,.
\end{equation}
\end{description}

\noindent
In the following we also assume that the spaces $\PDG , \VDG, \QDG$ have a local basis, a condition which is trivially satisfied by (essentially) all finite element spaces in the literature.
Let us define the following $L^2$-projection operator
${\Pi}_h \colon \ZCG \to \ZDG$ 
\begin{equation}
\label{eq:P0_Z}
\int_{\Omega} \ww_h
(\vv - \, {\Pi}_h  \vv) \, {\rm d} \Omega = 0 \qquad  \text{for all $\vv \in \ZCG$  and $\ww_h \in  \ZDG$.} 
\end{equation}
The following global approximation result is not trivial, but follows from \eqref{eq:int est 2} and classical techniques; we report the proof briefly for completeness. Note that mesh Assumption \textbf{(M2)} is important here.

\begin{lemma}
\label{lm:proj-bis}
For any smooth enough function $\vv \in \ZCG$, it holds 
\begin{align}
\label{eq:bHm}
|\vv - \Pi_h \vv|_{H^m(\Omega_h)} &\lesssim h^{s+1-m} |\vv|_{H^{s+1}(\Omega_h)} 
\qquad  &\text{$s,m \in \N$, $m \leq s+1$, $ 0 < s \leq k$,}
\\
\label{eq:bWminf}
|\vv - \Pi_h \vv|_{W^m_{\infty}(\Omega_h)} &\lesssim h^{s+1-m} 
|\vv|_{W^{s+1}_{\infty}(\Omega_h)} 
\qquad  &\text{$s,m \in \N$, $m \leq s+1$, $ 0 < s \leq k$.}
\end{align}
\end{lemma}
\begin{proof}
. \,
We first analyse the bound \eqref{eq:bHm}. Note that clearly, recalling {\bf (P6)}, 
\begin{equation}\label{new-1}
\|\vv - \Pi_h\vv\|  \lesssim h^{s+1} |\vv|_{H^{s+1}(\Omega_h)} \,.
\end{equation}
We now write
\[
\begin{aligned}
& |\vv - \Pi_h \vv|_{H^m(\Omega)}^2 =
\sum_{E \in \Omega_h} |\vv - \Pi_h \vv|_{m,E}^2  
\lesssim
\sum_{E \in \Omega_h} \left( |\vv - \Pi_k \vv|_{m,E}^2 +
|\Pi_h \vv - \Pi_k \vv|_{m,E}^2 \right)
\\
&
 \begin{aligned}
&  \lesssim
\sum_{E \in \Omega_h} \left( h_E^{2(s+1-m)} |\vv|_{s+1,E}^2 + 
|\Pi_k \vv - \Pi_h \vv |_{m,E}^2 \right)
\quad
& \text{(by Lemma \ref{lm:bramble})}
\\
& \lesssim
\sum_{E \in \Omega_h} \left(  h_E^{2(s+1-m)} |\vv|_{s+1,E}^2 + 
h_E^{-2m}\|\Pi_k \vv - \Pi_h \vv \|_{E}^2 \right)
\quad
& \text{(by  \textbf{(P2)} \& Lemma \ref{lm:inverse})}
\\
& \lesssim
\sum_{E \in \Omega_h} \left( h_E^{2(s+1-m)} |\vv|_{s+1,E}^2 + 
h_E^{-2m}\|\vv - \Pi_h \vv \|_{E}^2 \right)
\quad
& \text{(by a tri.ineq. \& Lemma \ref{lm:bramble})}
\\
& \lesssim
h^{2(s+1-m)} |\vv|_{H^{s+1}(\Omega_h)}^2 + 
h^{-2m} \| \vv  - \Pi_h \vv \|^2 
\quad
& \text{(by \textbf{(M2)})}
\\
& \lesssim
h^{2(s+1-m)} |\vv|_{H^{s+1}(\Omega_h)}^2 \, .
\quad
& \text{(by \eqref{new-1})}
\end{aligned}
\end{aligned}
\]
Concerning bound \eqref{eq:bWminf}, we start observing that the space $\ZDG$ has a local basis, simply obtained by applying the $\CC$ operator to the (local) basis of $\PDG$. By a scaling argument, and recalling assumption \textbf{(M2)}, it is easy to check that all conditions in \cite{Douglas:1974} are satisfied by $\ZDG$. Therefore, by applying \cite{Douglas:1974} we obtain
$$
\|\vv - \Pi_h\vv\|_{L^\infty(\Omega)}  \lesssim 
\min_{\ww \in \ZDG} \|\vv - \ww\|_{L^\infty(\Omega)}  
\leq \|\vv - \vv_h\|_{L^\infty(\Omega)}  \, ,
$$
where $\vv_h$ is the approximant introduced in {\bf (P6)}.
The remaining steps are fairly standard: we first apply two triangle inequalities combined with an inverse estimate, then Lemma \ref{lm:bramble} and {\bf (P6)}, finally the Holder inequality. 
We obtain
\begin{equation}\label{new-2}
\begin{aligned}
& \|\vv - \Pi_h\vv\|_{L^\infty(\Omega)}  
 \lesssim \max_{E \in \Omega_h} \big( \|\vv - \Pi_k \vv \|_{L^\infty(E)} + h_E^{-1} \| \Pi_k \vv - \vv \|_{L^2(E)} 
+ h_E^{-1} \| \vv - \vv_h\|_{L^2(E)} \big) \\
& \lesssim \max_{E \in \Omega_h} \big( h_E^{s+1} |\vv |_{W^{s+1}_\infty(E)} + h_E^{s} | \vv |_{s+1,E} \big) 
\lesssim \max_{E \in \Omega_h} h_E^{s+1} |\vv |_{W^{s+1}_\infty(E)}
\leq h^{s+1} |\vv|_{W^{s+1}_{\infty}(\Omega_h)} \,.
\end{aligned}
\end{equation}
Extending the above approximation results from the $L^\infty(\Omega)$ norm to the $W^m_\infty(\Omega)$ norm follows the same arguments as for the $p=2$ case shown above.
\end{proof}

In order to define a fully discrete scheme let us consider the time DG of polynomial degree one.
Therefore we define the discrete spaces
\begin{eqnarray}
\label{eq:disc spaces}
& \VDGt := \Pk_1(\I^\tau, \VDG)  \, , \quad
\QDGt := \Pk_1(\I^\tau, \QDG) \, , \quad
\ZDGt := \Pk_1(\I^\tau, \ZDG)  \, .
\end{eqnarray}

\noindent
Let us introduce the projection $\mathcal{P} \colon C^0(\I^\tau, \ZCG)\to \ZDGt$ defined by 
\begin{equation}
\label{eq:calP}
\mathcal{P}\vv := \mathcal{I}^\tau (\Pi_h \vv) =
\Pi_h(\mathcal{I}^\tau \vv) \quad 
\text{for all $\vv \in C^0(\I^\tau, \ZCG)$.}
\end{equation}

\begin{lemma}
\label{lm:P}
For any $I_n \in \I^\tau$, for any smooth enough function $\vv$ it holds
\[
|\vv - \mathcal{P}\vv|_{H^r(t_n, t_{n+1}, W^m_p(\Omega_h))}
\lesssim
\tau^{\ell-r} |\vv |_{H^{\ell}(t_n, t_{n+1}, W^{m}_p(\Omega_h))} +
h^{s+1-m} |\vv|_{H^{\overline r}(t_n, t_{n+1}, W^{s+1}_p(\Omega_h))}
\]
for all $r$, $s$, $m$, $l \in \N$, with $1 \le \ell \le 2$, $0 \leq r \leq \ell$ and 
, $m \leq s+1 \leq k+1$, $p \in \{2, \infty \}$. 
Here above $\overline{r}=\max{\{r,1\}}$.
\end{lemma}

\begin{proof}
. \,
From triangular inequality and definition \eqref{eq:calP} we infer
\[
|\vv - \mathcal{P}\vv|_{H^r(t_n, t_{n+1}, W^m_p(\Omega_h))}
\leq
|\vv - \mathcal{I}^\tau\vv|_{H^r(t_n, t_{n+1}, W^m_p(\Omega_h))}
+
|\mathcal{I}^\tau(\vv - \Pi_h\vv)|_{H^r(t_n, t_{n+1}, W^{m}_p(\Omega_h))}\,.
\]
The first term is bounded by applying Lemma \ref{lm:time-interp} with $V=W^m_p(\Omega_h)$.
For the second term 
\[
\begin{aligned}
&|\mathcal{I}^\tau(\vv - \Pi_h\vv)|^2_{H^r(t_n, t_{n+1}, W^{m}_p(\Omega_h))}
\lesssim
\int_{t_n}^{t_{n+1}}\left|\frac{\partial^r}{\partial t^r}(\mathcal{I}^\tau\vv(\cdot, t) - \Pi_h\mathcal{I}^\tau \vv(\cdot, t))\right|^2_{W^{m}_p(\Omega_h)} \,{\rm d}t
\\
& \qquad 
\begin{aligned}
& \lesssim
\int_{t_n}^{t_{n+1}} h^{2(s+1-m)} \left|\frac{\partial^r}{\partial t^r}\mathcal{I}^\tau \vv(\cdot, t)\right|^2_{W^{s+1}_p(\Omega_h)} \,{\rm d}t
\quad & \text{(by Lemma \ref{lm:proj-bis})}
\\
& \lesssim
h^{2(s+1-m)} |\vv|^2_{H^{\bar r}(t_n, t_{n+1}, W^{s+1}_p(\Omega_h))} \, .
\quad & \text{(by Lemma \ref{lm:time-interp})}
\end{aligned}
\end{aligned}
\]
\end{proof}

\begin{rmk}
Some finite element complexes (see for instance Fig. 4.1 in \cite{john-linke-merdon-neilan-rebholz:2017}) are based on macro-elements. As a consequence, properties \eqref{eq:int est 1} and \eqref{eq:int est 2} will hold on each macro-element instead of each element. Since the number of elements in each macro-element is uniformly bounded and they all have equivalent diameter, the previous results trivially extend also to such case.
\end{rmk}

Let $\uht_0 \in \VDG$ a suitable approximation of the initial data $\uu_0$. 
Then referring to the spaces \eqref{eq:disc spaces} and employing the DG time stepping method \cite[Chapter 12]{tohomee:book}, the fully discrete scheme for the Navier-Stokes equation \eqref{eq:ns variazionale} is given by
\begin{equation}
\label{eq:ns fem}
\left\{
\begin{aligned}
& \text{find $(\uht, \pht) \in \VDGt \times \QDGt$ such that for $n=0, \dots, N$} \\
&\begin{aligned}
\int_{I_n }(\uht_t, \vht) \, {\rm d}t +
 \int_{I_n} \nu  a(\uht, \vht) \,{\rm d}t + 
\int_{I_n} c(\uht; \, \uht, \vht) \,{\rm d}t  &+
\\ 
+\int_{I_n}b(\vht, \pht) \,{\rm d}t  + 
(\jump{\uht}_{t_n}, \vht^n_+) &=
\int_{I_n}(\ff, \vht)\,{\rm d}t \,\,\, & \text{$\forall \vht \in \Pk_1(I_n, \VDG)$,} \\
\int_{I_n}b(\uht, \qht) \,{\rm d}t &= 0 \,\,\, & \text{$\forall \qht \in \Pk_1(I_n, \QDG)$,}
\\
\uht^0_- &= \uht_0 \,.
\end{aligned}
\end{aligned}
\right.
\end{equation}
Notice that as a consequence of property \textbf{(P4)} the discrete solution $\uht$ of \eqref{eq:ns fem} is exactly divergence-free.

\subsection{Stabilization of the vorticity equation}
\label{sub:method}

It is well known that problem  \eqref{eq:ns fem}, although guaranteeing a divergence-free and pressure robust approximation,
would lead to a sub-optimal order of error reduction $O(h^k)$ in the $L^\infty(0,T;L^2(\Omega))$ norm. Furthermore, no direct control on the convection would be given by the error estimates.
As discussed for instance in \cite{garcia}, stabilized schemes can reach a better order reduction rate in the convection dominated regime, and in addition gain control on a stronger norm including convection-type terms.
In the following section we present the vorticity stabilized approach for the advection dominated Navier-Stokes equation. 
In order for our method to be well defined, we need \emph{only} the additional condition on the data $\cc\ff \in L^2(0, t_{N+1}, L^2(\Omega))$.
Nevertheless, we now require the following more stringent assumptions, that are needed to define the stabilized continuous problem below. We assume that the velocity solution $\uu$ of \eqref{eq:ns primale} satisfies the following

\noindent
\textbf{(A1) Regularity Assumption for the consistency of the method.} 
\begin{equation*}
\label{eq:regu}
\begin{aligned}
i) &\,\, \cc \uu_t \in  L^2(0, t_{N+1}, [L^2(\Omega)])\,,
&\quad
ii) &\,\, \uu \in L^2(0, t_{N+1},  [H^{3}(\Omega_h)]^2)\,,
\\
iii) &\,\, \uu \in C^0((0, t_{N+1}), \ZCG) \,,
&\quad
iv) &\,\, \uu \in L^2(0, t_{N+1}, [H^{3/2+\epsilon}(\Omega)]^2)\,,
\end{aligned}
\end{equation*}
for some $\epsilon >0$. 

\noindent
Then we consider the so-called vorticity equation on each element of $\Omega_h$, i.e.
\begin{equation}
\label{eq:ns residuo}
\cc(\uu_t  + (\Gr \uu ) \,\uu)   + 
 \cc(-\nu \, \divv (\epseps (\uu)))  
=  \cc \ff
\qquad 
\text{for all $E \in \Omega_h$, for a.e. $t \in I$.}
\end{equation}
Clearly the vorticity equation has the advantage that 
the gradient of the pressure disappears, therefore in the resulting numerical scheme the stabilization associated with \eqref{eq:ns residuo}, differently from the classical velocity-pressure stabilization, does not spoil the benefits related to the divergence-free property.

Following the approach introduced in \cite{Linke:2020} for the stationary Oseen problem, we propose a SUPG vorticity stabilization, that is based on the residual-based stabilization of the vorticity equation \eqref{eq:ns residuo}.
For all sufficiently regular functions let us introduce the following stabilizing forms and the stabilizing right hand side:
\begin{align}
\label{eq:CtE}
\CstabG(\ww_1, \ww_2; \, \uu, \vv) &:= \sum_{E \in \Omega_h}
\delta_E \int_E \cc(\uu_t + (\Gr \uu)\ww_1) \, \cc(\vv_t + (\Gr \vv) \ww_2)\, {\rm d}E 
\\
\label{eq:LtE}
\LstabG(\ww_2; \, \uu, \vv) &:= \sum_{E \in \Omega_h}\delta_E \int_E \cc(-\nu \divv (\epseps(\uu))) \,\cc(\vv_t + (\Gr \vv) \ww_2)\, {\rm d}E 
\\
\label{eq:jstab}
\JstabG(\ww_1, \ww_2; \, \uu, \vv) &:= \sum_{E \in \Omega_h}\frac{ \widehat{c} \, h_E^2}{2} \int_{\partial E} \jump{(\Gr \uu) \ww_1} \cdot 
\jump{(\Gr \vv) \ww_2}  \, {\rm d}e \,,
\\
\label{eq:FtE}
\FstabG(\ww_2; \, \vv) &:= \sum_{E \in \Omega_h}\delta_E \int_E \cc(\ff) \, \cc(\vv_t + (\Gr \vv) \ww_2) \, {\rm d}E
\end{align}
where the parameters $\delta_E$ will be specified later {and $\widehat{c}$ is a real constant uniformly bounded from above and below and introduced in order to allow some tuning of the stabilizing terms (see also \eqref{eq:delta}). Since such
uniform parameter does not affect the theoretical derivations, for the time being we
assume $\widehat{c} = 1$. 
We will be more precise about the practical values of
such constants in the numerical tests section.}
Let us finally introduce the form
\begin{equation}
\label{eq:KtG}
\KstabG(\ww_1, \ww_2; \, \uu, \vv):= \nu a(\uu, \vv) + c(\ww_1; \, \uu, \vv) + \CstabG(\ww_1, \ww_2; \, \uu, \vv)  +  \LstabG(\ww_2; \, \uu, \vv) + \JstabG(\ww_1, \ww_2; \, \uu, \vv)  \,.
\end{equation}
Referring to the spaces \eqref{eq:disc spaces}, the  forms \eqref{eq:KtG} and \eqref{eq:forma b},  the residual right hand side \eqref{eq:FtE}, the  vorticity-stabilized method for the Navier-Stokes equation is given by
\begin{equation}
\label{eq:ns stab}
\left\{
\begin{aligned}
& \text{find $(\uht, \pht) \in \VDGt \times \QDGt$ such that for $n=0, \dots, N$} \\
&\begin{aligned}
\int_{I_n } (\uht_t, \vht) \, {\rm d}t + (\jump{\uht}_{t_n}, \vht^n_+) &+
\int_{I_n} \KstabG(\uht, \uht; \, \uht, \vht) \,{\rm d}t +
\\
+\int_{I_n}b(\vht, \pht) \,{\rm d}t  
&=
\int_{I_n}(\ff, \vht) \,{\rm d}t + \int_{I_n} \FstabG(\uht; \, \vht) \,{\rm d}t 
 \,\,\, & \text{$\forall \vht \in \Pk_1(I_n, \VDG)$,} \\
\int_{I_n}b(\uht, \qht) \,{\rm d}t &= 0 \,\,\, & \text{$\forall \qht \in \Pk_1(I_n, \QDG)$,}
\\
\uht^0_- & = \uht_0 \,.
\end{aligned}
\end{aligned}
\right.
\end{equation}
As above we stress that the discrete velocity solution $\uht$ is exactly divergence-free.

\begin{rmk}
\label{rm:3d}
The proposed approach could be extended also to the three dimensional case, considering the $3D$ counterpart of
the stabilization of the vorticity equation \eqref{eq:ns residuo}, where the scalar-valued $\cc$ operator is replaced by the vector-valued operator $\CC:= \boldsymbol{\nabla} \times$.
%
\end{rmk}

\section{Theoretical Analysis} \label{sec:theo}

The aim of the following sections is to derive a FEM discretization of the stabilized equation \eqref{eq:ns stab}.
In the following the symbol $\lesssim$ will denote a bound up to a generic positive constant, independent of the mesh size $h$, of the time step $\tau$, of the parameters $\delta_E$, of the diffusive coefficient $\nu$,  of the problem solution $\uu$, of the problem data $\ff$ and $\uu_0$,
but which may depend on  $\Omega$ and $\I$, on the ``polynomial'' order of the method $k$, 
and on the mesh regularity constants $c_{{\rm M}1}$ and $c_{{\rm M}2}$ in Assumptions \textbf{(M1)} and \textbf{(M2)}.
%

\subsection{Stability analysis} \label{sub:stab}

Recalling the definition of discrete kernel $\ZDG$ in \eqref{eq:ZDG}, we consider the following formulation of \eqref{eq:ns stab}
\begin{equation}
\label{eq:ns thed}
\left\{
\begin{aligned}
& \text{find $\uht \in \ZDGt$ such that} \\
&\begin{aligned}
\AstabGN (\uht, \uht; \, \uht, \vht) & =
\rgN(\uht_0, \uht; \, \vht)
\,\,\, & \text{$\forall \vht \in \ZDGt$,}
\end{aligned}
\end{aligned}
\right.
\end{equation}
where
\begin{equation}
\label{eq:AstabG}
\begin{aligned}
\AstabGn (\ww_1, \ww_2; \, \vv, \zz) :=&
\int_0^{t_{n+1}} (\vv_t, \zz) \,{\rm d}t  +
\sum_{j=1}^n  (\jump{\vv}_{t_j}, \zz^j_+) +
(\vv^0_+, \zz^0_+) +
\\
& + \int_0^{t_{n+1}}\KstabG(\ww_1, \ww_2; \, \vv, \zz) \,{\rm d}t  
\end{aligned}
\end{equation}
\begin{equation}
\label{eq:rg}
\rgn (\boldsymbol{\sigma}, \ww_2; \, \zz) := 
\int_0^{t_{n+1}} (\ff, \zz)\, {\rm d}t +
\int_0^{t_{n+1}} \FstabG(\ww_2; \, \zz) \, {\rm d}t + 
(\boldsymbol{\sigma}, \zz^0_+) \,. 
\end{equation}
A first important remark is in order. 
Under the regularity Assumptions \textbf{(A1)} 
and recalling the kernel inclusion \textbf{(P5)},
the proposed method is strong consistent, i.e. 
the exact solution $\uu$  of the continuous problem \eqref{eq:ns variazionale ker},  is such that
\begin{equation}
\label{eq:ns thec}
\AstabGN(\uu, \uht; \, \uu, \vht)  =
\rgN(\uu_0, \uht; \, \vht)
\quad  \text{$\forall \vht \in \ZDGt$.}
\end{equation}

\noindent
For any $\wht \in \ZDGt$ we introduce the following norm 
\begin{equation}
\label{eq:norma}
\begin{aligned}
\| \vv \|_{\wht, n}^2 :=&
\frac{1}{2} \bigl( 
\|\vv^{n+1}_-\|^2 +
\sum_{j=1}^n \|\jump{\vv}_{t_j}\|^2 + \|\vv^{0}_+\|^2
\bigr)
+ {\nu \|\epseps (\vv)\|^2_{L^2(0, t_{n+1}, L^2(\Omega))}}+
\\
&+ 
\sum_{E \in \Omega_h} \delta_E \|\cc(\vv_t + (\Gr \vv) \wht)\|^2_{L^2(0, t_{n+1}, L^2(E))} +
\\
&+
\sum_{E \in \Omega_h} \frac{h_E^2}{2 } \|\jump{(\Gr \vv) \wht}\|^2_{L^2(0, t_{n+1}, L^2(\partial E))} \,.
\end{aligned}
\end{equation}
The following results are instrumental to prove the well-posedness of problem \eqref{eq:ns thed}.

\begin{proposition}[Coercivity]
\label{prp:coercivity}
Under the mesh Assumption \cfan{(M1)}, if the family of  parameters $\{\delta_E\}_{E \in \Omega_h}$ satisfies
\begin{equation}
\label{eq:delta_E}
\delta_E \leq \frac{h_E^4}{\nu \gamma_E^2} \qquad \forall E \in \Omega_h
\end{equation}
where $\gamma_E$ is the constant appearing in the inverse estimate of Lemma \ref{lm:inverse} for $m=2, p=2$, then for any $\wht \in \ZDGt$
the bilinear form $\AstabGn(\wht, \wht; \cdot, \cdot)$ satisfies for all $\vht \in \ZDGt$ the coerciveness inequality
\[
\AstabGn(\wht, \wht; \, \vht, \vht) 
\geq \frac{1}{2}  \|\vht\|^2_{\wht,n} \qquad n = 0,1,...,N \,.
\]
\end{proposition}

\begin{proof}
. \,

We preliminary observe that for $j=1, \dots, n$
\[
\begin{aligned}
\int_{t_j}^{t_{j+1}} (\vht_t, \vht) \, {\rm d}t + (\jump{\vht}_{t_j}, \vht_+^j) &=
\frac{1}{2}  \int_{t_j}^{t_{j+1}} \frac{d}{dt} \|\vht\|^2 \, {\rm d}t + 
\|\vht_+^j\|^2 - (\vht_-^j, \vht_+^j)
\\
& =
\frac{1}{2}  \|\vht_-^{j+1}\|^2 -  \frac{1}{2}  \|\vht_+^j\|^2 +
\|\vht_+^j\|^2 - (\vht_-^j, \vht_+^j)
\\
& =
\frac{1}{2} \bigl( \|\vht_-^{j+1}\|^2  + \|\jump{\vht}_{t_j}\|^2 - \|\vht_-^j\|^2  \bigr) \,.
\end{aligned}
\]
Therefore we have
\begin{equation}
\label{eq:coe2}
\begin{aligned}
\int_0^{t_{n+1}} (\vht_t, \vht) \,{\rm d}t  &+ 
\sum_{j=1}^n  (\jump{\vht}_{t_n}, \vht^n_+) +
\|\vht^0_+\|^2 =
\\
&= \int_{t_0}^{t_1} (\vht_t, \vht) \,{\rm d}t  + \|\vht^0_+\|^2 +
\sum_{j=1}^n \biggl( 
\int_{t_j}^{t_{j+1}} (\vht_t, \vht) \, {\rm d}t + 
(\jump{\vht}_{t_j}, \vht_+^j) 
\biggr) 
\\
& =
\frac{1}{2}\bigl( \|\vht_-^{1}\|^2 + \|\vht_+^0\|^2  \bigr)+
\frac{1}{2} \sum_{j=1}^n\bigl( \|\vht_-^{j+1}\|^2 + \|\jump{\vht}_{t_j}\|^2 - \|\vht_-^j\|^2  \bigr) 
\\
& = 
\frac{1}{2}  \bigl( \|\vht_-^{n+1}\|^2 + \sum_{j=1}^n  \|\jump{\vht}_{t_j}\|^2 +
\|\vht^0_+\|^2 \bigr) \,.
\end{aligned}
\end{equation}
Furthermore recalling \eqref{eq:c-skew}, direct computations yield 
\begin{equation}
\label{eq:coe3}
\begin{aligned}
&
\int_0^{t_{n+1}} \nu a(\vht, \vht) \, {\rm d} t \geq  \nu \|\epseps( \vht)\|^2_{L^2(0, t_{n+1}, L^2(\Omega))} \,,
\\
&\int_0^{t_{n+1}} c(\wht; \, \vht, \vht) \, {\rm d} t = 0 \,,
\\
&\int_0^{t_{n+1}} \CstabG(\wht, \wht; \, \vht, \vht) \, {\rm d} t =
\sum_{E \in \Omega_h} \delta_E \|\cc(\vht_t + (\Gr \vht) \wht)\|^2_{L^2(0, t_{n+1}, L^2(E))}\,,
\\
&\int_0^{t_{n+1}} \JstabG(\wht, \wht; \, \vht, \vht) \, {\rm d} t =
\sum_{E \in \Omega_h} \frac{h_E^2}{2} \|\jump{(\Gr \vht) \wht}\|^2_{L^2(0, t_{n+1}, L^2(\partial E))}\,.
\end{aligned}
\end{equation}
For the term $\LstabG(\wht; \, \cdot, \cdot)$
employing the Cauchy-Schwarz inequality,
Lemma \ref{lm:inverse} and since $\delta_E$ satisfies \eqref{eq:delta_E}, we infer
\begin{equation}
\label{eq:coe4}
\begin{aligned}
&\int_0^{t_{n+1}} \LstabG(\wht; \, \vht, \vht)  \, {\rm d}t
= \int_0^{t_{n+1}} \sum_{E \in \Omega_h}\delta_E \int_E \cc(-\nu \divv (\epseps(\vht))) \, \cc(\vht_t + (\Gr \vht) \wht) \, {\rm d}E  \,{\rm d}t
\\
&\quad \geq 
- \frac{1}{2} \sum_{E \in \Omega_h} \left(
\int_0^{t_{n+1}}   \delta_E \|\cc(\vht_t + (\Gr \vht) \wht)\|^2_E {\rm d}t +  
\int_0^{t_{n+1}}   \delta_E \nu^2 \|\cc(\divv(\epseps(\vht)))\|^2_E   \, {\rm d}t 
\right)
\\
& \quad\geq 
- \frac{1}{2} \sum_{E \in \Omega_h} \delta_E \|\cc(\vht_t + (\Gr \vht) \wht)\|^2_{L^2(0, t_{n+1}, L^2(E))} 
 - \frac{1}{2}  
\int_0^{t_{n+1}} \sum_{E \in \Omega_h} \frac{\delta_E \nu^2 \gamma_E^2}{h_E^4} \|\epseps(\vht)\|_E^2\, {\rm d}t
\\
& \quad\geq 
- \frac{1}{2} \sum_{E \in \Omega_h} \delta_E \|\cc(\vht_t + (\Gr \vht) \wht)\|^2_{L^2(0, t_{n+1}, L^2(E))} 
 - \frac{1}{2}  
\int_0^{t_{n+1}} \sum_{E \in \Omega_h} \nu \|\epseps(\vht)\|_E^2\, {\rm d}t
\\
& \quad\geq 
- \frac{1}{2} \sum_{E \in \Omega_h} \delta_E \|\cc(\vht_t + (\Gr \vht) \wht)\|^2_{L^2(0, t_{n+1}, L^2(E))} 
 - \frac{1}{2} 
\nu \|\epseps(\vht)\|^2_{L^2(0, t_{n+1}, L^2(\Omega))} \,.
\end{aligned}
\end{equation}
The thesis follows collecting \eqref{eq:coe2}, \eqref{eq:coe3} and \eqref{eq:coe4}.
\end{proof}

\begin{lemma}
\label{lm:1}
Let $\wht \in \ZDGt$ and let $\vv \in \Pk_1(\I^\tau, \VCG)$ satisfying items $i)$ and $ii)$ in Assumption \cfan{(A1)}, 
so that $\| \vv \|_{\wht,n}^2$ is well defined.
Then it holds that
\[
\|\vv\|^2_{L^2(0, t_{n+1}, L^2(\Omega))}
\lesssim  \tau \| \vv \|_{\wht,n}^2 +  \tau \sum_{j=1}^{n-1} \| \vv \|_{\wht,j}^2  \,.
\]
\end{lemma}

\begin{proof}
. \,
For any $p(t) \in \Pk_1(\I^\tau)$, a simple calculation shows that
\begin{equation}
\label{eq:lm1-1}
\|p\|_{L^2(\I_j)}^2 
\leq
\frac{\tau}{2} \bigl( (p_-^{j+1})^2 + (p_+^j)^2 \bigr)  \qquad j=0,1,...,N \,.
\end{equation}
Applying the trivial generalization of \eqref{eq:lm1-1} to the function-valued case we infer
\[
\begin{aligned}
\|\vv\|^2_{L^2(0, t_{n+1}, L^2(\Omega))} &= 
\sum_{j=0}^n \int_{t_j}^{t_{j+1}} \|\vv\|^2 \,{\rm d}t
\leq
\sum_{j=0}^n  \frac{\tau}{2} 
\bigl( \|\vv^{j+1}_-\|^2 + \|\vv^j_+\|^2\bigr)
\\
& \leq
\sum_{j=1}^n  \tau
\bigl( \|\vv^{j+1}_-\|^2 + \|\vv^j_-\|^2 + \|\jump{\vv}_{t_j}\|^2\bigr) + 
\frac{\tau}{2}
\bigl( \|\vv^{1}_-\|^2 + \|\vv^0_+\|^2\bigr)
\\
& \leq
\tau \bigl( \|\vv^{n+1}_-\|^2 + 
\sum_{j=1}^n  \|\jump{\vv}_{t_j}\|^2 +
\|\vv^0_+\|^2 \bigr) + 
2 \tau\sum_{j=1}^n  
\|\vv^j_-\|^2 
\\
& \leq
2 \tau \| \vv \|_{\wht, n}^2 + 
4 \tau \sum_{j=1}^{n-1} \| \vv \|_{\wht, j}^2\,.
\end{aligned}
\]
\end{proof}


\begin{lemma}
\label{lm:2}
Let $\wht \in \ZDGt$ and let $\vv \in \Pk_1(\I^\tau, \VCG)$ satisfying items $i)$ and $ii)$ in Assumption \cfan{(A1)}, 
so that $\| \vv \|_{\wht,n}^2$ is well defined.
Then the following hold
\begin{align}
\label{eq:lm2-1}
\rgn (\uht_0, \wht; \vv)
&\lesssim
\max \{1, \nu^{-1/2} \}
\bigl(\cfun{DATA}(n)\bigr)^{1/2} 
\|\vv\|_{\wht,n}\,,
\\
\label{eq:lm2-2}
\rgn (\uht_0, \wht; \vv)
&\lesssim
\bigl(\cfun{DATA}(n)\bigr)^{1/2} 
\left( (1 + \tau)\|\vv\|_{\wht,n}^2 + \tau \sum_{j=1}^{n-1} \| \vv \|_{\wht,j}^2 \right)^{1/2}\,,
\end{align}
where for $j=1, \dots, N$
\[
\cfun{DATA}(j) :=  \|\uht_0\|^2 +  \|\ff\|^2_{L^2(0, t_{j+1}, L^2(\Omega))} + 
\sum_{E \in \Omega_h} \delta_E \|\cc \ff \|_{L^2(0, t_{j+1},L^2(E))}^2 \,.
\]
 \end{lemma}

\begin{proof}
. \,
From definition \eqref{eq:rg} we need to bound the following terms
\[
\begin{aligned}
\rgn (\uht_0, \wht; \, \vv) &\leq 
\int_0^{t_{n+1}} \|\ff\| \|\vv\| \,{\rm} dt +
\int_0^{t_{n+1}} |\FstabG(\wht; \, \vv) | \, {\rm d}t + \|\uht_0\| \| \vv^0_+\|
=: r_1 + r_2 + r_3 \,.
\end{aligned}
\]
For the term $r_1$, Cauchy-Schwarz inequality yields
\[
r_1 \leq 
\|\ff\|_{L^2(0, t_{n+1}, L^2(\Omega))} \, 
\|\vv\|_{L^2(0, t_{n+1}, L^2(\Omega))} 
\,. 
\]
From the previous bound, employing the Korn inequality we infer 
\begin{equation}
\label{eq:r1-1}
r_1 \lesssim 
\|\ff\|_{L^2(0, t_{n+1}, L^2(\Omega))} \, 
\|\epseps(\vv)\|_{L^2(0, t_{n+1}, L^2(\Omega))} 
\lesssim 
\nu^{-1/2}\|\ff\|_{L^2(0, t_{n+1}, L^2(\Omega))} \, 
\|\vv\|_{\wht, n} 
\,.
\end{equation}
Alternatively, employing Lemma \ref{lm:1}, we have
\begin{equation}
\label{eq:r1-2}
r_1 \lesssim 
\|\ff\|_{L^2(0, t_{n+1}, L^2(\Omega))} \, 
\left(\tau \| \vv \|_{\wht, n}^2 + 
\tau \sum_{j=1}^{n-1} \| \vv \|_{\wht, j}^2  \right)^{1/2}
\,.
\end{equation}
Direct computations show that
\[
\begin{aligned}
r_2 &\leq\int_{0}^{t_{n+1}} 
\sum_{E \in \Omega_h} \delta_E \|\cc \ff\|_E 
\|\cc(\vv_t + (\Gr \vv) \wht )\|_E \, {\rm d}t
\\
& \leq  
\sum_{E \in \Omega_h} \delta_E \|\cc \ff\|_{L^2(0, t_{n+1}, L^2(E))} \,
\|\cc(\vv_t + (\Gr \vv) \wht )\|_{L^2(0, t_{n+1}, L^2(E))} 
\\
& \leq  
\biggl( \sum_{E \in \Omega_h}  \delta_E\|\cc \ff\|_{L^2(0, t_{n+1}, L^2(E))}^2 \biggr)^{1/2} \,
\biggl( \sum_{E \in \Omega_h} \delta_E
\|\cc(\vv_t + (\Gr \vv) \wht )\|_{L^2(0, t_{n+1}, L^2(E))}^2
\biggr)^{1/2} 
\\
& \leq  
\biggl( \sum_{E \in \Omega_h}  \delta_E\|\cc \ff\|_{L^2(0, t_{n+1}, L^2(E))}^2 \biggr)^{1/2} \, \|\vv\|_{\wht, n} \,,
\end{aligned}
\]
and
\[
r_3 \leq 
\|\uht_0\| \, \|\vv^0_+\|
\leq 
\|\uht_0\| \, \|\vv\|_{\wht, n} \,.
\]
Estimates \eqref{eq:lm2-1} and \eqref{eq:lm2-2} can be easily derived employing the bounds for $r_2$ and $r_3$ and bounds \eqref{eq:r1-1} and \eqref{eq:r1-2} respectively.
\end{proof}
The following result establishes the existence of discrete solutions.
\begin{proposition}[Well-posedness]
\label{prp:stability}
Under the mesh Assumption \cfan{(M1)}, if the family of  parameters $\{\delta_E\}_{E \in \Omega_h}$ satisfies \eqref{eq:delta_E}, problem \eqref{eq:ns thed} admits at least a solution $\uht \in \ZDGt$. 
Moreover, assuming $\tau \lesssim 1$, for $n=0,\dots, N$ the following bound holds
\begin{equation}
\label{eq:bound}
\|\uht\|_{\uht, n}^2 \lesssim \cfun{DATA}(n) \,.
\end{equation}
\end{proposition}

\begin{proof}
. \,


For any $\wht \in  \ZDGt$ and for $n=0, \dots, N$, let $\beth_{ \wht,n} \colon \ZDGt \to \R$ given by
\[
\beth_{ \wht,n}(\vht) :=
\AstabGn(\wht, \wht; \, \wht, \vht) -
\rgn(\uht_0, \wht; \, \vht) 
\quad \text{for all $\vht \in \ZDGt$.}
\]
Let us split the norm $\| \vht \|_{\wht, n}^2 = ||| \vht |||_{n}^2 + ||| \vht |||_{\wht, n}^2$, where
\[
\begin{aligned}
||| \vht |||_{n}^2 :=&
\frac{1}{2} \bigl( 
\|\vht^{n+1}_-\|^2 +
\sum_{j=1}^n \|\jump{\vht}_{t_j}\|^2 + \|\vht^{0}_+\|^2
\bigr)
+ \nu \|\epseps(\vht)\|^2_{L^2(0, t_{n+1}, L^2(\Omega))}
\\
||| \vht |||_{\wht, n}^2 :=&
 \sum_{E \in \Omega_h} \left( \delta_E \|\cc(\vht_t + (\Gr \vht) \wht)\|^2_{L^2(0, t_{n+1}, L^2(E))}
+ \frac{h_E^2}{2 } \|\jump{(\Gr \vht) \wht}\|^2_{L^2(0, t_{n+1}, L^2(\partial E))}
\right).
\end{aligned}
\]
Note that $||| \cdot |||_{n}^2$ is a norm on $\ZDGt$.
Then employing Proposition \ref{prp:coercivity}
and \eqref{eq:lm2-1}, for any $\vht \in \ZDGt$, the following holds
\[
\begin{aligned}
\beth_{ \vht,n}(\vht) &\geq
\frac{1}{2}|||\vht|||_n^2 +
\frac{1}{2}|||\vht|||_{\vht,n}^2 - 
c_{{\mathcal R}} \max \{1, \nu^{-1/2} \}\texttt{DATA}(n)^{1/2} \left( |||\vht|||_n  + |||\vht|||_{\vht, n} \right)
\end{aligned}
\]
where  $c_{{\mathcal R}}$ denotes the hidden constant in \eqref{eq:lm2-1} and \eqref{eq:lm2-2}.
Then from the Cauchy-Schwarz inequality
\begin{equation}
\label{eq:sta1}
\begin{aligned}
\beth_{ \vht,n}(\vht) &\geq
\frac{1}{2}|||\vht|||_n^2 - 
c_{{\mathcal R}} \max \{1, \nu^{-1/2} \}\texttt{DATA}(n)^{1/2} \, |||\vht|||_n  -
\frac{c_{{\mathcal R}}^2}{2} \max \{1, \nu^{-1} \}\texttt{DATA}(n) \,.
\end{aligned}
\end{equation}
Let 
$$\boldsymbol{\mathcal S}:= \left\{\vht \in \ZDGt \,\,\, \text{s.t.}\,\,\, |||\vht|||_n \leq (1 + \sqrt{2}) c_{\mathcal R} \max \{1, \nu^{-1/2} \} \texttt{DATA}(n)^{1/2}
\right\}$$
then \eqref{eq:sta1} implies
\[
\beth_{ \vht,n}(\vht) \geq 0 
\quad \text{for all $\vht \in \partial \boldsymbol{\mathcal S}$.}
\]
Therefore employing the fixed-point Theorem \cite[Chap. IV, Corollary 1.1]{girault-raviart:book} there exists $\uht \in \boldsymbol{\mathcal S}$ such that $\beth_{\uht,n}(\vht) =0$ for all $\vht \in \ZDGt$, i.e. $\uht$ is a solution of \eqref{eq:ns thed}.
Concerning the bound \eqref{eq:bound} employing Proposition \ref{prp:coercivity},  bound \eqref{eq:lm2-2} and the Cauchy-Schwarz inequality, we have
\[
\begin{aligned}
\frac{1}{2}\| \uht \|_{\uht, n}^2 & \leq
\AstabGn(\uht, \uht; \, \uht, \uht) =
\rgn(\uht_0, \uht; \, \uht) 
\\
& \leq 
c_{{\mathcal R}} \texttt{DATA}(n)^{1/2} \left( 
(1 + \tau)\| \uht \|^2_{\uht, n} + \tau \sum_{j=1}^{n-1} \| \uht \|_{\uht, j}^2 \right)^{1/2}
\\
& \leq 
c_{{\mathcal R}}^2 \texttt{DATA}(n) \left(\frac{5}{4} + \tau \right) +
\frac{1}{4}\| \uht \|^2_{\uht, n}  + \tau \sum_{j=1}^{n-1} \| \uht \|_{\uht, j}^2 
\,,
\end{aligned}
\]
then we obtain
\[
\| \uht \|^2_{\uht, n}  \leq 
c_{{\mathcal R}}^2 \texttt{DATA}(n) \left(5 + 4 \tau \right) +
4 \tau \sum_{j=1}^{n-1} \| \uht \|_{\uht, j}^2 \,.
\] 
Estimate \eqref{eq:bound} follows by the bound above and the Gronwall lemma \ref{prp:gronwall}.
\end{proof}

\noindent
Let $\uht \in \ZDGt$ be a solution of problem \eqref{eq:ns thed} then we introduce the following notation
\begin{equation}\label{eq:addedlater}
\| \cdot\|_{\nstabn} := \| \cdot\|_{\uht, n} 
\quad \text{for $n=0, \dots, N$.}
\end{equation}

\begin{corollary}
\label{cor:stability}
Let $\uht \in \ZDGt$ be a solution of Problem \eqref{eq:ns thed}.
Then for $n=0,\dots, N$ it holds that
\[
\|\uht\|^2_{L^{\infty}(t_n, t_{n+1}, L^2(\Omega))} \lesssim \cfun{DATA}(n) \,.
\]
\end{corollary}

\begin{proof}
. \,
Employing the bound \eqref{eq:bound} we have
\[
\begin{aligned}
&\|\uht_-^{n+1}\|^2 \leq 2\|\uht\|_{\nstabn}^2 \lesssim \cfun{DATA}(n)
\\
&\|\uht_+^n\|^2 \lesssim   
\|\jump{\uht}_{t_n}\|^2 + \|\uht_-^n\|^2
\leq
 \|\uht\|_{\nstabn}^2 +  \|\uht\|_{{\rm stab},n-1}^2 \lesssim
\texttt{DATA}(n) + \texttt{DATA}(n-1) \lesssim   \texttt{DATA}(n)  
\end{aligned}
\]
then being $\uht \in \Pk_1(\I^\tau,\ZDG)$ we obtain
\[
\|\uht(\cdot, t)\|^2 \lesssim \texttt{DATA}(n) \qquad \text{for all $t \in I_n$.}
\]
\end{proof}

\subsection{Error analysis} \label{sub:err}


Let $\uu \in C^0((0,t_{N+1}), \ZCG)$ and $\uht \in \ZDGt$ be the solutions of problem \eqref{eq:ns variazionale ker}  and problem \eqref{eq:ns thed}, respectively, and referring to \eqref{eq:calP}, let us define the following error functions
\[
\eei := \uu - \uui \,,
\qquad
\eea := \uht - \uui \,.
\]
In order to derive the error estimates, we require the following stronger regularity conditions on the solution $\uu$

\smallskip\noindent
\textbf{(A2) Regularity Assumption for the error estimates.} 
\[
\begin{aligned}
i)& \,\, \uu \in L^2(0, t_{N+1}, [H^{\textrm{max} \{ 3,\reg+1 \} } (\Omega_h)]^2) \,,
\\
ii)& \,\, \uu \in H^{1}(0, t_{N+1}, [H^{\reg+1} (\Omega_h)]^2) \,,  
&\quad
iii)& \,\, \uu \in L^2(0, t_{N+1}, [W^{\reg+1}_\infty (\Omega_h)]^2) \,,  
\\
iv)& \,\,  \uu \in H^{2}(0, t_{N+1}, [W^{2}_\infty (\Omega_h)]^2)
&\quad
 v)& \,\, \uu \in L^2(0, t_{N+1}, [H^{3/2+\epsilon}(\Omega)]^2)\,,
\end{aligned}
\]
for some $\epsilon >0$ and $0 < \reg \leq k$.

As it often happens in the theoretical analysis of complex problems, the above regularity assumption cannot be guaranteed in general. Therefore our analysis serves the purpose of showing that, at least under favourable conditions, the scheme is able to deliver an accurate and robust (in the sense described in the introduction) approximation.

In order to shorten some equations, in the following we will denote by $\Lambda$ a generic constant depending on the norms of $\uu$ in the spaces listed in Assumptions  \textbf{(A2)} (and, as usual, on $\Omega$, on $k$ and on the mesh shape regularity constants $c_{\rm{M}1}$ and $c_{\rm{M}2}$ but not on $\nu$, $h$, $\tau$, and the stabilization parameters $\delta_E$), whereas $C$ will denote a generic constant depending on $\Omega$, on $k$ and on the mesh shape regularity constants $c_{\rm{M}1}$ and $c_{\rm{M}2}$.

\noindent
According to Lemma \ref{lm:P}, under Assumptions \textbf{(M1)}, \textbf{(M2)} and \textbf{(A2)}  the following hold
\begin{align}
\label{eq:err-int-1}
&\|\eei^n\|^2_\Omega \leq
\Lambda   h^{2 \reg+ 2} 
\\
\label{eq:err-int-2}
&\|\eei\|^2_{L^2(0, t_{n+1}, L^2(\Omega))}  \leq
\Lambda  \left(\tau^{4} + h^{2 \reg+ 2} \right) 
\\
\label{eq:err-int-3}
&\|\Gr \eei\|^2_{L^2(0, t_{n+1}, L^2(\Omega))}  \leq
\Lambda
\left(\tau^{4} + h^{2 \reg} \right) 
\\
\label{eq:err-int-4}
&\|\cc \eei_t\|^2_{L^2(0, t_{n+1}, L^2(\Omega))} \leq
\Lambda
\left(\tau^{2} + h^{2 \reg} \right) 
\\
\label{eq:err-int-5}
&\|\eei\|^2_{L^2(0, t_{n+1}, W^1_{\infty}(\Omega))}  \leq
\Lambda 
\left(\tau^{4} + h^{2 \reg} \right) 
\\
\label{eq:err-int-6}
&\|\eei\|^2_{L^2(0, t_{n+1}, W^2_{\infty}(\Omega))}  \leq
\Lambda 
\left(\tau^{4} + h^{2 \reg-2} \right)
\\
\label{eq:err-int-7}
&\|\eei\|^2_{L^2(0, t_{n+1}, H^3(\Omega_h))}  \leq
\Lambda 
\left(\tau^{4} + h^{2 \reg-4} \right) \, .
\end{align}
In order to derive our analysis, we now make a choice for the parameter $\delta_E$, see also Section \ref{sec:num} for related practical observations. 
Let us set, for all $E \in \Omega_h$, 
\begin{equation}
\label{eq:delta}
\delta_E = c_E \min{\{ h_E^3 , h_E^4/\nu \}} \ , \qquad
\delta := \sup_{E \in \Omega_h} \delta_E \, ,
\end{equation}
where $c_E$ are positive real constants, uniformly bounded from above and below (for instance one can assume $c_E = c$, a fixed positive constant).  Note that, for sufficiently small choices of $c_E$, the above definition clearly satisfies \eqref{eq:delta_E}.

\noindent
We also introduce the following quantity in order to shorten the notation
\begin{multline}
\label{eq:sigmaE}
\sigma := \max\{
\nu \tau^{4},   
\nu h^{2 \reg}, 
\delta \tau^{2}, 
\delta h^{2 \reg-2}, 
\delta \tau^{4}h^{-2}, 
 h^{2 \reg+1}, 
\nu^2 \delta h^{2 \reg -4}, 
\nu^2 \delta \tau^{4}, 
\delta^{-1}\tau^{4}, 
\delta^{-1}h^{2 \reg+4},
h^{-3}\tau^{4}
 \} \,.
\end{multline}

\begin{proposition}[Estimate of $\|\eei\|_{\nstabn}$]
\label{prp:epi}
Under mesh Assumptions \cfan{(M1)}, \cfan{(M2)} and regularity Assumption \cfan{(A2)}, the term $\|\eei\|^2_{\nstabn}$ can be bounded as follows
\[
\|\eei\|^2_{\nstabn} \leq 
\Lambda \,  \sigma (1 +  \cfun{DATA}(n))\,.
\]
\end{proposition}

\begin{proof}
. \,
We preliminary observe that if the solution $\uu$ satisfies \textbf{(A1)}, the norm $\|\eei\|_{\nstabn}$ is well defined.
We need to estimate each term in the definition of the norm $\|\cdot\|_{\nstabn}$.

Employing \eqref{eq:err-int-1}  and since $\|\eei(\cdot)\| \in C^0(\I)$ we have
\begin{equation}
\label{eq:eei-1}
\begin{aligned}
\frac{1}{2} \bigl( 
\|\eei^{n+1}_-\|^2 +
\sum_{j=1}^n \|\jump{\eei}_{t_j}\|^2 + \|\eei^{0}_+\|^2
\bigr) 
= \frac{1}{2} \bigl( 
\|\eei^{n+1}_-\|^2 +
\|\eei^{0}_+\|^2
\bigr) 
\leq \Lambda \, h^{2 \reg + 2} 
\,.
\end{aligned}
\end{equation}
By \eqref{eq:err-int-3} we obtain
\begin{equation}
\label{eq:eei-2}
\begin{aligned}
\nu \|\epseps( \eei)\|^2_{L^2(0, t_{n+1}, L^2(\Omega))} \leq
\nu \|\Gr \eei\|^2_{L^2(0, t_{n+1}, L^2(\Omega))} 
\leq
\Lambda\, \nu \left(\tau^{4} +   h^{2 \reg} \right) 
\,.
\end{aligned}
\end{equation}
By \eqref{eq:err-int-4} we infer
\begin{equation}
\label{eq:eei-3a}
\begin{aligned}
\sum_{E \in \Omega_h} \delta_E  
\|\cc \eei_t\|^2_{L^2(0, t_{n+1}, L^2(E))} 
\leq 
 \delta  
\|\cc \eei_t\|^2_{L^2(0, t_{n+1}, L^2(\Omega))}
\leq
\Lambda \, \delta \left(\tau^{2} +   h^{2 \reg} \right) 
\end{aligned}
\end{equation}
where we recall $\delta$ was defined in \eqref{eq:delta}.
Employing Lemma \ref{lm:inverse}, we obtain
\[
\begin{aligned}
\sum_{E \in \Omega_h} &\delta_E  \|\cc((\Gr \eei) \uht)\|^2_{L^2(0, t_{n+1}, L^2(E))} 
=
\sum_{E \in \Omega_h} \delta_E  \int_0^{t_{n+1}}\|\cc((\Gr \eei) \uht)\|^2_E \, {\rm d}t
\\
& \leq 
 \sum_{E \in \Omega_h} \delta_E  
\int_0^{t_{n+1}} 
\left(|\eei|^2_{W^2_{\infty}(E)} \, \|\uht\|^2_E +
|\eei|^2_{W^1_{\infty}(E)} \, \|\Gr \uht\|^2_E
 \right) \, {\rm d}t
\\
& \leq
 \sum_{E \in \Omega_h} \delta_E 
 \int_0^{t_{n+1}} 
\bigl(|\eei|^2_{W^2_{\infty}(E)}   +
h_E^{-2}|\eei|^2_{W^1_{\infty}(E)} 
 \bigr) C \, \|\uht\|^2_E \, {\rm d}t 
\\
& \leq
\sum_{E \in \Omega_h} \delta_E  
\bigl(\|\eei\|^2_{L^2(0, t_{n+1}, W^2_{\infty}(E))}   +
h_E^{-2}\|\eei\|^2_{L^2(0, t_{n+1}, W^1_{\infty}(E))} 
 \bigr) C \, \|\uht\|^2_{L^{\infty}(0, t_{n+1}, L^2(E))}
 \,.
\end{aligned}
\]
Therefore employing \eqref{eq:err-int-5}  and \eqref{eq:err-int-6}, and Corollary \ref{cor:stability} we infer
\begin{equation}
\label{eq:eei-3b}
\begin{aligned}
\sum_{E \in \Omega_h} &\delta_E  \|\cc((\Gr \eei) \uht)\|^2_{L^2(0, t_{n+1}, L^2(E))} 
\\
&\lesssim \,
\delta \bigl(\|\eei\|^2_{L^2(0, t_{n+1}, W^2_{\infty}(\Omega_h))}   +
h^{-2}\|\eei\|^2_{L^2(0, t_{n+1}, W^1_{\infty}(\Omega_h))} 
 \bigr)\|\uht\|^2_{L^{\infty}(0, t_{n+1}, L^2(\Omega))}
\\
& \leq
\Lambda\, \delta \big( 
\tau^{4} + h^{2 \reg - 2} +  \tau^{4} h^{-2}
\big)\texttt{DATA}(n) \,.
\end{aligned}
\end{equation}
Finally for the last term in definition \eqref{eq:norma}, by \eqref{eq:utile-jump}, \eqref{eq:err-int-5} and Corollary \ref{cor:stability}
\begin{equation}
\label{eq:eei-4}
\begin{aligned}
\sum_{E \in \Omega_h} \frac{h_E^2}{2 } \|\jump{(\Gr \eei) \uht}\|^2_{L^2(0, t_{n+1}, L^2(\partial E))} &=
 \sum_{E \in \Omega_h} \frac{h_E^2}{2 } \int_0^{t_{n+1}} \|\jump{(\Gr \eei) \uht}\|^2_{\partial E} \, {\rm d}t
\\
& \lesssim
\int_0^{t_{n+1}} \!\! h \, |\eei|^2_{W^1_\infty(\Omega_h)} 
 \sum_{E \in \Omega_h} \|\uht\|^2_E  \, {\rm d}t
\\
& \lesssim
h \|\eei\|^2_{L^2(0, t_{n+1}, W^1_\infty(\Omega_h))} \|\uht\|^2_{L^\infty(0, t_{n+1}, L^2(\Omega))}
\\
&\leq
\Lambda h \left(\tau^{4}  + h^{2\reg}\right) 
\texttt{DATA}(n) \,.
\end{aligned}
\end{equation}
The thesis now follows collecting \eqref{eq:eei-1}--\eqref{eq:eei-4} and recalling the definition of $\sigma$ in \eqref{eq:sigmaE}.
\end{proof}

\begin{proposition}[Velocity error equation]
\label{prp:error}
Let $\uu \in C^0((0,t_{N+1}), \ZCG)$ and $\uht \in \ZDGt$ be the solutions of problem \eqref{eq:ns variazionale ker} and problem \eqref{eq:ns thed}, respectively.
Assume $\uu$ satisfies \cfan{(A1)}.
Then if the family of  parameters $\{\delta_E\}_{E \in \Omega_h}$ satisfies \eqref{eq:delta_E}, it holds that

\begin{equation}
\label{eq:abstract}
\frac{1}{2} \|\eea\|^2_{\nstabn} \leq
\err0 + \errc + \errC + \errJ +  \errA
\end{equation}
where
\begin{equation}
\label{eq:err-quan}
\begin{aligned}
\err0 &:= (\uht_0 - \uu_0, \eea^0_+) \,,
\\ 
\errc &:= \int_0^{t_{n+1}}c(\uu- \uht; \, \uu,  \eea) \,{\rm d}t\,,
\\ 
\errC &:= \int_0^{t_{n+1}}\left(\CstabG(\uu, \uht; \, \uu, \eea) - \CstabG(\uht, \uht; \, \uu, \eea)\right) \,{\rm d}t\,,
\\ 
\errJ &:= \int_0^{t_{n+1}}\JstabG(\uu - \uht, \uht; \, \uu, \eea) \,{\rm d}t\,,
\\ 
\errA &:= \AstabGn(\uht, \uht; \, \eei, \eea)\,.
\end{aligned}
\end{equation}
\end{proposition}
\begin{proof}
. \,

Employing Proposition \ref{prp:coercivity}, problem equations \eqref{eq:ns thec} and \eqref{eq:ns thed} and definitions \eqref{eq:AstabG} and \eqref{eq:rg}, direct computations yield
\[
\begin{aligned}
\frac{1}{2}&\|\eea\|^2_{\nstabn} \leq 
\AstabGn(\uht, \uht; \, \eea, \eea) = \AstabGn(\uht, \uht; \, \uht - \uui, \eea) 
\\
&=
\rgn(\uht_0, \uht; \, \eea) - \rgn(\uu_0, \uht; \, \eea) + 
\AstabGn(\uu, \uht; \, \uu, \eea)  -
\AstabGn(\uht, \uht; \, \uui, \eea) 
\\
&=
(\uht_0 - \uu_0, \eea^+_0) + 
\AstabGn(\uu, \uht; \, \uu, \eea)  - \AstabGn(\uht, \uht; \, \uu, \eea) +
\AstabGn(\uht, \uht; \, \eei, \eea) 
\\
&=
\err0 + \errc + \errC + \errJ +  \errA \,.
\end{aligned}
\]
\end{proof}

\begin{proposition}
\label{prp:err-A}
The terms $\err0$, $\errc$, $\errC$ and $\errJ$ in \eqref{eq:err-quan} can be bounded as follows
\[
\err0 + \errc + \errC + \errJ \leq
\Lambda \, \sigma +
\Lambda(1 + \delta h^{-2})\|\eea\|_{L^2(0, t_{n+1}, L^2(\Omega))}^2 
+ \frac{1}{8} \|\eea\|^2_{\nstabn} \,.
\]
\end{proposition}

\begin{proof}
. \,
We estimate separately each term.

\noindent
$\bullet$ estimate of $\err0$.
Employing the Cauchy-Schwarz inequality and assuming that the approximation of the initial data $\uht_0$ realizes \eqref{eq:int est 2-a}, we get
\begin{equation}
\label{eq:err0}
\err0 \leq \|\uht_0 - \uu_0 \| \|\eea_+^0\| \leq 
  \Lambda h^{2 \reg+2} 
+ \frac{1}{16} \|\eea_+^0\|^2 \,.
\end{equation}

\noindent
$\bullet$ estimate of $\errc$. Simple computations yield
\[
\begin{aligned}
\errc &\leq \int_0^{t_{n+1}} |\uu|_{W^1_\infty(\Omega_h)}  \|\uu- \uht\|    \|\eea\| \,{\rm d}t
\leq
\int_0^{t_{n+1}} |\uu|_{W^1_\infty(\Omega_h)} \left( \|\eei\|^2 + 2  \|\eea\|^2 \right) \,{\rm d}t 
\\
& \leq
\|\uu\|_{L^\infty(0, t_{n+1}, W^1_\infty(\Omega_h))} \left( \|\eei\|^2_{L^2(0, t_{n+1}, L^2(\Omega))} +  2 \|\eea\|^2_{L^2(0, t_{n+1}, L^2(\Omega))} \right) \,.
\end{aligned}
\]
Therefore from \eqref{eq:err-int-2} we infer
\begin{equation}
\label{eq:errc}
\begin{aligned}
\errc &\leq 
\Lambda \left(\tau^{4} + h^{2 \reg + 2} \right)+  
\Lambda \|\eea\|^2_{L^2(0, t_{n+1}, L^2(\Omega))} \,.
\end{aligned}
\end{equation}

\noindent
$\bullet$ estimate of $\errC$. 
By the Cauchy-Schwarz inequality
\begin{equation}
\label{eq:errC-1}
\begin{aligned}
\errC &= \sum_{E \in \Omega_h}
\delta_E \int_0^{t_{n+1}}\int_E \cc((\Gr \uu)(\uu - \uht)) \, \cc(\eea_t + (\Gr \eea) \uht)\, {\rm d}E \, {\rm d}t 
\\
& \leq
 \sum_{E \in \Omega_h} 2{\delta_E}
\int_0^{t_{n+1}} \|\cc((\Gr \uu) (\uu - \uht))\|_E^2 \,{\rm d}t+ \\
& \qquad +
\frac{1}{8} \sum_{E \in \Omega_h} \delta_E 
\|\cc(\eea_t + (\Gr\eea) \uht) \|_{L^2(0, t_{n+1}, L^2(E))}^2  \,.
\end{aligned}
\end{equation}
Let $\zeta_{C,1}$ be the first term in the right-hand side of \eqref{eq:errC-1}. Direct computations, an inverse estimate (c.f. Lemma \ref{lm:inverse}), and bounds \eqref{eq:err-int-2} and \eqref{eq:err-int-3} yield 
\begin{equation}
\label{eq:errC-2}
\begin{aligned}
\zeta_{C,1} & 
\leq
\sum_{E \in \Omega_h} 2{\delta_E}
\int_0^{t_{n+1}}
\left(
|\uu|_{W^2_\infty(E)}^2\|\uu - \uht\|_E^2 + |\uu|_{W^1_\infty(E)}^2\|\Gr(\uu - \uht)\|_E^2 \right)
 \,{\rm d}t
 \\
& \leq
4 \sum_{E \in \Omega_h} \delta_E 
\int_0^{t_{n+1}}
\|\uu\|_{W^2_\infty(E)}^2 \left(
 \|\eei\|_E^2 + \|\eea\|_E^2 + \|\Gr\eei\|_E^2 + \|\Gr\eea\|_E^2 \right)
 \,{\rm d}t
  \\
& \leq
4 \sum_{E \in \Omega_h}  \delta_E 
\int_0^{t_{n+1}}  \|\uu\|_{W^2_\infty(E)}^2 \left(
 \|\eei\|_E^2 +  \|\Gr\eei\|_E^2 + C \,h_E^{-2}\|\eea\|_E^2 \right)
 \,{\rm d}t
\\
& \leq
4 \|\uu\|_{L^\infty(0, t_{n+1}, W^2_\infty(\Omega))}^2
\sum_{E \in \Omega_h}  \delta_E  \left(
 \|\eei\|_{L^2(0, t_{n+1}, L^2(E))}^2 + 
  \|\Gr\eei\|_{L^2(0, t_{n+1}, L^2(E))}^2 \right) +
\\
& \quad + 4C \, 
  \|\uu\|_{L^\infty(0, t_{n+1}, W^2_\infty(\Omega))}^2
\sum_{E \in \Omega_h} \delta_E 
h_E^{-2}\|\eea\|_{L^2(0, t_{n+1}, L^2(E))}^2 
\\
& \leq
\Lambda \delta \left(
 \tau^{4} + h^{2 \reg} \right) + 
\Lambda
\delta 
h^{-2}\|\eea\|_{L^2(0, t_{n+1}, L^2(\Omega))}^2 \,.
\end{aligned}
\end{equation}
%
Therefore from \eqref{eq:errC-1} and \eqref{eq:errC-2} we have
\begin{equation}
\label{eq:errC}
\errC 
\leq
\Lambda \delta \left(
 \tau^{4} + h^{2 \reg} \right) + 
\Lambda
\delta 
h^{-2}\|\eea\|_{L^2(0, t_{n+1}, L^2(\Omega))}^2 
+
\frac{1}{8} \sum_{E \in \Omega_h} \delta_E 
\|\cc(\eea_t + (\Gr\eea) \uht) \|_{L^2(0, t_{n+1}, L^2(E))}^2  .
\end{equation}

\noindent
$\bullet$ estimate of $\errJ$. 
By item $iv)$ in  Assumptions \textbf{(A1)}   {$\Gr \uu(\cdot, t) \in [C^0(\Omega)]^{2 \times 2}$ and $\uu(\cdot, t) - \uht(\cdot, t) \in [C^0(\Omega)]^2$ for all $t \in I$}, therefore we have
\begin{equation}
\label{eq:errJ}
\begin{aligned}
\errJ & = 0 \,.
\end{aligned}
\end{equation}
The thesis follows collecting \eqref{eq:err0}, \eqref{eq:errc}, \eqref{eq:errC}, \eqref{eq:errJ} and recalling definitions \eqref{eq:addedlater} and \eqref{eq:sigmaE}.
\end{proof}

\begin{proposition}
\label{prp:err-B}
The term $\errA$ in \eqref{eq:err-quan} can be bounded as follows
\[
\begin{aligned}
\errA &\leq  
16 \|\eei\|_{\nstabn}^2 +
\Lambda \sigma +
 \frac{1}{8} \|\eea\|_{\nstabn}^2 \,.
\end{aligned}
\]
\end{proposition}

\begin{proof}
. \,
We split $\errA$ as the sum of the following terms
\begin{equation}
\label{eq:err-B-0}
\begin{aligned}
\errA &= 
\int_0^{t_{n+1}} \nu a(\eei, \eea) \,{\rm dt} +
\int_0^{t_{n+1}} \CstabG(\uht, \uht; \,\eei, \eea) \,{\rm dt}  +
\int_0^{t_{n+1}} \JstabG(\uht, \uht; \,\eei, \eea) \,{\rm dt}
\\
&  +
\int_0^{t_{n+1}} \LstabG(\uht; \,\eei, \eea) \,{\rm dt} +
\sum_{j=1}^n  (\jump{\eei}_{t_j}, \eea^j_+) +
\left((\eei^0_+, \eea^0_+) +
\int_0^{t_{n+1}} \left( (\eei_t, \eea)  + 
 c(\uht; \, \eei, \eea) \right) \,{\rm d}t \right)
\\
& =: \sum_{i=1}^6 \alpha_i \,.  
\end{aligned}
\end{equation}
We now estimate each term in the sum above.
Let, for the time being, $\gamma$ be a positive real to be chosen at the end of the proof.

\noindent
$\bullet$ estimate of $\alpha_1$, $\alpha_2$, $\alpha_3$. Cauchy-Schwarz inequality yields
\begin{equation}
\label{eq:err-B-1}
\begin{aligned}
\alpha_1 & 
\leq 
\frac{1}{4\gamma}\|\eei\|_{\nstabn}^2 +   
\gamma \nu\|\epseps(\eea)\|_{L^2(0, t_{n+1}, L^2(\Omega))}^2 
\\
\alpha_2  
&\leq 
\frac{3}{4\gamma}\|\eei\|_{\nstabn}^2  +
\frac{\gamma}{3} 
\sum_{E \in \Omega_h} \delta_E 
\|\cc(\eea_t + (\Gr\eea)\uht)\|_{L^2(0, t_{n+1}, L^2(E))}^2
\\
\alpha_3 &
\leq 
\frac{1}{2\gamma}\|\eei\|_{\nstabn}^2 +   
\frac{\gamma}{2} 
\sum_{E \in \Omega_h} \frac{h_E^2}{2} 
\|\jump{(\Gr \eea) \uht)}\|_{L^2(0, t_{n+1}, L^2(\partial E))}^2 \,.
\end{aligned}
\end{equation}

\noindent
$\bullet$ estimate of $\alpha_4$. Employing again Cauchy-Schwarz inequality, from \eqref{eq:err-int-7} we infer
\begin{equation}
\label{eq:err-B-4}
\begin{aligned}
\alpha_4
&\leq 
\frac{\gamma}{3} 
\sum_{E \in \Omega_h} \delta_E 
\|\cc(\eea_t +(\Gr\eea)\uht)\|_{L^2(0, t_{n+1}, L^2(E))}^2 +
\\
& \qquad \qquad +    
\frac{3}{4\gamma} \sum_{E \in \Omega_h}\delta_E \nu^2 \int_0^{t_{n+1}}\|\cc(\divv(\epseps(\eei)))\|_E^2 \,{\rm d}t
\\
& \leq
\frac{\gamma}{3} 
\sum_{E \in \Omega_h} \delta_E 
\|\cc(\eea_t +(\Gr\eea)\uht)\|_{L^2(0, t_{n+1}, L^2(E))}^2  +    
\frac{3}{4\gamma}  \sum_{E \in \Omega_h}\delta_E \nu^2 \|\eei\|_{L^2(0, t_{n+1}, H^3(E))}^2 
\\
& \leq
\frac{\gamma}{3} 
\sum_{E \in \Omega_h} \delta_E 
\|\cc(\eea_t +(\Gr\eea)\uht)\|_{L^2(0, t_{n+1}, L^2(E))}^2  +    
\frac{3}{4\gamma} \delta \nu^2 \|\eei\|_{L^2(0, t_{n+1}, H^3(\Omega_h))}^2 
\\
& \leq
\frac{\gamma}{3} 
\sum_{E \in \Omega_h} \delta_E 
\|\cc(\eea_t +(\Gr\eea)\uht)\|_{L^2(0, t_{n+1}, L^2(E))}^2  +    
\Lambda \delta \nu^2 \left(\tau^4 + h^{2\reg - 4} \right)  \,.
\end{aligned}
\end{equation}

\noindent
$\bullet$ estimate of $\alpha_5$. Since   $\eei \in C^0(\I, \ZCG)$ we have
\begin{equation}
\label{eq:err-B-5}
\alpha_5 = \sum_{j=1}^n  (\jump{\eei}_{t_j}, \eea^j_+) = 0 \,.
\end{equation}

\noindent
$\bullet$ estimate of $\alpha_6$. 
Integrating by parts in time and recalling \eqref{eq:c-skew} we infer
\begin{equation}
\label{eq:err-B-6-0}
\begin{aligned}
\alpha_6 &=
(\eea^{n+1}_-, \eei^{n+1}_-) -  \sum_{j=1}^n (\jump{\eea}_{t_j}, \eei^j) 
- \int_0^{t_{n+1}} \left( (\eea_t, \eei)  + 
 c(\uht; \, \eea, \eei) \right) \,{\rm d}t  
 \,.
\end{aligned}
\end{equation}
We estimate each term in the sum above.
For the first term, the Cauchy-Schwarz inequality yields
\begin{equation}
\label{eq:err-B-6-1}
(\eea^{n+1}_-, \eei^{n+1}_-) 
\leq \frac{1}{2\gamma} \|\eei\|_{\nstabn}^2 +  
\frac{\gamma}{2}  \|\eea^{n+1}_-\|^2  \,.
\end{equation}
Concerning the second term, recalling the definition of the $L^2$-projection operator $\Pi_h$, we have
\begin{equation}
\label{eq:err-B-6-2}
\sum_{j=1}^n (\jump{\eea}_{t_j}, \eei^j) = 
\sum_{j=1}^n \bigl(\jump{\eea}_{t_j}, \uu(\cdot, t_j) - \Pi_h(\uu(\cdot, t_j)) \bigr) = 
0 \,.
\end{equation}
For the last term, since $\eei \in C^0((0, t_{N+1}), \ZCG)$, from \eqref{eq:exact cont} it exists $\ssi \in C^0((0, t_{N+1}), \PCG)$ such that
$\CC \ssi = -\eei$, satisfying corresponding approximation estimates 
\begin{equation}\label{hottie}
\begin{aligned}
& \|\ssi\|^2_{L^2(0, t_{n+1}, L^2(\Omega))}
\leq
\Lambda (\tau^{4} + h^{2\reg + 4})  \\
& \|\nabla \ssi\|^2_{L^2(0, t_{n+1}, L^2(\Omega))}  \leq
\Lambda (\tau^{4} + h^{2\reg + 2}) \, .
\end{aligned}
\end{equation}
First introducing the above stream function,
then integrating by parts and employing the Cauchy-Schwarz inequality we obtain
\begin{equation}
\label{eq:err-B-6-3}
\begin{aligned}
&\int_0^{t_{n+1}} \left( (\eea_t, \CC\ssi)  + 
 c(\uht; \, \eea, \CC\ssi) \right) \,{\rm d}t 
 =
 \\
& =
\int_0^{t_{n+1}} \sum_{E \in \Omega_h} \biggl(
(\cc\eea_t, \ssi)_E + (\cc \left((\Gr \eea)\uht\right), \ssi)_E -
(\jump{(\Gr \eea)\uht} \cdot\tf^E, \ssi)_{\partial E}
\biggr)\,{\rm d}t 
\\
& = 
\int_0^{t_{n+1}} \sum_{E \in \Omega_h} 
 (\cc \left(\eea_t + (\Gr \eea)\uht\right), \ssi)_E
 \,{\rm d}t 
  -
 \int_0^{t_{n+1}} \sum_{E \in \Omega_h} 
(\jump{(\Gr \eea)\uht} \cdot\tf^E, \ssi)_{\partial E}
\,{\rm d}t 
\\
& \leq
\int_0^{t_{n+1}} \sum_{E \in \Omega_h} \biggl(
\frac{\gamma}{3}\delta_E \|\cc \left(\eea_t + (\Gr \eea)\uht\right)\|_E^2 + \frac{3}{4\delta_E \gamma} \|\ssi\|_E^2 
\biggr)\,{\rm d}t +
\\
& \qquad +
\int_0^{t_{n+1}} \sum_{E \in \Omega_h} \biggl(
\frac{\gamma}{2} \frac{h_E^2}{2}\|\jump{(\Gr \eea)\uht}\|_{L^2(\partial E)}^2 + \frac{1}{h_E^2 \gamma} \|\ssi\|_{\partial E}^2
\biggr)\,{\rm d}t 
\\ 
& \leq
\frac{\gamma}{3} 
\sum_{E \in \Omega_h} \delta_E 
\|\cc(\eea_t +(\Gr\eea)\uht)\|_{L^2(0, t_{n+1}, E)}^2  +
\frac{\gamma}{2} 
\sum_{E \in \Omega_h} \frac{h_E^2}{2} 
\|\jump{(\Gr \eea) \uht)}\|_{L^2(0, t_{n+1}, L^2(\partial E))}^2 + 
\\
& \qquad +
\sum_{E \in \Omega_h}
\frac{3}{4\delta_E \gamma}
\|\ssi\|_{L^2(0, t_{n+1}, L^2(E))}^2 +
\int_0^{t_{n+1}} \sum_{E \in \Omega_h}  \frac{1}{h_E^2 \gamma} \|\ssi\|^2_{\partial E}
\,{\rm d}t 
\,.
\end{aligned}
\end{equation}
Employing \eqref{hottie} we have
\begin{equation}
\label{eq:err-B-6-4}
\begin{aligned}
\sum_{E \in \Omega_h}
\frac{3}{4\delta_E \gamma}
\|\ssi\|_{L^2(0, t_{n+1}, L^2(E))}^2
& \leq 
C \delta^{-1}
\|\ssi\|_{L^2(0, t_{n+1}, L^2(\Omega))}^2 
\leq  
\Lambda \delta^{-1}  (\tau^{4} + h^{2\reg + 4}) \,.
\end{aligned} 
\end{equation}
From the trace inequality \ref{lm:trace} and employing again \eqref{hottie}
\begin{equation}
\label{eq:err-B-6-5}
\begin{aligned}
\int_0^{t_{n+1}} \sum_{E \in \Omega_h}  \frac{1}{h_E^2 \gamma} \|\ssi\|_{\partial E}^2 
&  \leq C \,
\int_0^{t_{n+1}}  \! \sum_{E \in \Omega_h}  \left(
h_E^{-3}  \|\ssi\|_E^2 +
h_E^{-1}  \|\nabla \ssi\|_E^2 \right) \,{\rm d}t
\\
&  \leq C \, 
h^{-3}   \|\ssi\|^2_{L^2(0, t_{n+1}, L^2(\Omega))} + C \,
h^{-1}  \|\nabla \ssi\|^2_{L^2(0, t_{n+1}, L^2(\Omega))} 
\\
& \leq
\Lambda \,   h^{-3}  (\tau^{4} + h^{2\reg + 4}) \,.
\end{aligned} 
\end{equation}
Therefore from \eqref{eq:err-B-6-0}--\eqref{eq:err-B-6-5} we obtain
\begin{equation}
\label{eq:err-B-6}
\begin{aligned}
\alpha_6 &\leq  
\frac{\gamma}{3} 
\sum_{E \in \Omega_h} \delta_E 
\|\cc(\eea_t +(\Gr\eea)\uht)\|_{L^2(0, t_{n+1}, E)}^2 +
\frac{\gamma}{2} 
\sum_{E \in \Omega_h} \frac{h_E^2}{2} 
\|\jump{(\Gr \eea) \uht)}\|_{L^2(0, t_{n+1}, L^2(\partial E))}^2 +
\\
& \qquad \qquad + 
\frac{\gamma}{2}  \|\eea^{n+1}_-\|^2 +
 \frac{1}{2\gamma} \|\eei\|_{\nstabn}^2 +
\Lambda \,   (\delta^{-1} +  h^{-3}) (\tau^{4} + h^{2\reg + 4}) \,.
\end{aligned}
\end{equation}
The proof follows collecting \eqref{eq:err-B-1}, \eqref{eq:err-B-4}, \eqref{eq:err-B-5} and \eqref{eq:err-B-6} in \eqref{eq:err-B-0} and by choosing $\gamma=1/8$.

\end{proof}


We are now ready to state the following convergence result.
\begin{proposition}
\label{prp:error-eea}
Let $\uu \in C^0((0,t_{N+1}), \ZCG)$ and $\uht \in \ZDGt$ be the solutions of problem \eqref{eq:ns variazionale ker} and problem \eqref{eq:ns thed}, respectively.
Under the mesh Assumptions \cfan{(M1)}, \cfan{(M2)} and regularity Assumptions  \cfan{(A2)},
 if the family of  parameters $\{\delta_E\}_{E \in \Omega_h}$ satisfies \eqref{eq:delta_E} and $\tau \lesssim 1$ it holds that

\begin{equation}
\label{eq:eea}
\|\uu - \uht\|^2_{\nstabn} \lesssim
\sigma \,.
\end{equation}
\end{proposition}

\begin{proof}
. \,
From  Proposition \ref{prp:error}, Proposition \ref{prp:err-A} and Proposition \ref{prp:err-B} we have
\[
\frac{1}{4} \|\eea\|^2_{\nstabn}  \leq
\Lambda \sigma + 
\Lambda (1 + \delta h^{-2}) \|\eea\|^2_{L^2(0, t_{n+1}, L^2(E))} +
16 \|\eei\|^2_{\nstabn} \,.
\]
By definition \eqref{eq:delta} the parameter  $\delta \lesssim h^3$, then using 
Proposition \ref{prp:epi} and recalling that $\Lambda$ and $\texttt{DATA}(n)$ are independent of $h$, $\tau$, $\nu$
\[
\begin{aligned}
\|\eea\|^2_{\nstabn} &\lesssim
\sigma + \|\eea\|^2_{L^2(0, t_{n+1}, L^2(\Omega))}
\,.
\end{aligned}
\]
By Lemma \ref{lm:1} and since $\tau \lesssim 1$ 
\[
\begin{aligned}
\|\eea\|^2_{\nstabn}
& \lesssim		
\sigma + 
\tau \sum_{j=1}^{n-1}\|\eea\|_{\nstabj}^2
 \,.
\end{aligned}
\]
Then employing the discrete Gronwall inequality (Proposition \ref{prp:gronwall})
\begin{equation}
\label{eq:eea-finale}
\begin{aligned}
\|\eea\|^2_{\nstabn} 
& \lesssim		
\sigma + 
\sum_{j=1}^{n-1}
\sigma \tau
 \lesssim		
\sigma  \,.
\end{aligned}
\end{equation}
The thesis follows from Proposition \ref{prp:epi}, bound \eqref{eq:eea-finale} and the triangular inequality.
\end{proof}

Let us analyse the asymptotic order of convergence for the proposed method both in the convective and in the diffusion dominated regime. Let us set
\begin{equation}
\label{eq:tau-h}
\begin{aligned}
&\tau \lesssim h^{(\reg+2)/2} \quad &\text{if $\reg \leq 4$,}
\\
&\tau \lesssim h^{\reg-1} \quad &\text{if $\reg > 4$.}
\end{aligned}
\end{equation}
In accordance with \eqref{eq:delta} and assuming maximal
regularity, so that $\reg = k$, we have the following cases

\begin{itemize}
\item \texttt{convection dominated regime:} $\nu \ll h$ then $\delta \approx h^3$ 
\begin{equation}
\label{eq:sigmaE-conv}
\sigma := \max\{ 
h^{3}\tau^{2}, \, 
h^{-3}\tau^{4}, \,
h^{2 k+1}
 \} 
\end{equation}
therefore
\begin{equation}
\label{eq:error-conv}
\|\uu - \uht\|^2_{\nstabn} \lesssim h^{2k+1} \,.
\end{equation}

\item \texttt{diffusion dominated regime:} $h \lesssim \nu$ then $\delta \approx \frac{h^4}{\nu}$
\begin{equation}
\label{eq:sigmaE-diff}
\sigma := \max\{
\nu h^{-4}\tau^{4}, \,
\nu h^2 \tau^{2}, \,
\nu h^{2 k}
 \} 
\end{equation}
therefore
\begin{equation}
\label{eq:error-diff}
\|\uu - \uht\|^2_{\nstabn} \lesssim \nu h^{2k} \,.
\end{equation}
\end{itemize}
We conclude that the scheme yields an optimal rate of convergence for both regimes. 
In particular, the convection quasi-robustness is reflected by the fact that (1) the constants in the estimate are independent of $\nu$, (2) the convergence rate gains an additional $h^{1/2}$ factor in convection dominated cases, (3) the convergence is in a norm that contains terms controlling explicitly the convection. A list of remarks are in order.

\begin{rmk}[Pressure robustness] 
\label{rem:press-rob}
As already observed, the proposed scheme is pressure robust in the sense of \cite{john-linke-merdon-neilan-rebholz:2017} (a modification of the continuous problem that only affects the pressure leads to changes in the discrete solution that only affect the discrete pressure). 
The estimate in Proposition \ref{prp:error-eea} reflects such property of the scheme, since the velocity error estimates are independent of the continuous pressure $p$. Analogously, note that in all our analysis (such as Lemma \ref{lm:2} and the definition of $\cfun{DATA}(n)$ in Proposition \ref{prp:stability}) the loading term $\ff$ could be substituted by its Helmholtz-Hodge projection $\mathcal{H}(\ff)$, defined in Section \ref{sec:cont}.
\end{rmk}

\begin{rmk}[Avoiding quasi-uniformity assumption] 
\label{rem:no-quasi}
The only reason for the quasi-uniformity assumption {\bf (M2)} is term \eqref{eq:err-B-6-2}. Indeed, such term vanishes only because we have chosen the approximation operator $\Pi_h$ in the definition of ${\mathcal P}$; in turn, using a global $L^2$ projector such as $\Pi_h$ yields a quasi-uniformity condition, see Lemma \ref{lm:proj-bis}. An alternative approach would be to use a local approximant, such as that appearing in \eqref{eq:int est 2}, instead of $\Pi_h$. This choice would allow to avoid assumption {\bf (M2)}, but the price to pay would be the following. 
We would be forced to bound the term \eqref{eq:err-B-6-2} as follows
$$
\sum_{j=1}^n (\jump{\eea}_{t_j}, \eei^j)  \leq \|\eea\|_{\nstabn}  \biggl( \sum_{j=1}^n \| \eei^j \|_{\Omega}^2 \biggr)^{1/2}
\lesssim \Lambda \, n^{1/2} \, h^{s+1} \|\eea\|_{\nstabn} \, .
$$
Therefore, since $N \sim \tau^{-1}$, such term would lead to an additional addendum of order $O(\tau^{-1} h^{2s+2})$ in the estimates. 
Finally, note that one could also make use of $\Pi_h$ but resort to a  local quasi-uniformity condition, in the spirit of \cite{CT:87} (but such choice would yield an even more technical analysis). 
\end{rmk}

\begin{rmk}[Error estimates for the pressure]
Since estimating the pressure error is not the focus or novelty of this contribution, we limit ourselves in reporting briefly a simple approach.
Starting from the inf-sup stability of the discrete couple $(\VDG,\QDG)$ stated in {\bf (P1)}, by standard techniques we can easily derive a discrete space-time inf-sup condition involving the $L^2(0,T, L^2(\Omega))$ norm for the pressure and the $L^2(0,T, H^1(\Omega))$ for the velocity.  
One then combines such inf-sup condition (applied to the difference among the discrete pressure $P$ and a suitable projection in $\QDGt$ of the continuous pressure $p$) with the discrete and continuous error equations, and finally makes use of the developed bounds for the velocity error \eqref{eq:eea}. 
After a triangle inequality, these steps lead to an estimate with many terms (depending on $h,\tau,\delta,\nu,\sigma$); in order to simplify the bound we here assume the reasonable situation in which $\nu \lesssim 1$ and $\tau \lesssim h$. 
After manipulations, we finally obtain the following estimate
$$
\| p - P \|_{L^2(0,T;L^2(\Omega))} \lesssim 
\Big(h^{-1/2} + \min{ \{\nu^{1/2} + h, \nu^{-1/2} h \}\tau^{-1} } \Big) \, \sigma^{1/2} \, ,
$$
where $\sigma$ was defined in \eqref{eq:sigmaE} (see also \eqref{eq:eea} to compare with the velocity error estimates). 
Note that the $\tau^{-1}$ terms are related to the approximation of the time-derivatives of the velocity error, and already present in previous literature \cite{Ahmed-Stokes}. In order to eliminate such terms, one should either first (try to) derive direct error estimates for the time derivative of the velocities, or resort to weaker time norms in the spirit of \cite{Hm1}.
\end{rmk}

\section{Numerical tests}
\label{sec:num}

In this section we present three numerical experiments to assess the practical performances of the proposed stabilized scheme \eqref{eq:ns stab}. 
In the first test we exhibit that the method provides the expected optimal convergence order both in the convection and in the diffusion dominated regime (cf. \eqref{eq:error-conv} and \eqref{eq:error-diff} respectively). 
In the second test we show that the stabilization strategy here proposed does not spoil the benefits related to the divergence-free property of the discrete solution. 
In the third test we validate the proposed method on a standard benchmark problem (the lattice vortex problem).

In the tests we consider a classical family of discrete spaces \eqref{eq:spaces} verifying properties \textbf{(P1)}--\textbf{(P6)}  introduced in Section \ref{sub:fem} that is  the second order Scott-Vogelius element 
\[
\VDG := [\Pk_2(\Omega_h)]^2 \cap \VCG \,,
\qquad 
\QDG := \Pk_1(\Omega_h) \cap \QCG\,,
\] 
on barycenter-refined meshes \cite{john-linke-merdon-neilan-rebholz:2017} (see Fig. \ref{fig:mesh-table} (left) for a sample of barycentric refined triangulations).
The associated $H^2$-conforming space $\PDG$ is the Hsieh-Clough-Tocher element \cite{HCT}.



\begin{figure}
     \begin{subfigure}[t]{0.5\textwidth}
     \hspace{-0cm}
         \includegraphics[width=0.8\textwidth]{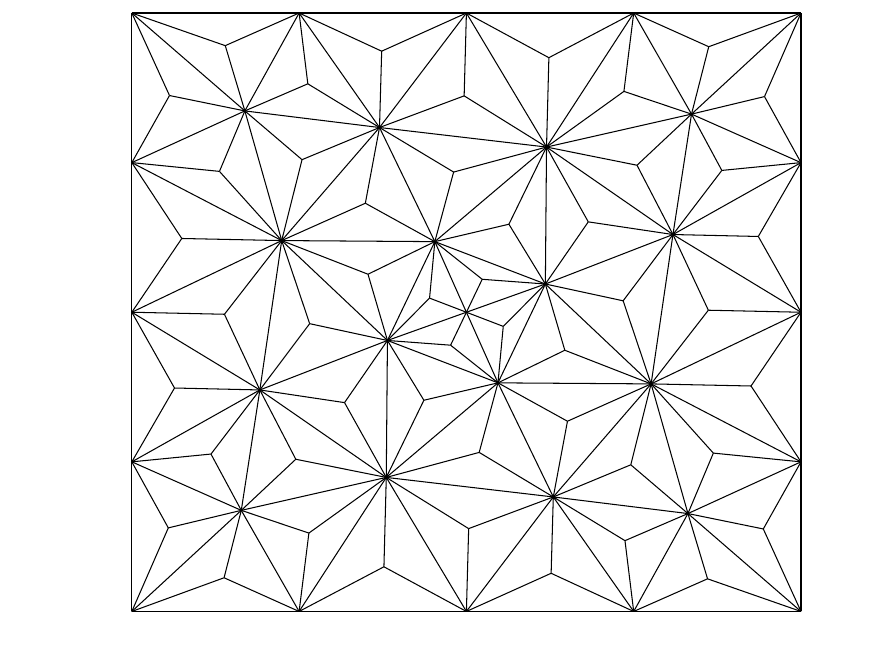}
     \label{fig:mesh}
     \end{subfigure}
     \hfill
     \begin{subfigure}[t]{0.5\textwidth}
     \vspace{-3.0cm}
	\hspace{-0cm}
         \begin{tabular}{|c|c|c|}
\cline{2-3}
\multicolumn{1}{c|}{}&$h$ &$\tau$
\\
\hline
\texttt{mesh 1}    &\texttt{3.809e-01} & \texttt{0.500e-00} \\
\texttt{mesh 2}    &\texttt{2.204e-01} & \texttt{0.125e-00} \\
\texttt{mesh 3}    &\texttt{1.055e-01} & \texttt{3.125e-01} \\
\texttt{mesh 4}    &\texttt{5.187e-02} & \texttt{7.812e-03} \\
\hline
\end{tabular}
         \label{fig:table}
     \end{subfigure}
\caption{Barycenter-refined \texttt{mesh 2} (left).  Mesh size $h$ and associated time step size $\tau$ for all refinement levels (right).}
\label{fig:mesh-table}
\end{figure}

 For the tests in Subsection \ref{sub:numConv} and
Subsection \ref{sub:numDivFree} we consider the following setting.
The domain is $\Omega :=[0,\,1]^2$ and the final time is $T:=1$.
The domain $\Omega$ is partitioned with a sequence of barycenter-refined meshes  based on a
sequence of conforming Delaunay triangulations.
We consider four levels of refinements: \texttt{mesh 1},
\texttt{mesh 2} (depicted in Fig. \ref{fig:mesh-table} (left)), \texttt{mesh 3} and \texttt{mesh 4}.
In accordance with \eqref{eq:tau-h}, in order to recover the optimal rate of convergence both in the convective and in the diffusion dominated regime, the best choice for the time step size is $\tau \approx \gamma h^2$ where $\gamma$ is a positive uniform constant. 
In Fig.~\ref{fig:mesh-table} (right) we collect both $h$ and $\tau$ for each level of refinement. 
Notice that the mesh size $h$ (approximatively) halves at each refinement step, whereas the time step $\tau$ is divided by four.

For the three tests the stabilization parameters $\{\delta_E\}_{E \in \Omega_h}$ are chosen in accordance with \eqref{eq:delta} where we set $c_E = \texttt{1e-2}$ for all $E\in \Omega_h$ and  $\widehat{c}=\texttt{1e-2}$ (cf. \eqref{eq:jstab}). 
Concerning the choice of the stabilization parameter, we address  the influence of $c_E$ and $\widehat{c}$ on the proposed method at the end of Section \ref{sub:numConv}.
As usual, problem \eqref{eq:ns stab} is calculated time-slab by time-slab (first the discrete solution in $I_0$ is found, then $I_1, I_2,...$). For each time slab $I_n$, $n=0,1,..N$, the nonlinear problem \eqref{eq:ns stab} is solved using a fixed point strategy with initial velocity taken as the velocity solution at the end of the previous time-slab  prolonged constant in time in the interval $I_n$, i.e. $\uht^{\rm start}_{|I_n} \equiv \uht^n_- $; we recall that the value $\uht^0_- = \uht_0$ is taken as a suitable interpolation of the initial data $\uu_0$.



In order to evaluate the error between the exact velocity solution $\uu_{\rm  ex}$ and the discrete velocity solution $\uht$,
we consider the $H^1$-seminorm and $L^2$-norm error at final time $T$:
\[
\text{err}(\uht, L^2) := \Vert \uu_{\rm  ex}(\cdot, T) - \uht(\cdot, T) \Vert_{L^2(\Omega)} \,,
\qquad
\text{err}(\uht, H^1) := \Vert \Gr\uu_{\rm  ex}(\cdot, T) - \Gr\uht(\cdot, T) \Vert_{L^2(\Omega)} \,.
\]
Analogously for the pressure variable we consider the error quantity
\[
\text{err}(\pht, L^2) := \Vert p_{\rm ex}(\cdot, T) - \pht(\cdot, T) \Vert_{L^2(\Omega)} \,.
\]
Notice that the convergence results in Section \ref{sub:err} yield error bounds in the space-time norm $\| \cdot \|_{\nstabn}$ and therefore do not provide explicit bounds for all the aforementioned error quantities; we here below summarize the theoretical expectations on the above errors. 

Most importantly, recalling definition \eqref{eq:norma} and bounds \eqref{eq:error-conv} and \eqref{eq:error-diff}, we have
\begin{equation}
\label{eq:bound-sv2}
\begin{aligned}
& \bullet \,\, \texttt{convection dominated regime} 
 &\qquad 
&\text{err}(\uht, L^2) \lesssim C(\uu_{\rm ex}) \, \, h^{2.5}\,,
\\
& \bullet \,\, \texttt{diffusion dominated regime}    &\qquad
&\text{err}(\uht, L^2) \lesssim C(\uu_{\rm ex}) \, \, \sqrt{\nu} h^{2}\,.
\end{aligned}
\end{equation}
In bounds \eqref{eq:bound-sv1} and \eqref{eq:bound-sv2} the hidden constants depend only on the mesh regularity constants $c_{{\rm M}1}$, $c_{{\rm M}2}$ in Assumptions \textbf{(M1)}, \textbf{(M2)}. The quantity $C(\uu_{\rm ex})$ is a generic constant depending on the norms of $\uu_{\rm ex}$ in the spaces listed in Assumptions  \textbf{(A2)} with $k=2$. Due to bounds \eqref{eq:bound-sv2}, the error quantity $\text{err}(\uht, L^2)$ will become our main reference in the following tests.

The other two norms above are too strong to be controlled by the $\| \cdot \|^2_{\nstabn}$ norm used in the theoretical convergence estimates,  but we include them in our numerical investigation since it is informative to investigate its behavior. What we can surely assert is that the best rate of convergence that one can expect in the above norms corresponds to the one given by the interpolation estimates.
For the adopted Scott-Vogelius element, according with \eqref{eq:int est 2-a} and \eqref{eq:int est 3}, such \emph{optimistic} standpoint would lead to the bounds
\begin{equation}
\label{eq:bound-sv1}
\text{err}(\uht, L^2) \lesssim h^3 |\uu_{\rm ex}|_{3, \Omega_h}\,,
\quad
\text{err}(\uht, H^1) \lesssim h^2 |\uu_{\rm ex}|_{3, \Omega_h}\,,
\quad
\text{err}(\pht, L^2) \lesssim h^2 |p_{\rm ex}|_{2, \Omega_h}\, ,
\end{equation}
where we are clearly assuming that the solution is sufficiently regular.

In the forthcoming tests we compare the performances of the proposed stabilized scheme \eqref{eq:ns stab} (labelled as \texttt{stab}) with those of the non-stabilized scheme \eqref{eq:ns fem} (labelled as \texttt{no\,stab}).

\subsection{Convergence analysis}
\label{sub:numConv}

The aim of the present test is to assess the convergence of the stabilized scheme \eqref{eq:ns stab} both in the convection and diffusion dominated regime.

We consider a family of unsteady Navier–Stokes equations \eqref{eq:ns primale}, one per each choice of the viscosity $\nu$, where the load $\ff$ (which turns out to depend on $\nu$), the initial datum $\uu_0$ and the Dirichlet boundary conditions are chosen in accordance with the analytical solution
\[
\begin{aligned}
\uu_{\rm  ex}(x, \,y, \, t) &= 0.5 \begin{bmatrix}
                      -\cos^2(\pi(x-0.5)) \, \cos(\pi(y-0.5))\, \sin(\pi(y-0.5))\\
           \phantom{-}\cos^2(\pi(y-0.5)) \, \cos(\pi(x-0.5)) \, \sin(\pi(x-0.5))
            \end{bmatrix} \cos(t)\,,
\\
p_{\rm  ex}(x, \,y, \, t) &=(\sin(\pi(x-0.5)) - \sin(\pi(y-0.5)))\cos(t)\,.
\end{aligned}
\]
Since we are interested in a robustness analysis with respect to the viscosity, we set $\nu= \texttt{1}$ (\texttt{diffusion dominated regime}) and $\nu= \texttt{1e-11}$ (\texttt{convection dominated regime}).

In Fig.~\ref{fig:convStand} we consider the \texttt{diffusion dominated} case  $\nu= \texttt{1}$ and we display the errors $\text{err}(\uht, H^1)$, $\text{err}(\uht, L^2)$ and $\text{err}(\pht, L^2)$ obtained with the \texttt{stab} and the \texttt{no\,stab} method with the sequence of meshes and time steps of Fig.~\ref{fig:mesh-table} (right).
We notice that both methods provide the optimal rate of convergence 
(see the optimistic bounds \eqref{eq:bound-sv1}) that is, as could be expected for diffusion dominated problems,
the rate of convergence is dictated by the interpolation errors.
It is important to note that the stabilization strategy here proposed does not introduce any suboptimal term in the \texttt{diffusion dominated regime}; 
the convergence lines for the two methods are really close to each other. 
%

\begin{figure}[!htb]
\centering
\begin{tabular}{cc}
\includegraphics[width=0.46\textwidth]{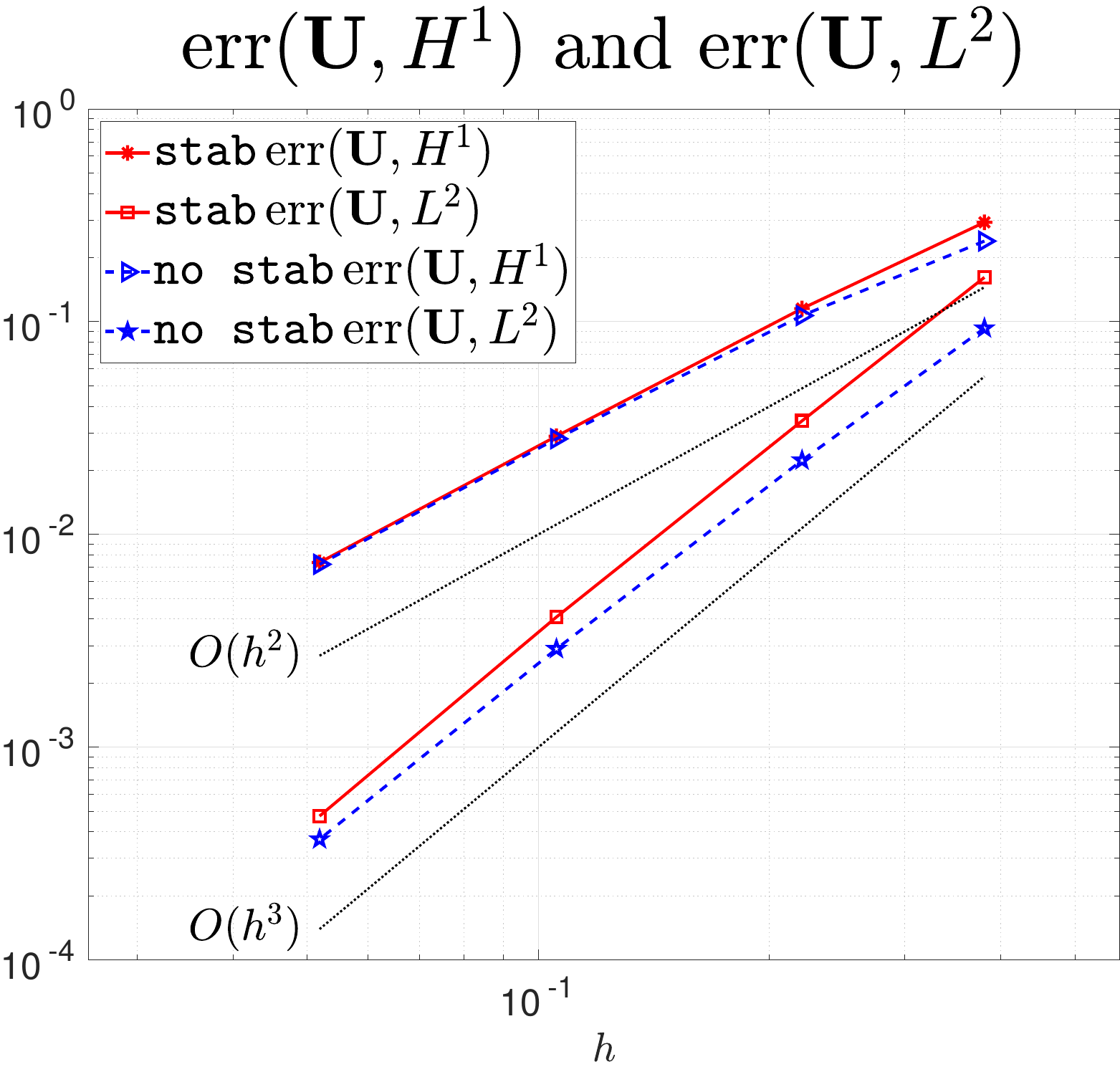}&
\includegraphics[width=0.46\textwidth]{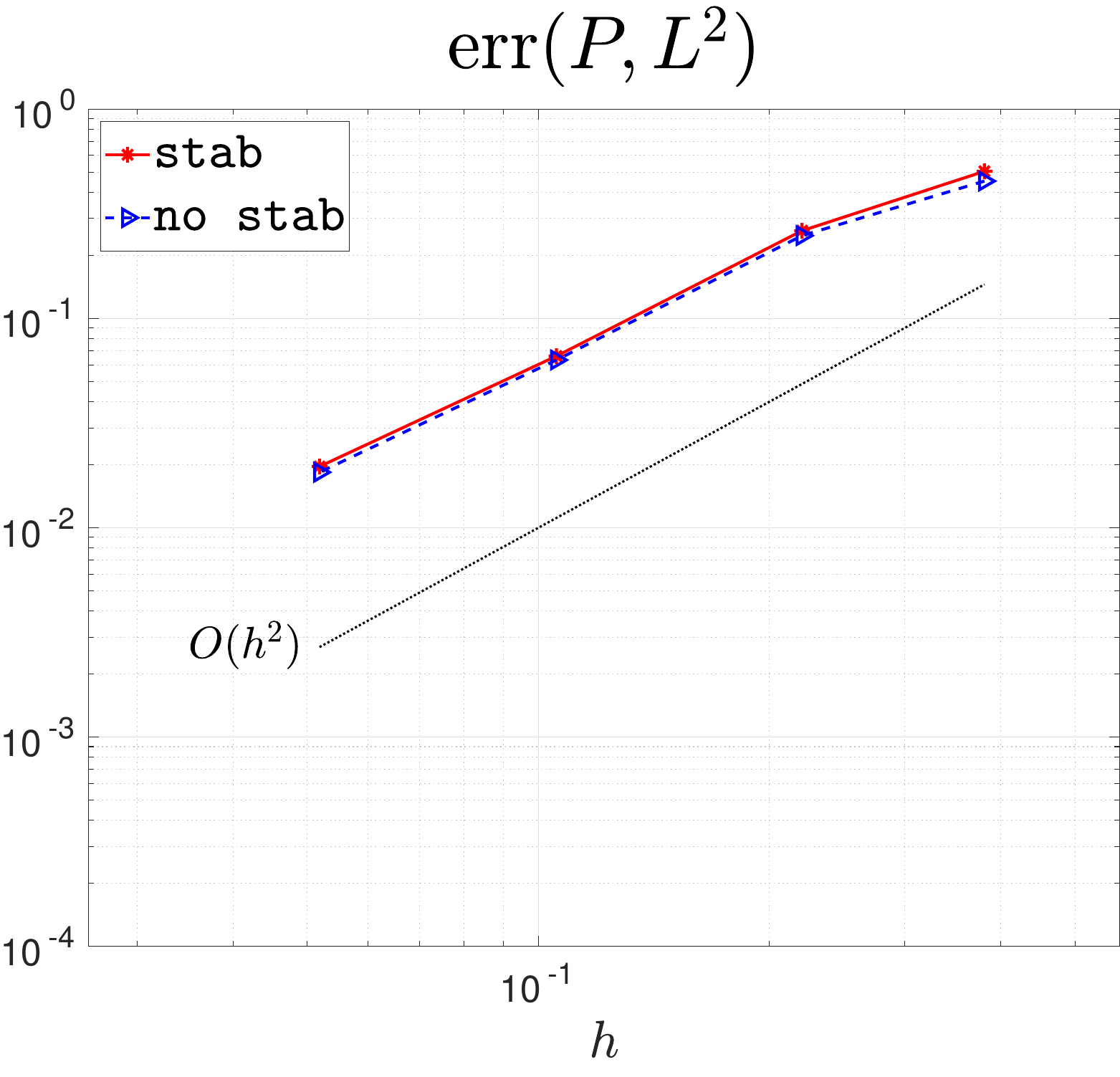}
\end{tabular}
\caption{Convergence analysis, $\nu = \texttt{1}$ (\texttt{diffusion dominated regime}). Convergence histories of both $\text{err}(\uht, H^1)$ and $\text{err}(\uht, L^2)$, left, and $\text{err}(\pht, L^2)$ right.}
\label{fig:convStand}
\end{figure}


We now consider the \texttt{convection dominated regime} $\nu = \texttt{1e-11}$.
In Fig.~\ref{fig:conv11} we show the errors $\text{err}(\uht, L^2)$ and $\text{err}(\pht, L^2)$ comparing the \texttt{stab} and the \texttt{no\,stab} strategy.
In this setting the benefits provided by the \texttt{stab} method are quite evident (at least for the velocity error).
For the \texttt{stab} method we recover for $\text{err}(\uht, L^2)$ the theoretical bound \eqref{eq:bound-sv2}. 
We stress that in \texttt{convection dominated} cases, the rate of convergence $h^{2.5}$ is known to be the best that can be expected.
 Whereas the \texttt{no\,stab} method, in accordance with the analysis in \cite[Section 4.2.1]{garcia} provides a slower convergence rate, that is approximately $h^{2}$.
Concerning the pressure error $\text{err}(\pht, L^2)$ both methods exhibit the optimal (w.r.t. the interpolation error) rate of convergence $h^2$.

\begin{figure}[!htb]
\centering
\begin{tabular}{cc}
\includegraphics[width=0.46\textwidth]{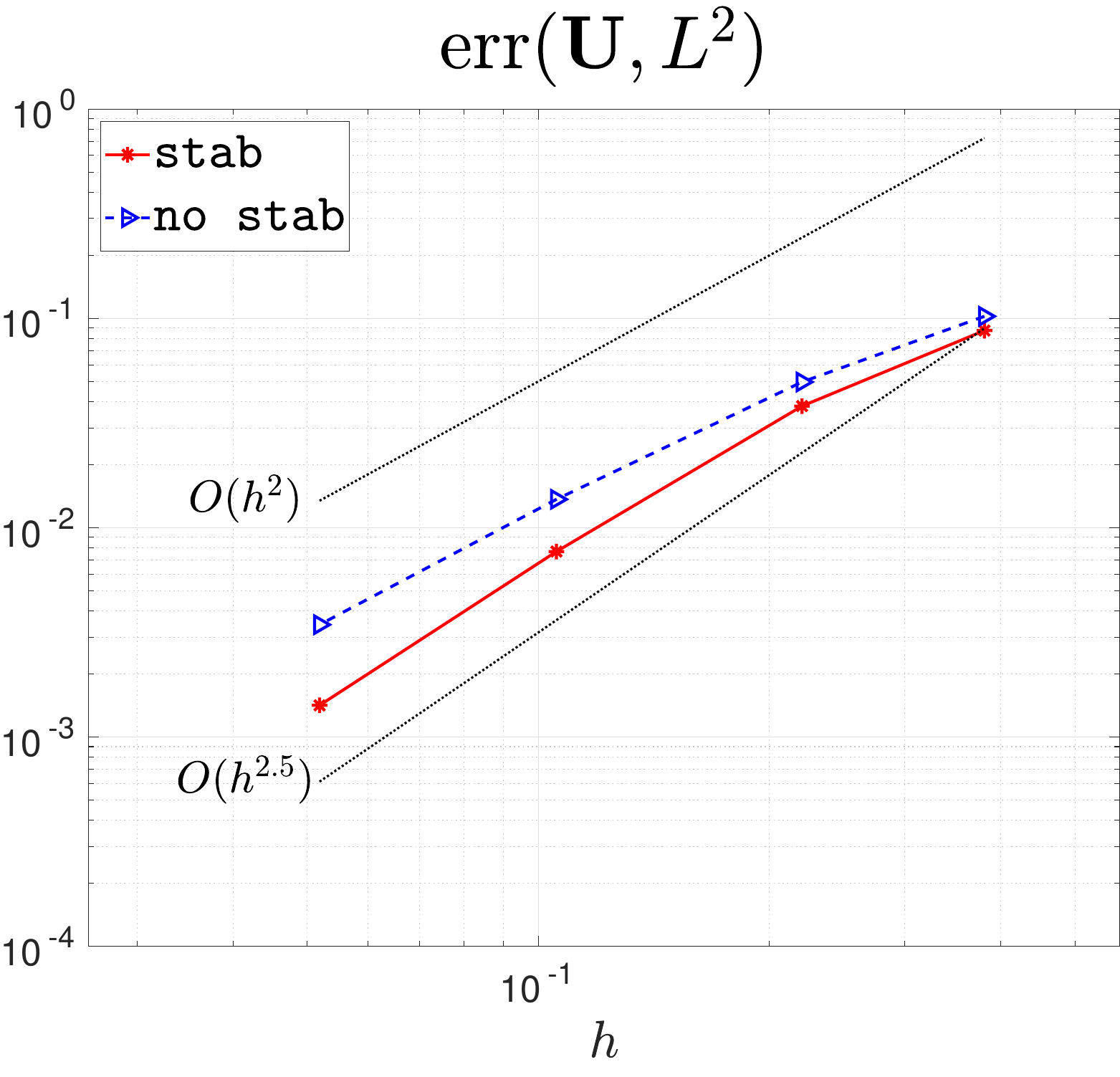}&
\includegraphics[width=0.46\textwidth]{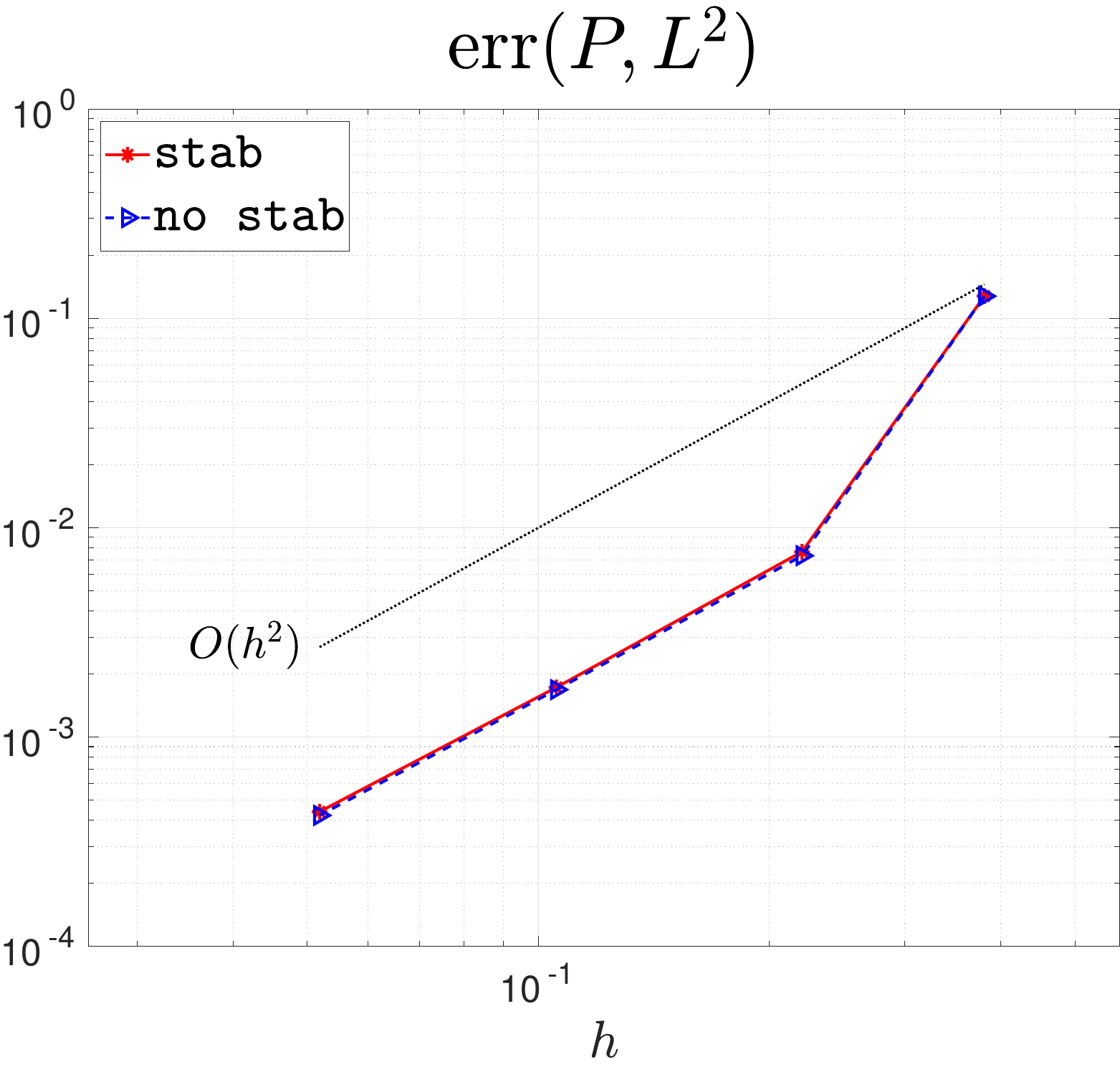}
\end{tabular}
\caption{Convergence analysis, $\nu = \texttt{1e-11}$ (\texttt{diffusion dominated regime}). Convergence histories of $\text{err}(\uht, L^2)$ and $\text{err}(\pht, L^2)$.}
\label{fig:conv11}
\end{figure}

Finally, we highlight the robustness of the proposed method with respect to dominant convection (i.e. small values of $\nu$).
More specifically in Fig.~\ref{fig:convNu} we display
the errors $\text{err}(\uht, L^2)$ and $\text{err}(\pht, L^2)$ for a sequence of decreasing values of $\nu$ (from \texttt{1} to \texttt{1e-11}) on the \texttt{mesh 3} with the associated time step $\tau$ (cf. Fig. \ref{fig:mesh-table} (right)).

\begin{figure}[!htb]
\centering
\begin{tabular}{cc}
\includegraphics[width=0.46\textwidth]{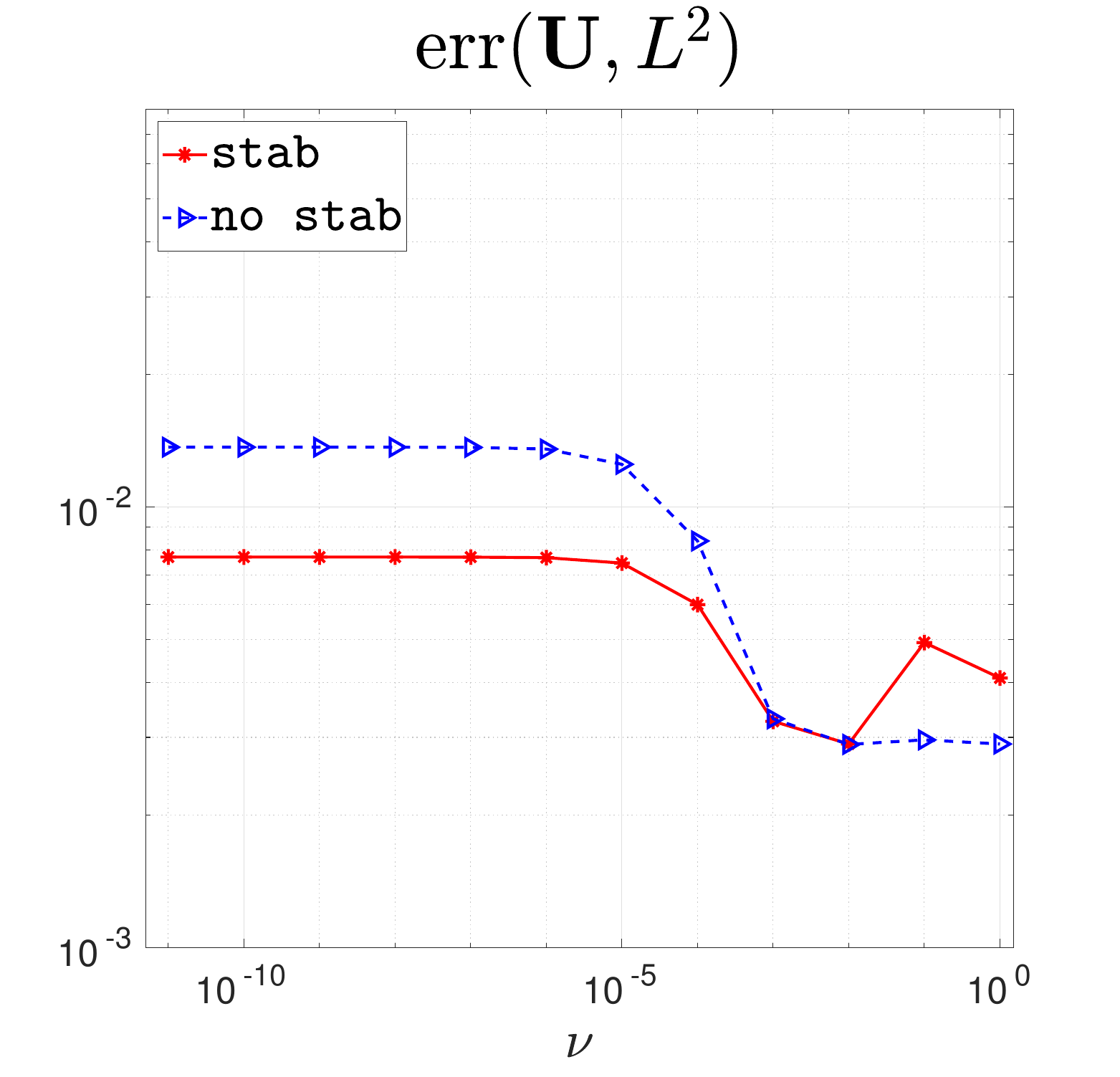}&
\includegraphics[width=0.46\textwidth]{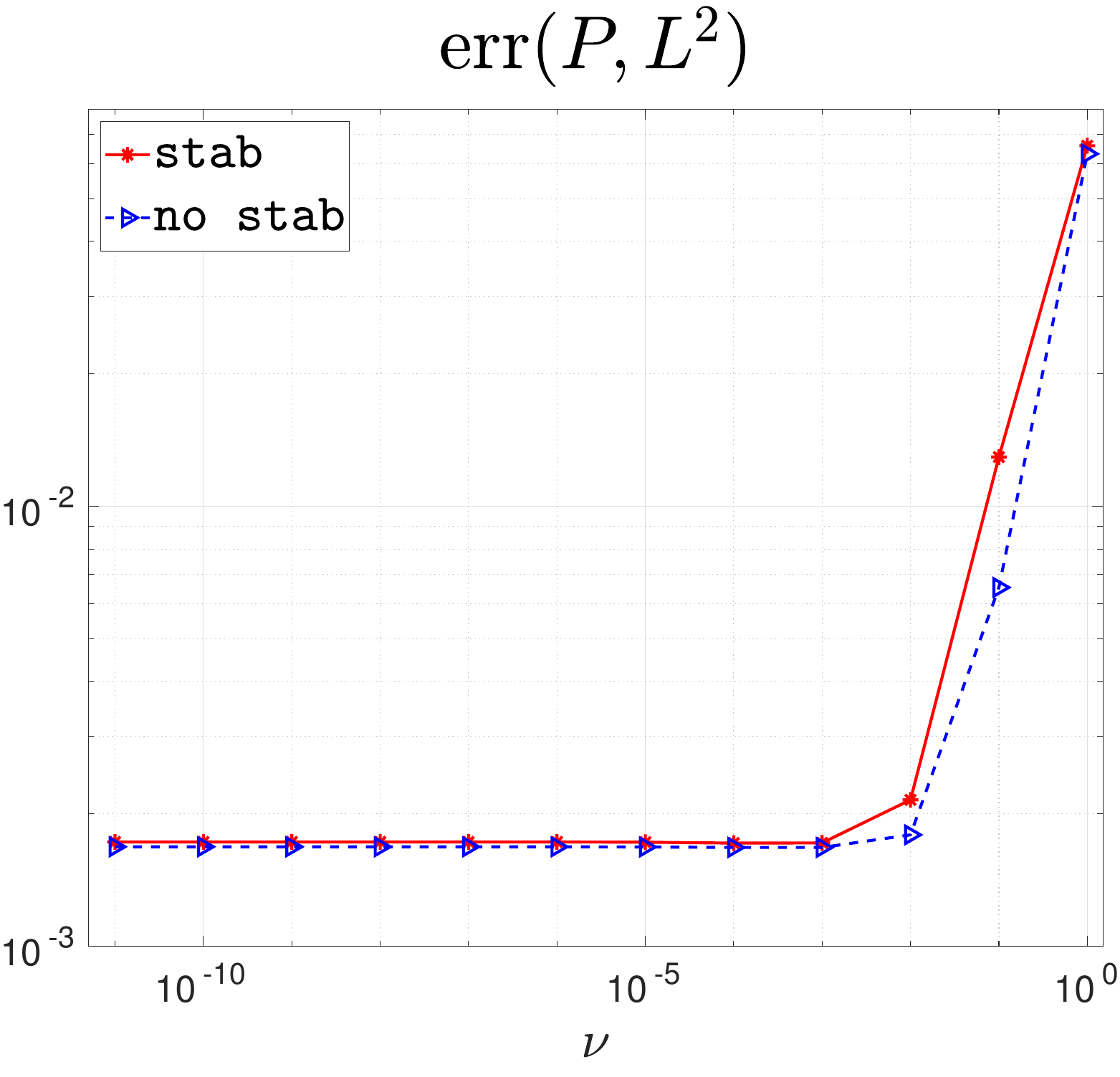}
\end{tabular}
\caption{Convergence analysis: errors $\text{err}(\uht, L^2)$ and $\text{err}(\pht, L^2)$  as a function of  the viscosity $\nu$ on  \texttt{mesh 3}.}
\label{fig:convNu}
\end{figure}

Concerning the error  $\text{err}(\uht, L^2)$, we notice that up to the threshold value $\nu = \texttt{1e-03}$ the problem appears  \texttt{diffusion dominated} and the \texttt{no\,stab} scheme provides better results. Whereas for smaller values of the viscosity
the \texttt{stab} scheme  is more accurate. 
However for $\nu \leq \texttt{1e-5}$ both methods exhibit a plateau of the errors.
For the \texttt{stab} scheme this is in accordance with \eqref{eq:bound-sv2} (\texttt{convection dominated regime}), whereas for the \texttt{non stab} scheme this behaviour can be justified with the presence of the $L^2$-form arising from the time-derivative that introduces a stabilizing reaction-type term in the problem.
We finally note that the errors $\text{err}(\pht, L^2)$ are almost identical for both methods.

We close this section with a sensitivity analysis on the parameter $c_E$ and $\widehat{c}$ (cf. \eqref{eq:delta} and \eqref{eq:jstab}), outlined in Fig.~\ref{fig:CAnalysis}.
We consider $c_E=\widehat{c}$ for all $E\in \Omega_h$ and we vary $c_E$ in the range $[\texttt{1e-4}, \texttt{1}]$, plotting the error for the second mesh of our family (similar results where obtained for the other meshes). 
The plot indicates that the proposed method, particularly in the  \texttt{convection dominated regime}, is fairly robust with respect to the stabilization parameters.
We finally notice that the adopted choice $c_E = \texttt{1e-2}$ appears to accommodate in a satisfactory way to both regimes.

\begin{figure}[!htb]
\centering
\begin{tabular}{cc}
\includegraphics[width=0.46\textwidth]{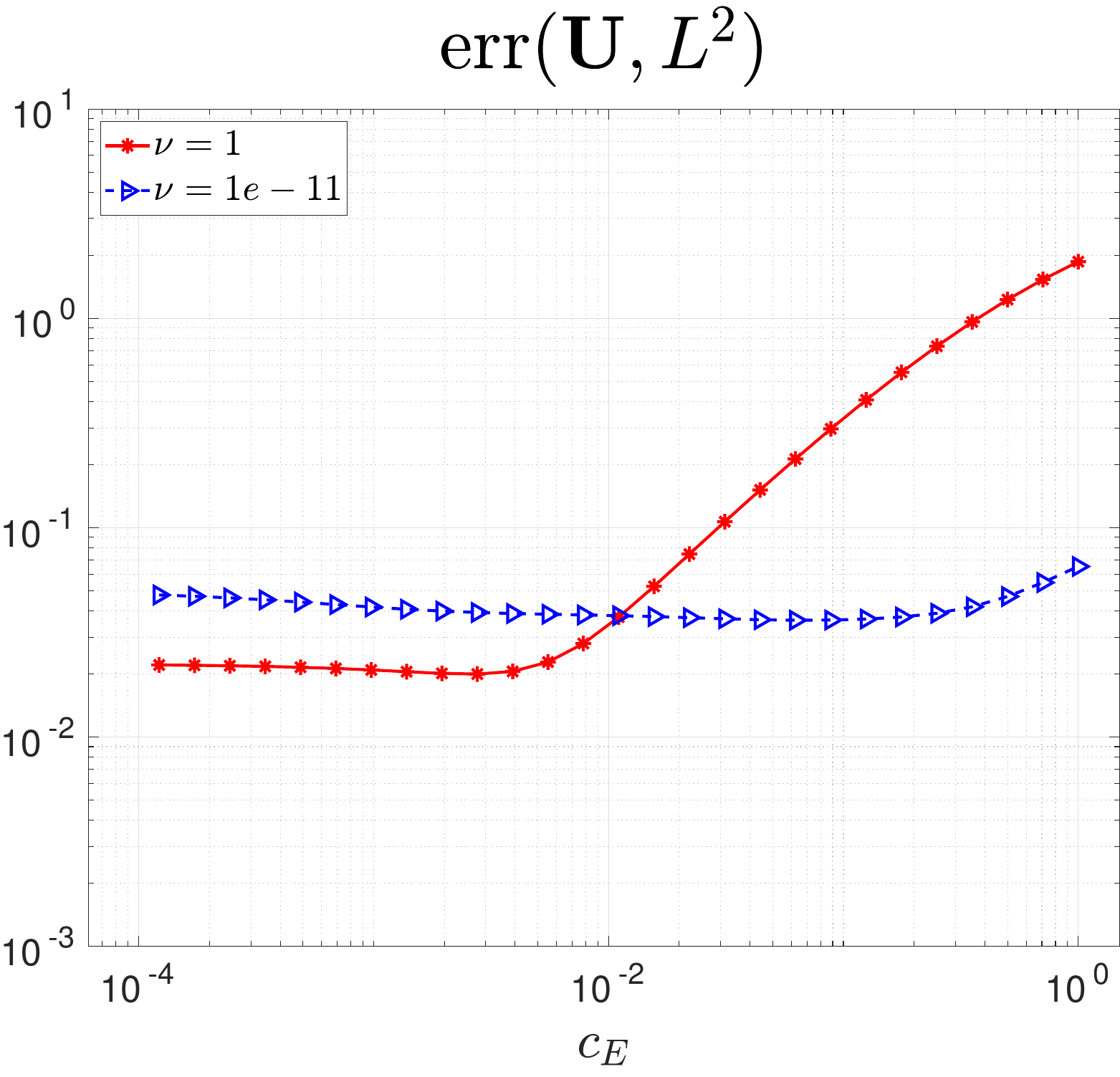}&
\includegraphics[width=0.46\textwidth]{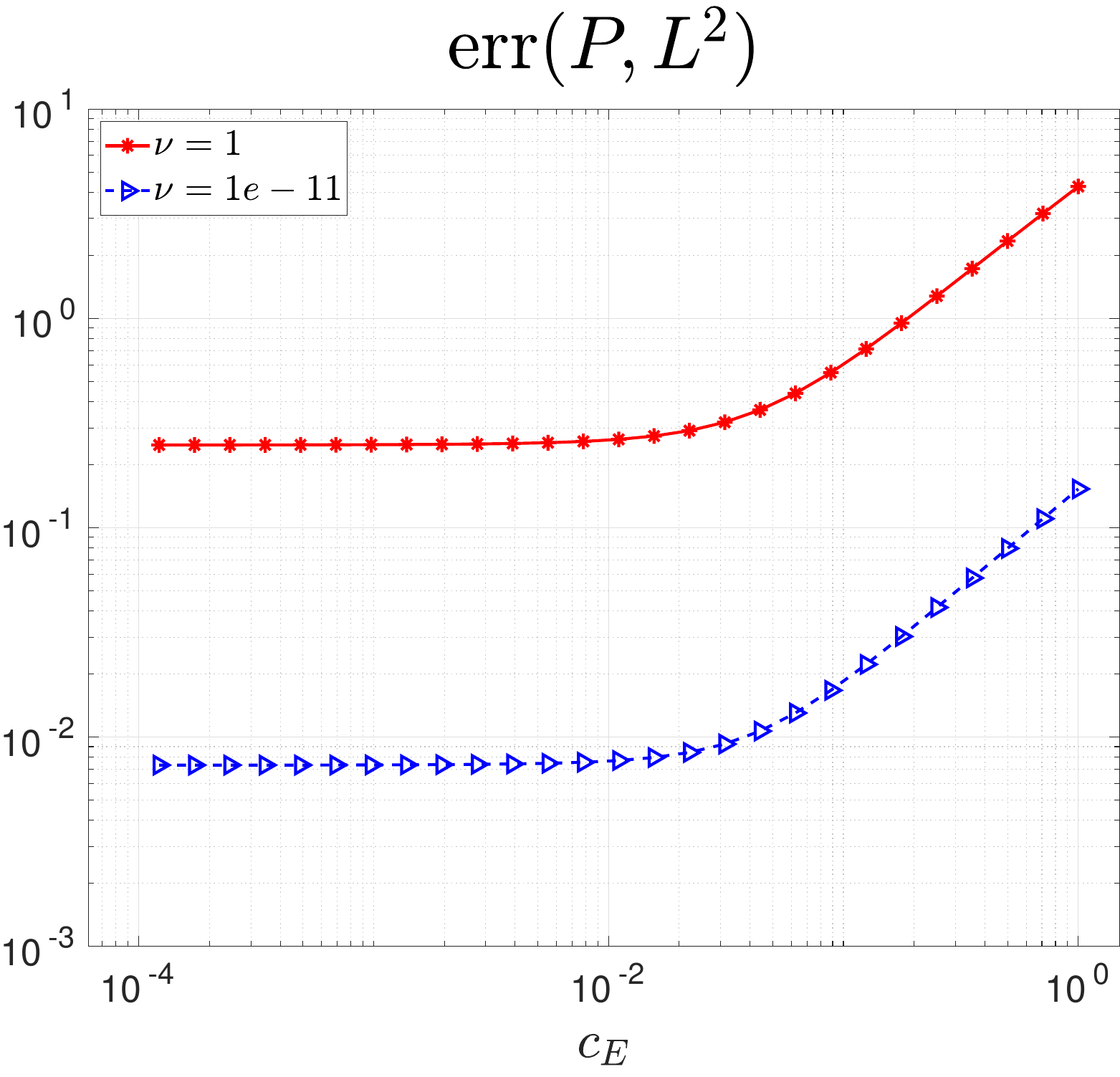}
\end{tabular}
\caption{Convergence analysis: sensitivity analysis of the parameter $c_E$ on $\text{err}(\uht, L^2)$ and $\text{err}(\pht, L^2)$ for $\nu = \texttt{1e-11}$ and $\nu = \texttt{1}$ considering \texttt{mesh 2}.}
\label{fig:CAnalysis}
\end{figure}

\subsection{Pressure robustness analysis}
\label{sub:numDivFree}

In the present test we verify that differently from most stabilizations, the proposed stabilized method does not spoil the benefits related to the divergence-free property.
To assess this feature, we consider a family of Navier-Stokes equations \eqref{eq:ns primale}, one per each choice of the viscosity $\nu$, where the load $\ff$ (which turns out to depend on $\nu$), the initial datum $\uu_0$ and the Dirichlet boundary conditions are chosen in accordance with the analytical solution
\[
\begin{aligned}
\uu_{\rm ex}(x,\,y, \, t) &=\begin{bmatrix}
            y^2\\
            x^2          
            \end{bmatrix} (t+1) \,,
\\
p_{\rm ex}(x,\,y, \, t) & =  \bigg(x^2y+y^3-\frac{5}{12}\bigg)\bigg(t^3-t+1\bigg)\,.
\end{aligned}
\]
Notice that, in the setting of the Scott-Vogelius element, the velocity solution $\uu_{\rm ex} \in \VDGt$ (whereas 
$p_{\rm ex} \not\in \QDGt$).
Therefore, in accordance with \eqref{eq:bound-sv2}, 
the stabilized method produces for this example
the exact velocity solution (up to machine precision)  at all regimes.
In Table~\ref{tab:errExe2} we exhibit
the errors $\text{err}(\uht, L^2)$ for a sequence of decreasing values of $\nu$ (from \texttt{1} to \texttt{1e-11}) on the \texttt{mesh 2} with the associated time step $\tau$ (cf. Fig. \ref{fig:mesh-table} (right)).

\begin{table}[!htb]
\centering
\begin{tabular}{|c|c|c|c|c|c|c|}
\hline
\multicolumn{1}{|c|}{$\nu$}  &
\texttt{1e-00} &\texttt{1e-03} &\texttt{1e-05} &\texttt{1e-07} &
\texttt{1e-09} &\texttt{1e-11} \\
\hline
\texttt{no\,stab} &
\texttt{4.28e-15} &\texttt{1.72e-13} &\texttt{1.24e-13} &
\texttt{1.27e-13} &\texttt{1.28e-13} &\texttt{1.27e-13} \\
\texttt{stab}    &
\texttt{3.21e-15} &\texttt{1.02e-14} &\texttt{1.26e-14} &\texttt{1.23e-14} &\texttt{1.14e-14} &\texttt{1.08e-14} \\
\hline
\end{tabular}
\caption{Pressure robustness analysis: $\text{err}(\uht, L^2)$ as a function of the viscosity $\nu$ on \texttt{mesh 2}.}
\label{tab:errExe2}
\end{table}
The numerical test confirms the theoretical predictions since the velocity error attains machine precision and is not polluted by the pressure: the proposed stabilization does not spoil the pressure robustness of the original scheme.
\subsection{Dynamics of planar lattice flow}
\label{sub:applicative}
In the present test we consider the lattice vortex problem from \cite[Section 4.1]{Olshanskii}. 
We solve the Navier-Stokes equation~\eqref{eq:ns primale} 
on the square domain $\Omega = [0, 1]^2$ and final time $T:=10$ sec.,
where the load $\ff$ (which turns out to be $\boldsymbol{0}$), the initial datum $\uu_0$ and the Dirichlet boundary conditions are chosen in accordance with the exact solution
\[
\begin{aligned}
\uu_{\rm ex}(x,\,y, \, t) &=\begin{bmatrix}
            \sin(2\pi\,x)\sin(2\pi\,y)\\
            \cos(2\pi\,x)\cos(2\pi\,y)
            \end{bmatrix} e^{-8\nu\pi^2t}\,,
\\
p_{\rm ex}(x,\,y, \, t) & =  \frac{1}{2}\bigg(\sin^2(2\pi x)+\cos^2(2\pi y)\bigg)e^{-16\nu\pi^2t}\,.
\end{aligned}
\]
The viscosity is $\nu=\texttt{1e-05}$.
Notice that, due to the small viscosity, the solution is essentially constant in time. 
The velocity $\uu_{\rm ex}$ is characterized by several spinning vortices whose edges touch and the velocity magnitude $\vert \uu_{\rm ex} \vert$ is slightly decreasing in time and is less or equal than $1$. 

The proposed test has two important features: i) the analytical solution is available, this makes immediate the computations of the errors; ii) the vortices of the solution may lead to instabilities in the numerical simulations. 

The test is carried out employing the second order Scott-Vogelius element on the unstructured barycenter-refined mesh with diameter $h=\texttt{8.2e-02}$
depicted in Fig.~\ref{fig:meshAndData} (left).
Note that the adopted mesh is not aligned with the separatrices
of the flow vortices.
The time step is $\tau=\texttt{1.0e-2}$. 
In the setting here considered, the mesh size and the time step are much larger than the choices adopted in \cite[Section 4.1]{Olshanskii} (that are $h=\texttt{1/64}$ and $\tau=\texttt{1.0e-3}$).
\begin{figure}[!htb]
\begin{subfigure}[t]{0.5\textwidth}
\hspace{-0cm}
\includegraphics[width=0.9\textwidth]{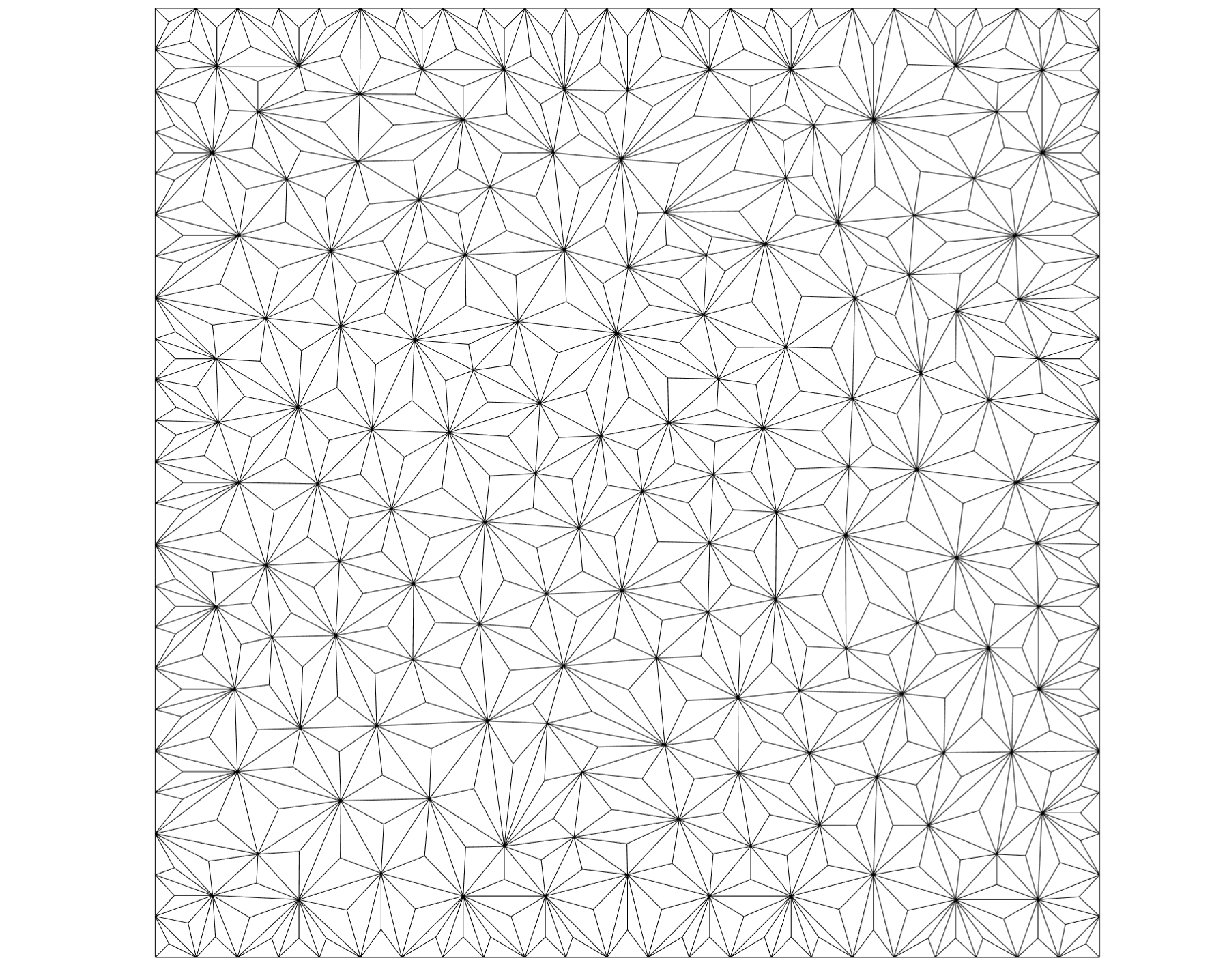}
\end{subfigure}
\hfill
\begin{subfigure}[t]{0.5\textwidth}
\vspace{-3.5cm}
\hspace{-0cm}
\begin{tabular}{|c|c|c|}
\cline{2-3}
\multicolumn{1}{c|}{}&\texttt{no\,stab} &\texttt{stab} \\
\hline
size      &\num{15746}      &\num{15746}\\
\#nnz     &\num{656056}     &\num{861720} \\
sparsity   &0.26\%     &0.35\% \\
\hline
\end{tabular}
\end{subfigure}
\caption{Lattice vortex problem: mesh adopted (left) and some data on the linear system matrix (right).}
\label{fig:meshAndData}
\end{figure}

We assess the practical performance of the proposed stabilized scheme \eqref{eq:ns stab} compared with that of the non-stabilized scheme \eqref{eq:ns fem} for this convection dominated problem.
In Fig.~\ref{fig:VelFiled} we exhibit the snapshots of the solutions $\uht^{\texttt{stab}}$ and $\uht^{\texttt{no\,stab}}$ computed with the \texttt{stab} and the \texttt{no\,stab} schemes respectively at times $t=2.5$, $5$ and $10$ seconds.
We show the magnitude of the velocity solutions (first and third rows of Fig.~\ref{fig:VelFiled}) and the vector plot of the flows  (second and fourth rows of Fig.~\ref{fig:VelFiled}).
It can be observed that, as time proceeds, the $\uht^{\texttt{no\,stab}}$ solution completely loses the correct vortex shape, moreover at the final time $t=10$ sec. the velocity magnitude is even greater than 4.5 (the threshold value is 1).
Whereas the \texttt{stab} method is not effected by instabilities: in $\uht^{\texttt{stab}}$ solution,
the spinning  vortices preserve their shapes as time advances. The velocity magnitude slowly decreases in time in accordance with the exact solution and is lower than 1.

In Fig.~\ref{fig:ErrorOverTime} we display 
the relative errors ($H^1$-seminorm and $L^2$-norm)
obtained with the \texttt{stab} method and the \texttt{no\,stab} method versus time.
It is easily appreciated that the \texttt{stab} method is clearly superior in terms of solution accuracy.
\begin{figure}[!htb]
\begin{center}
\centering
\begin{tabular}{cc}
\includegraphics[width=0.46\textwidth]{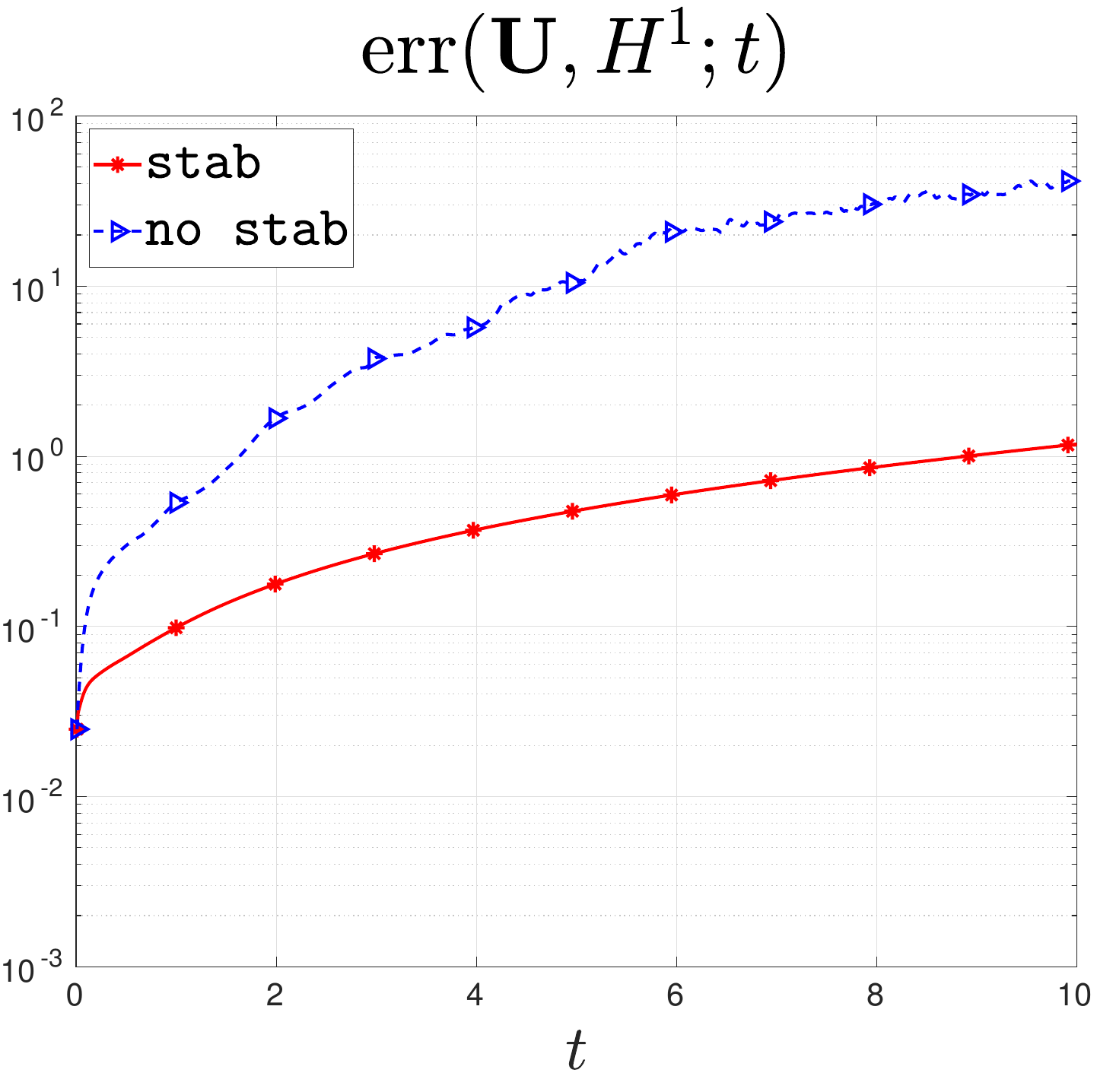}&
\includegraphics[width=0.46\textwidth]{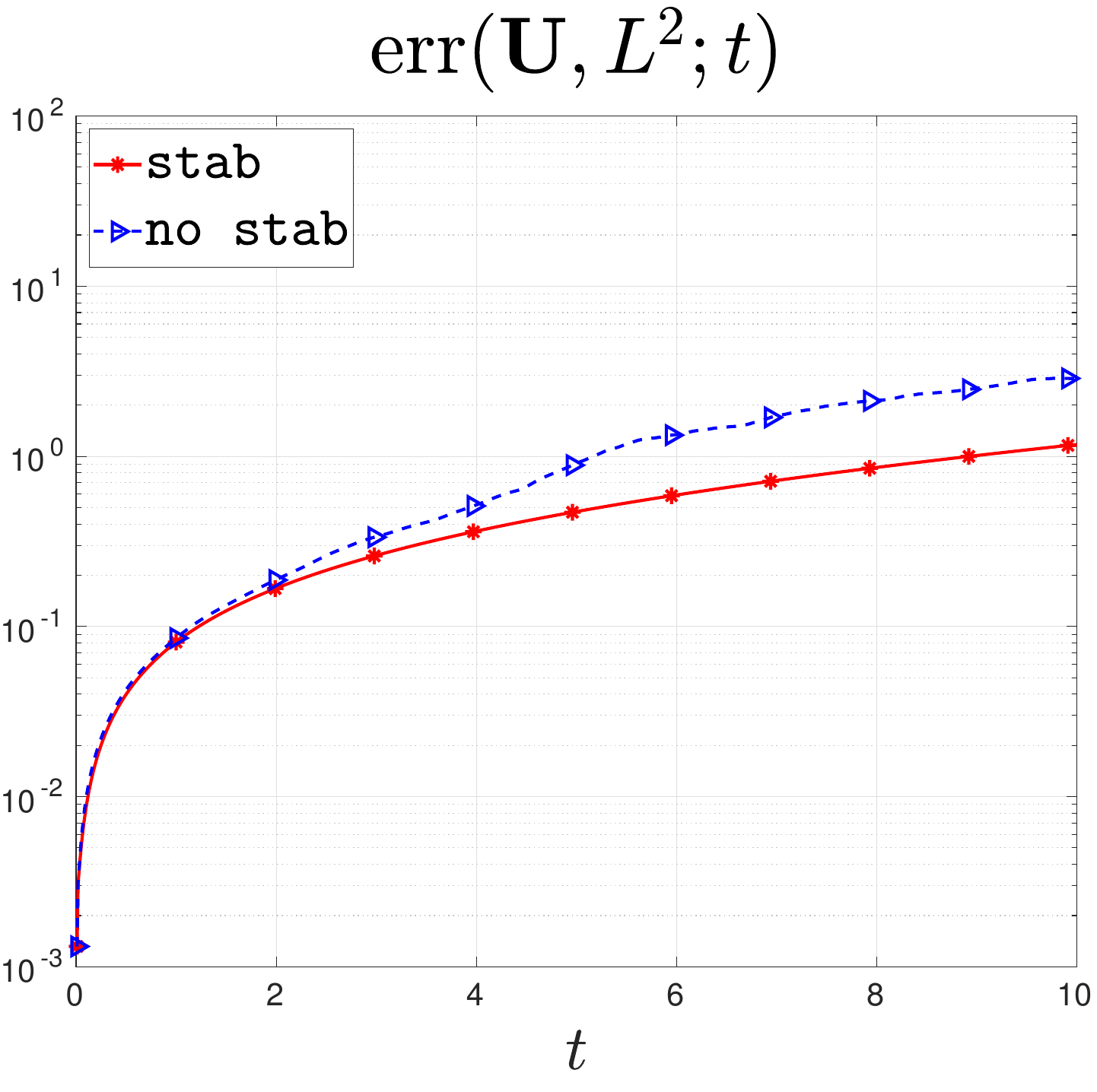}
\end{tabular}
\end{center}
\caption{Lattice vortex problem: $H^1$-seminorm (left) $L^2$-norm errors (right) on the velocity field.
}
\label{fig:ErrorOverTime}
\end{figure}
We finally address a simple comparison of the \texttt{stab} and the \texttt{no\,stab} methods in terms of computational cost.
In order to avoid comparisons which may depend on the particular implemented computer code, the amount of parallelization in the matrix assembly and on the linear system resolution procedure, 
we here limit ourselves in investigating the sparsity pattern of the global matrix associated to the linear system of the first fixed point iteration (at the first time step). 
In the table in Fig.~\ref{fig:meshAndData} (right) we report the number of non-zero \#nnz entries (and the percentage of such entries over total entries) for both the \texttt{stab} and the \texttt{no\,stab} schemes. 
In Fig.~\ref{fig:matrix} we depict the different sparsity patterns for the two methods.  
It is clear that the \texttt{stab} scheme has a denser pattern mainly due to the added jump terms.
Depending on the type of sparse solver adopted, this could lead in principle to longer computer times.  
\begin{figure}[!htb]
\begin{center}
\begin{subfigure}[t]{0.45\textwidth}
\includegraphics[width=0.8\textwidth]{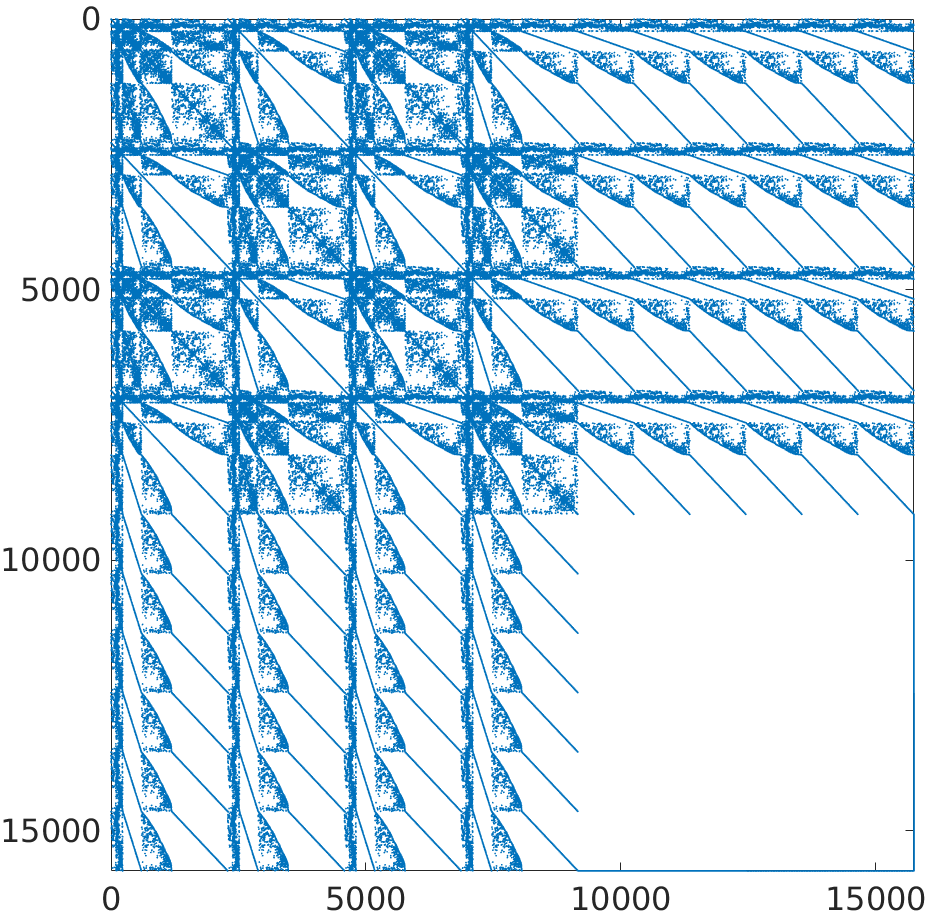}
\end{subfigure}
\begin{subfigure}[t]{0.45\textwidth}
\includegraphics[width=0.8\textwidth]{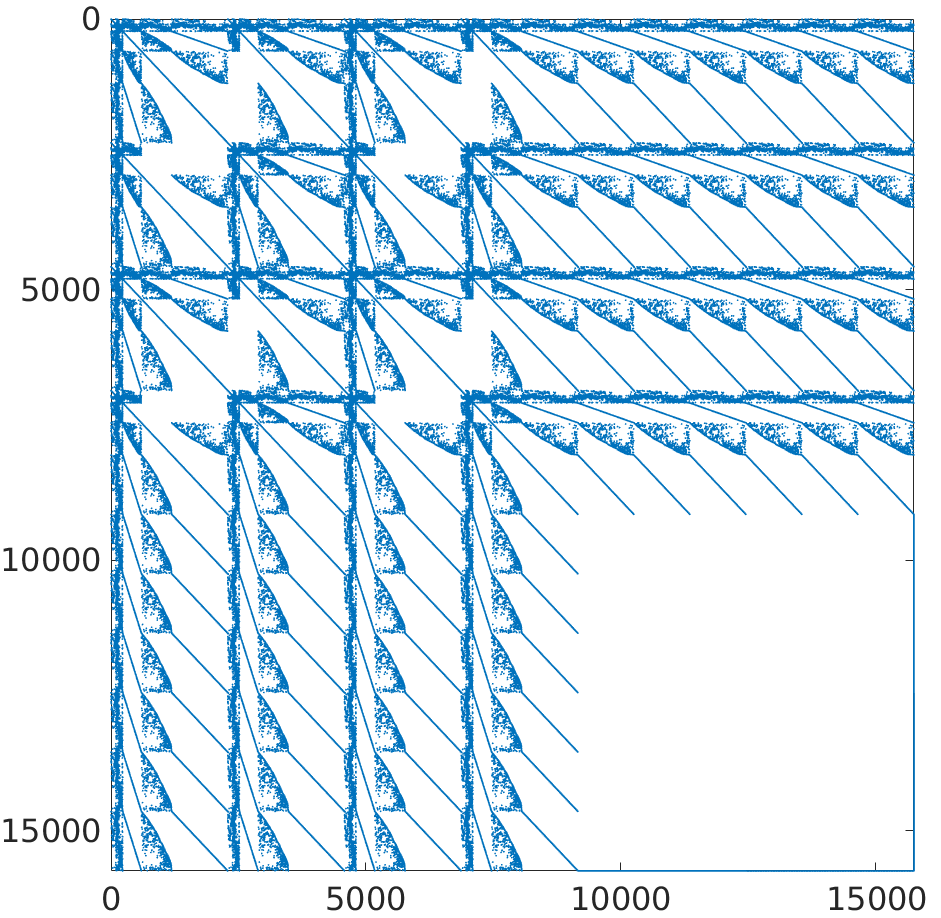}
\end{subfigure}
\end{center}
\caption{Lattice vortex problem: sparsity pattern of the matrix arising from the \texttt{stab} method and the \texttt{no\,stab} method.}
\label{fig:matrix}
\end{figure}
\begin{figure}[!htb]
\centering
\begin{tabular}{cccc}
\hspace{1em}$t=2.50$ sec. &$t=5.00$ sec. &$t=10.0$ sec.\\
\multirow{2}{*}{\rotatebox[origin=c]{90}{\texttt{no\,stab}}}
\hspace{1em}\includegraphics[height=0.30\textwidth]{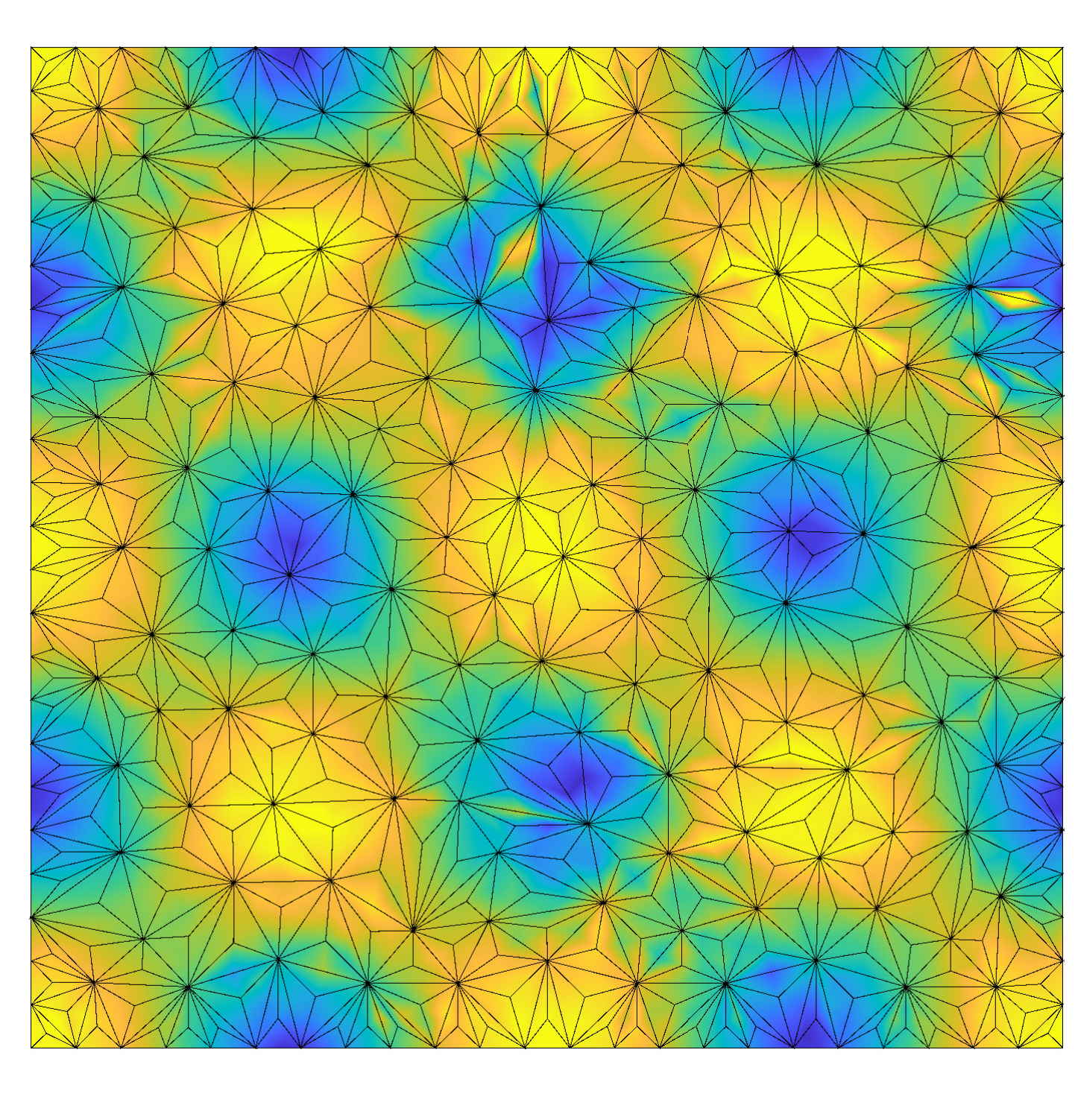} &
\includegraphics[height=0.30\textwidth]{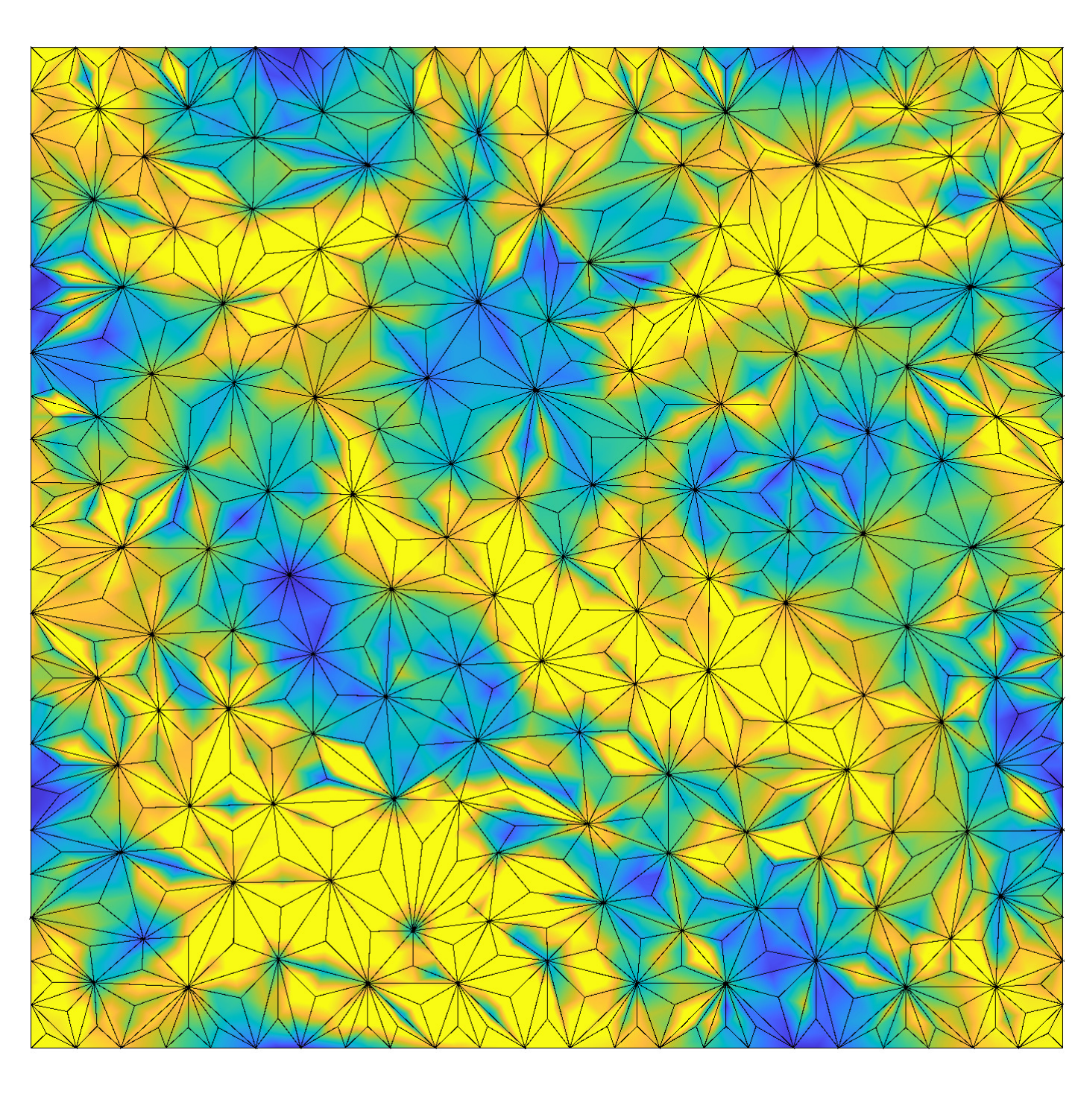} &
\includegraphics[height=0.30\textwidth]{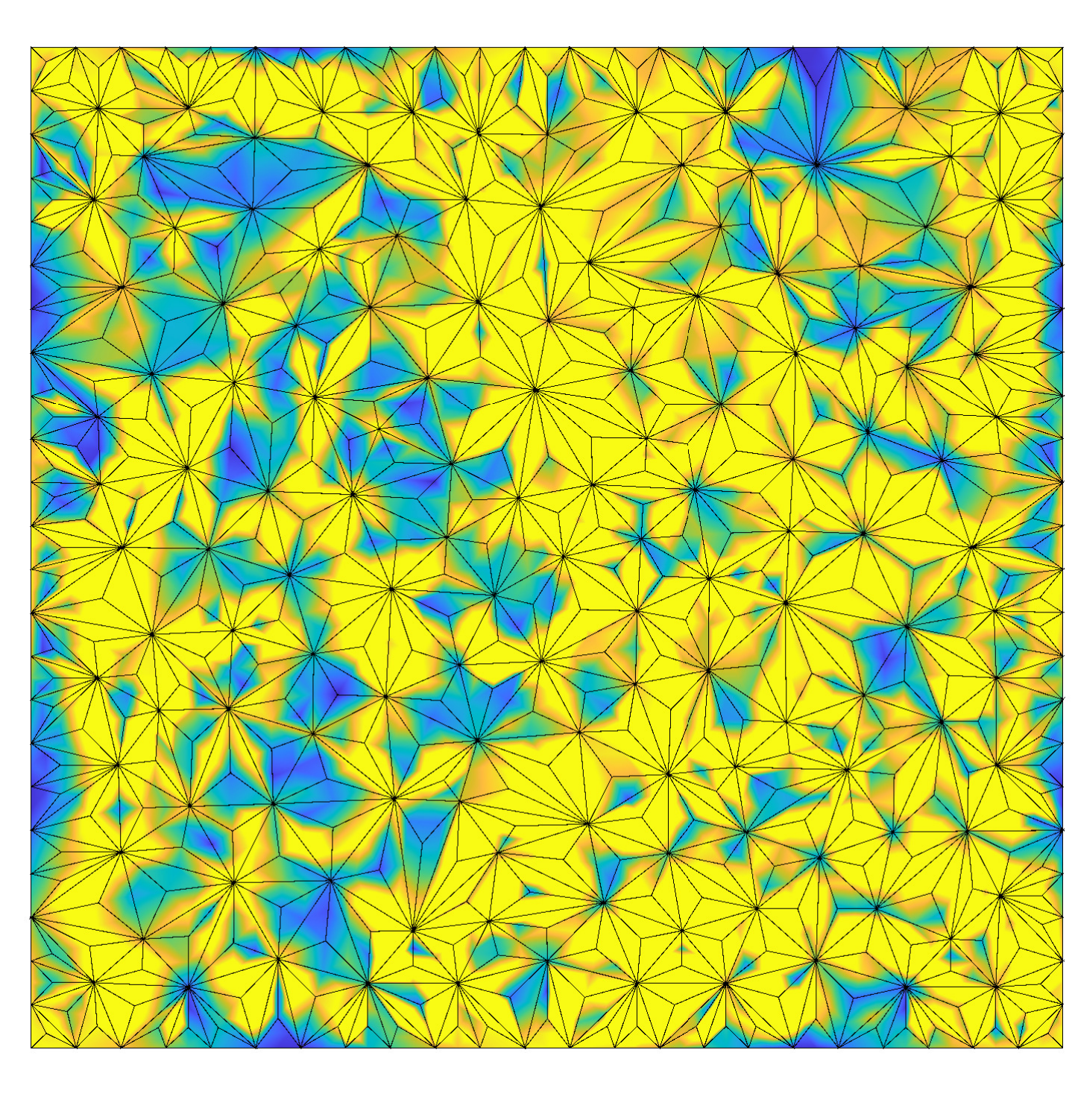} \\ 
\hspace{1em}\includegraphics[height=0.30\textwidth]{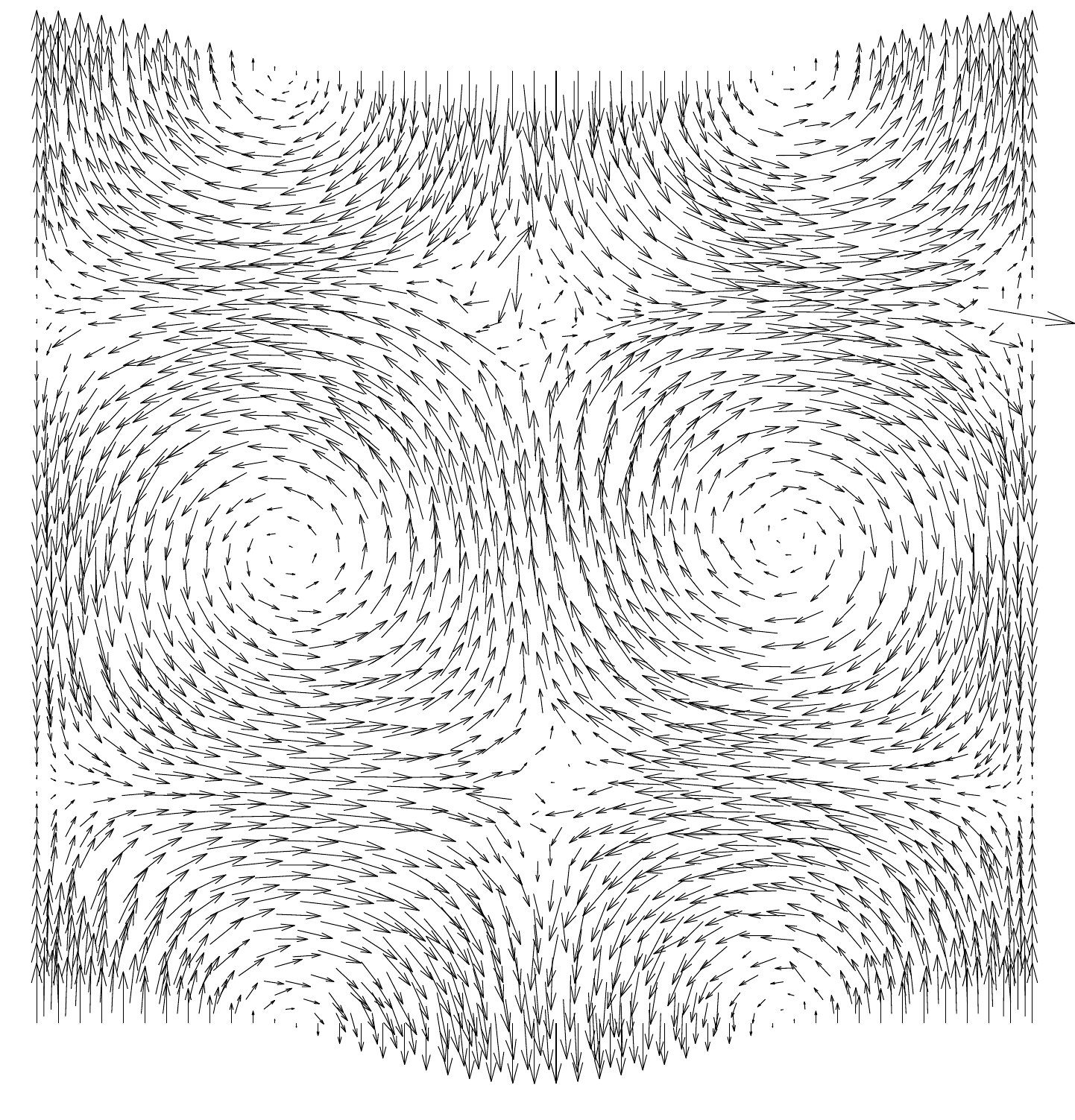} &
\includegraphics[height=0.30\textwidth]{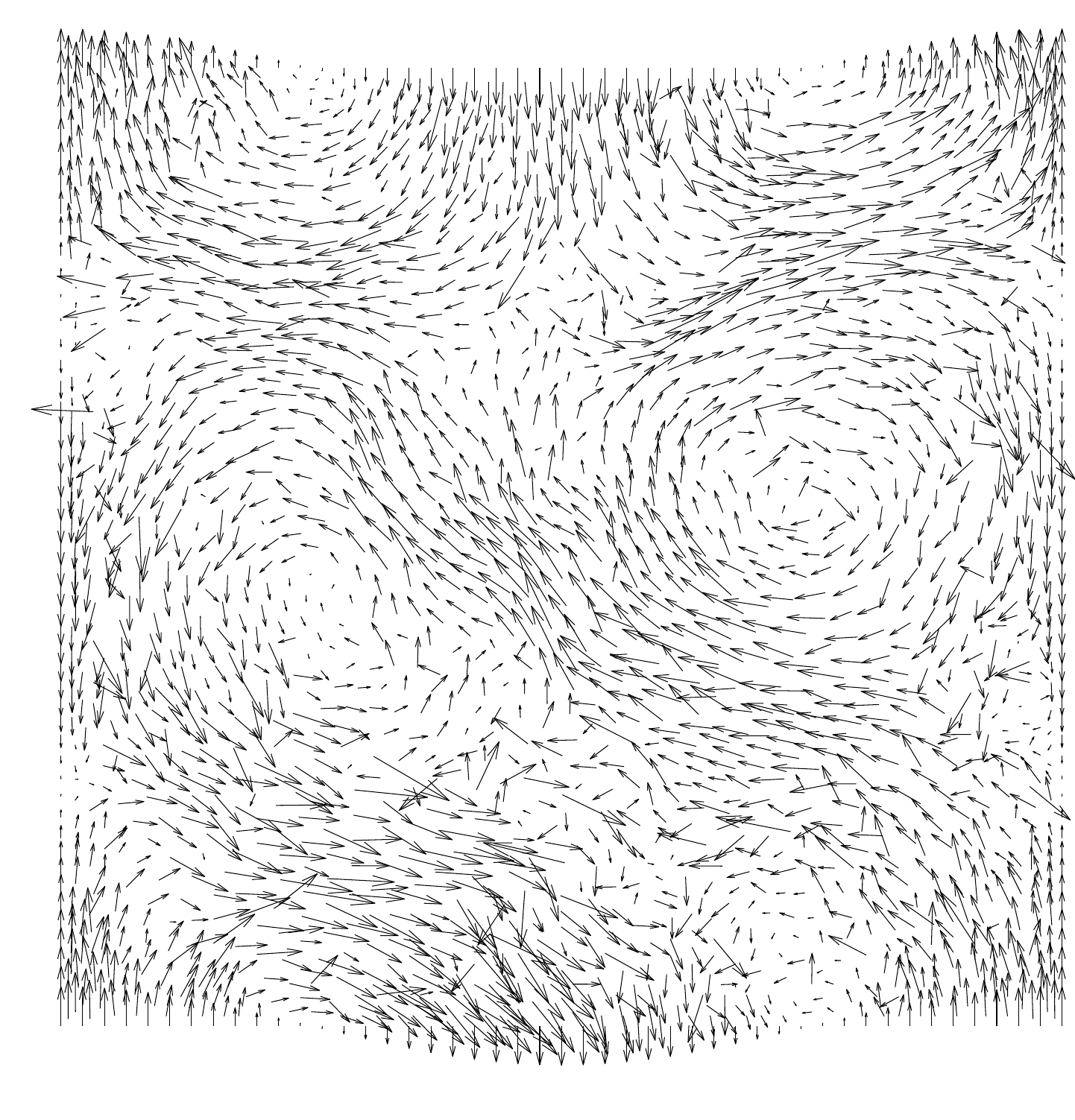} &
\includegraphics[height=0.30\textwidth]{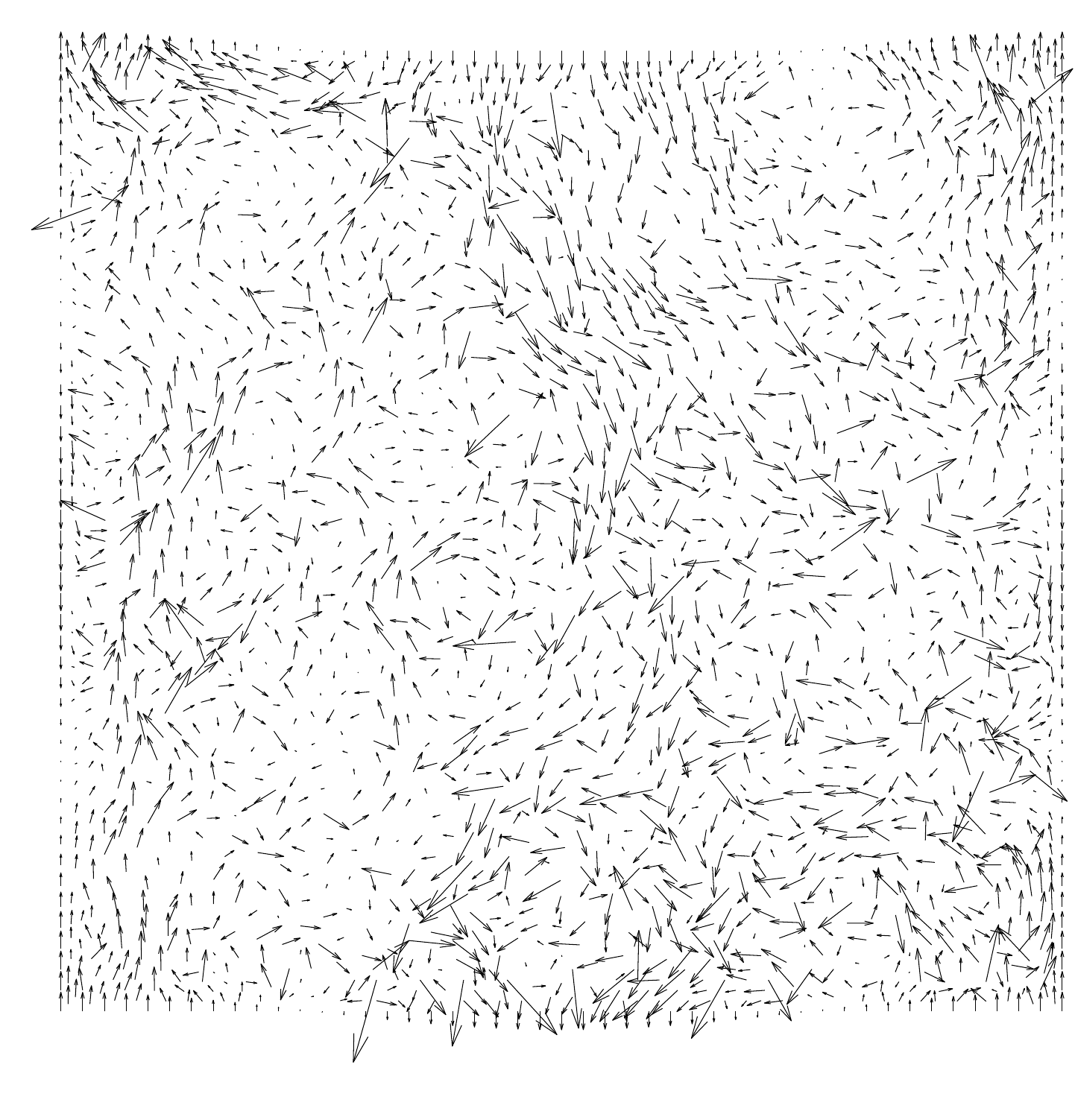} \\
\multirow{2}{*}{\rotatebox[origin=c]{90}{\texttt{stab}}}
\hspace{1em}\includegraphics[height=0.30\textwidth]{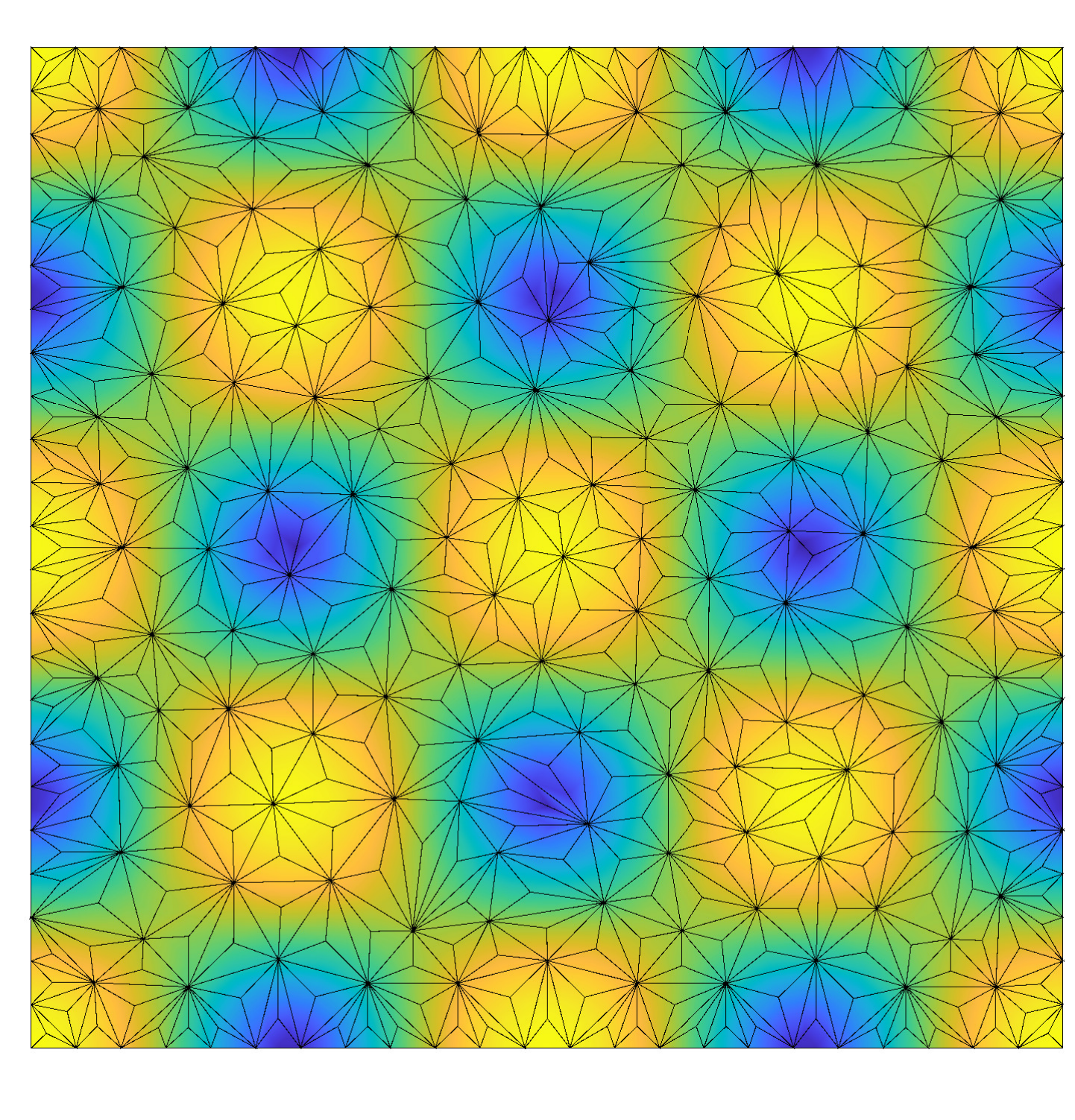} &
\includegraphics[height=0.30\textwidth]{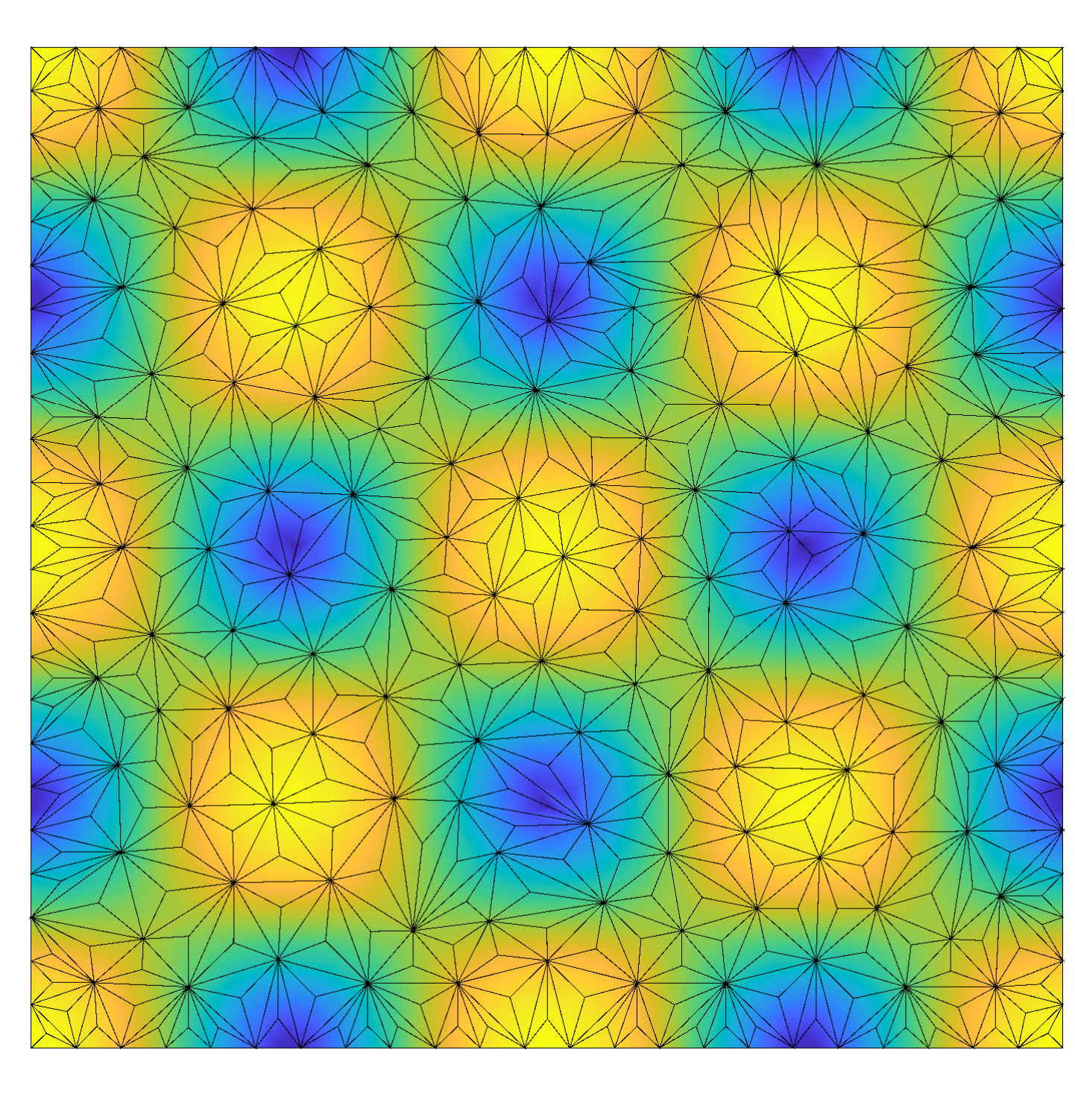} &
\includegraphics[height=0.30\textwidth]{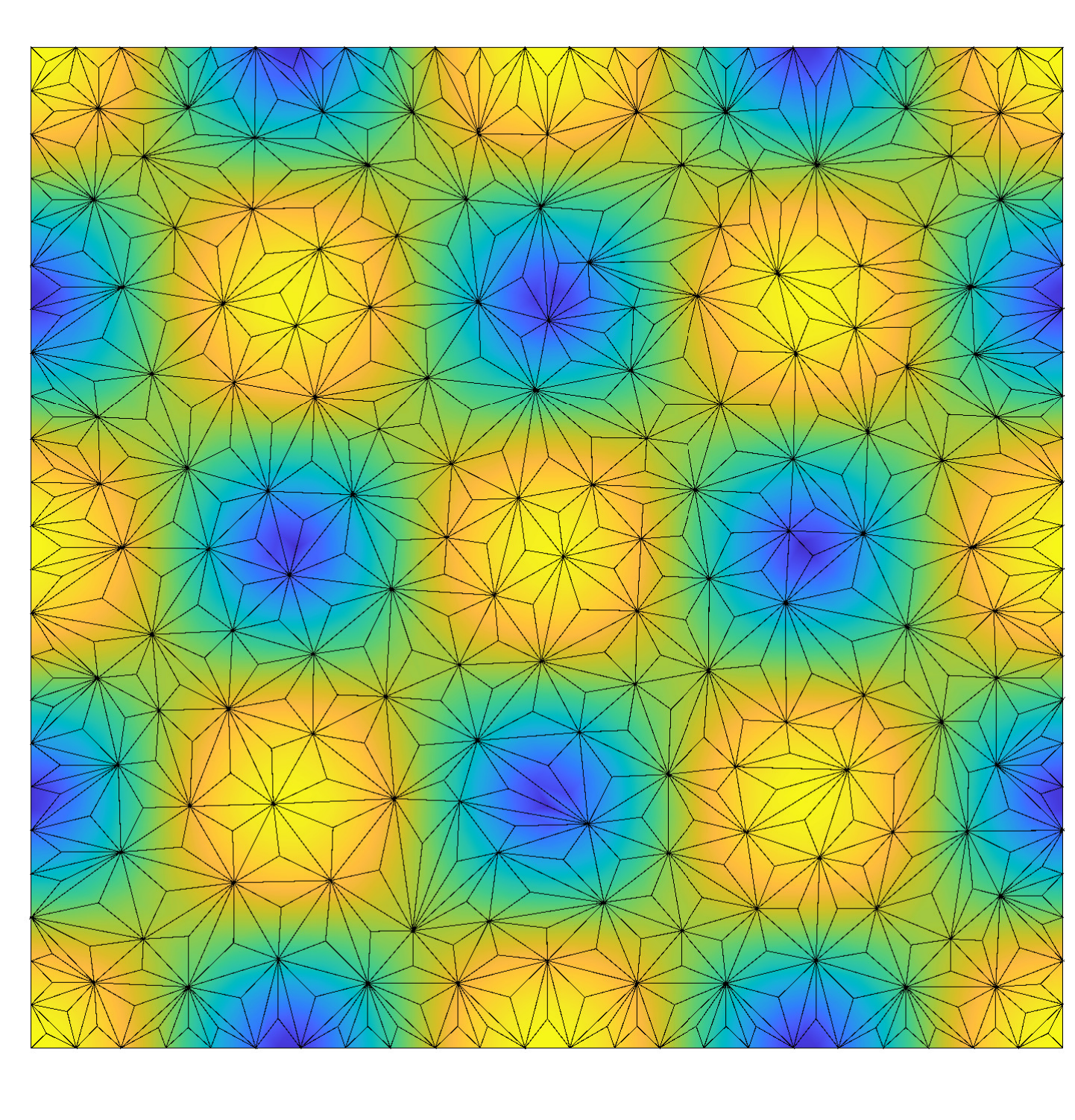} \\ 
\hspace{1em}\includegraphics[height=0.30\textwidth]{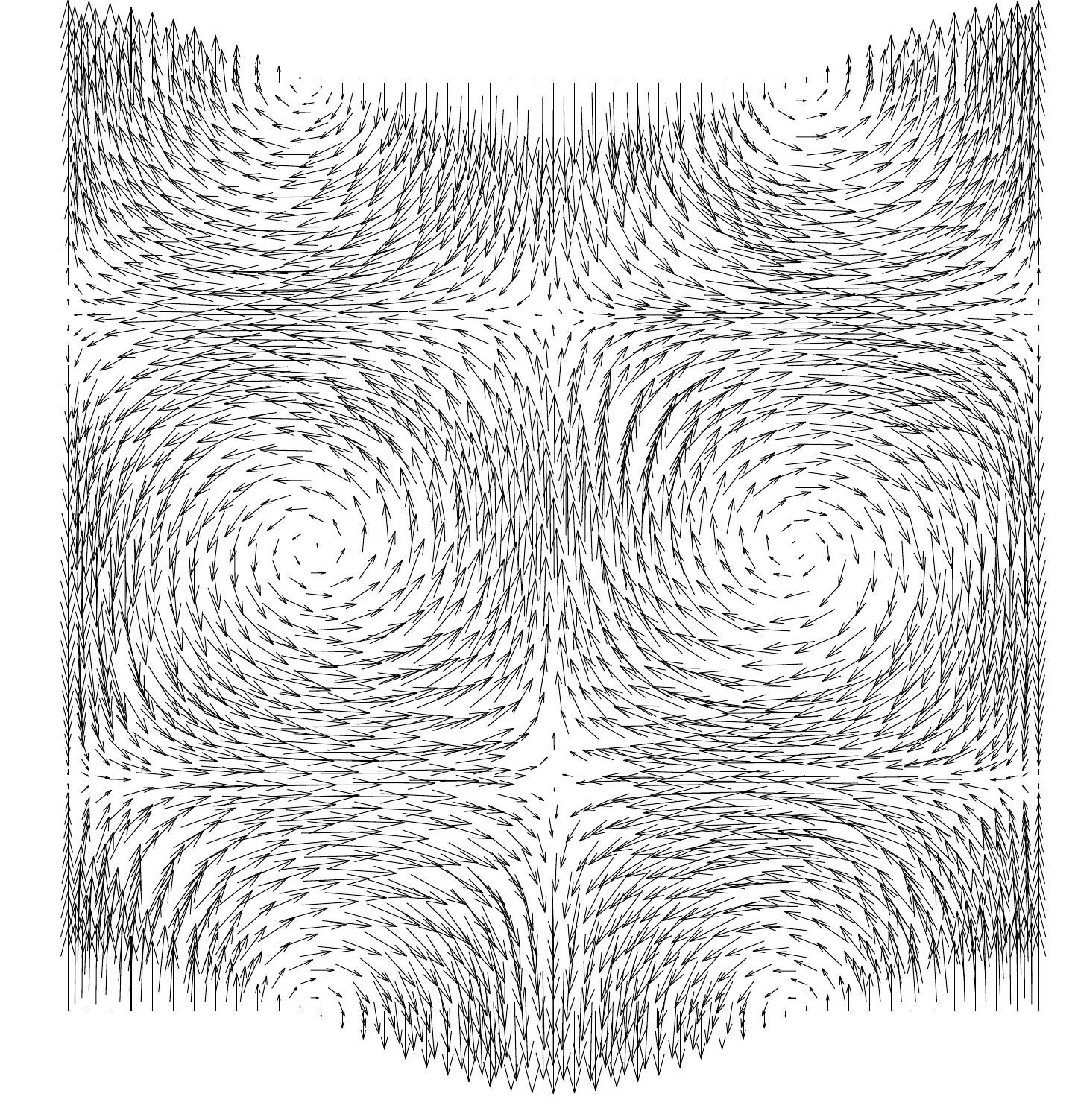} &
\includegraphics[height=0.30\textwidth]{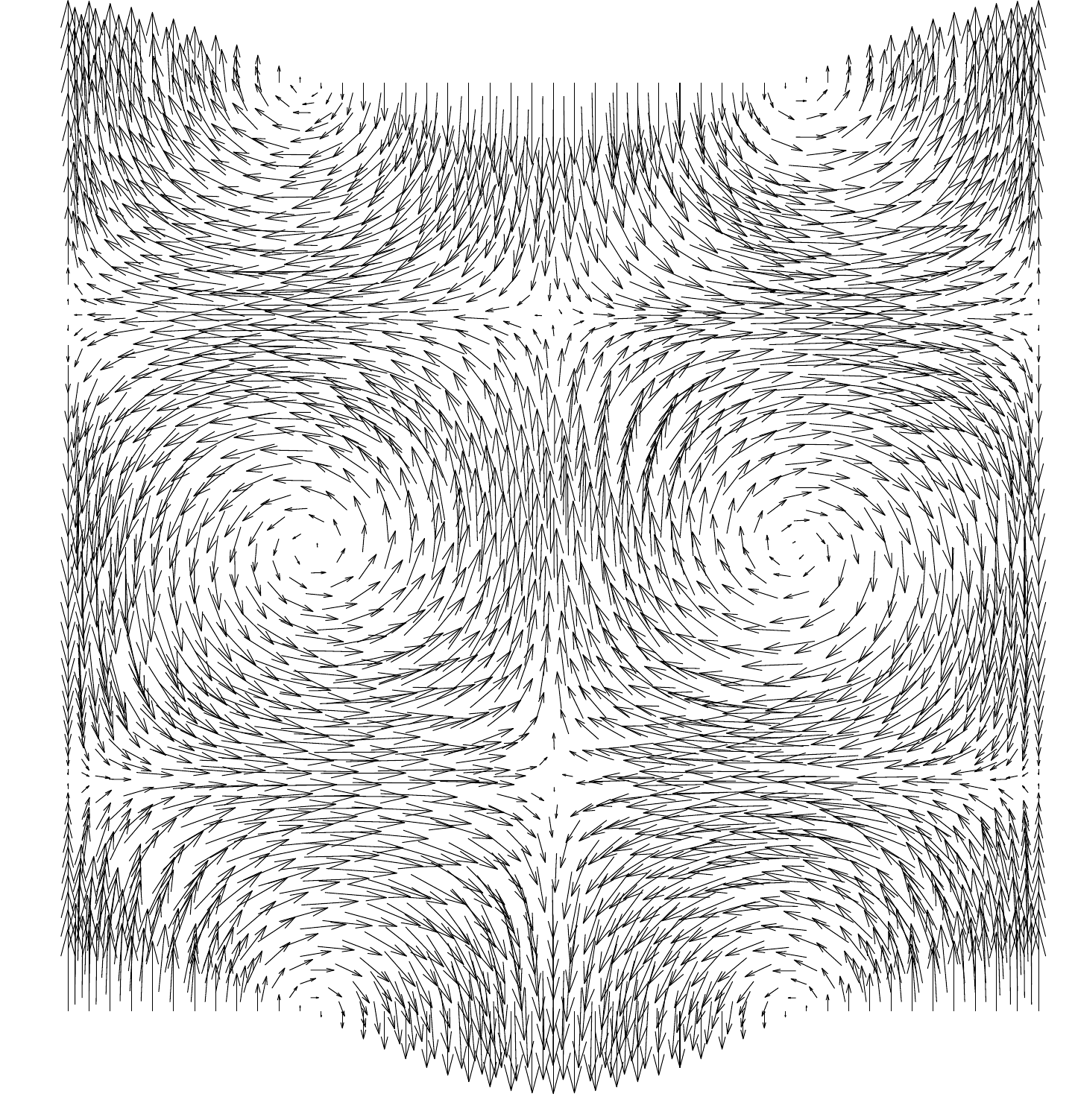} &
\includegraphics[height=0.30\textwidth]{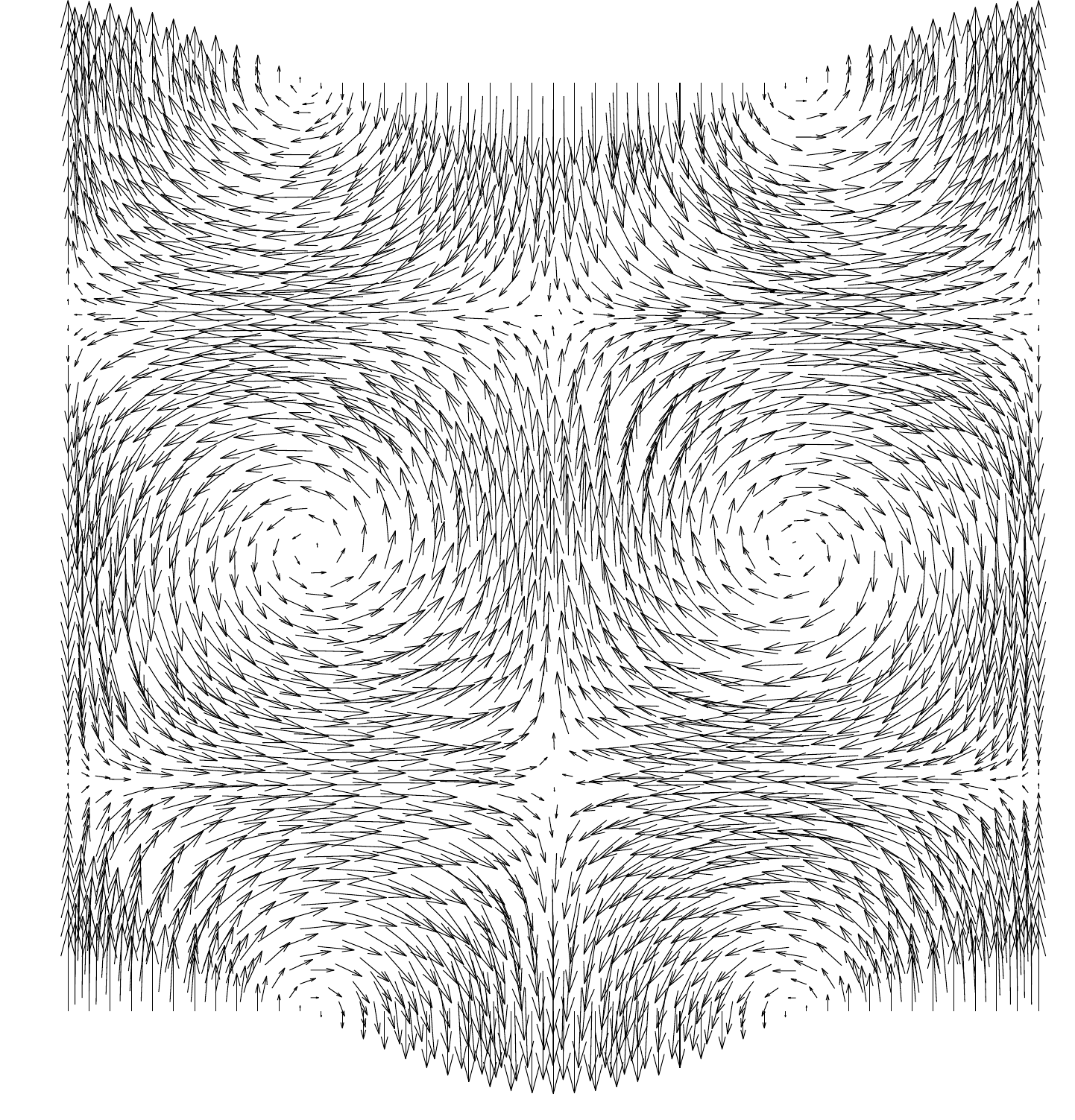} \\
\multicolumn{3}{c}{\begin{overpic}[abs,unit=1mm,width=0.50\textwidth]{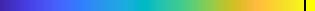}
                    \put(-1,-3){0.}
                    \put(69,-3){1.}
                    \put(73,0.5){$>1.$}
                   \end{overpic}}\\
\multicolumn{3}{c}{$|\uu_h|$}
\end{tabular}
\caption{Lattice vortex problem: comparison between the \texttt{stab} and the \texttt{no\,stab} scheme at different times.
Magnitude of the velocity fields (first and third rows) and  vector plot of the flows. (second and fourth rows).}
\label{fig:VelFiled}
\end{figure}

\section*{Funding}
LBdV was partially supported by the Italian MIUR through the PRIN grant n. 608 201744KLJL. 
This support is gratefully acknowledged.

\clearpage 

\addcontentsline{toc}{section}{\refname}
\bibliographystyle{plain}
\bibliography{references}
\end{document}